\documentclass[12pt]{amsart}
\usepackage{amsmath,amsthm,amsfonts,amssymb,amscd,euscript}
\usepackage{color}
\usepackage{graphics}
\usepackage{tikz}
\input{xy}
\xyoption{all}
\newlength{\defbaselineskip}
\theoremstyle{plain}
\newtheorem{theorem}{Theorem}[section]

\newtheorem{proposition}[theorem]{Proposition}
\newtheorem{corollary}[theorem]{Corollary}
\newtheorem{lemma}[theorem]{Lemma}

\theoremstyle{definition}
\newtheorem{definition}[theorem]{Definition}

\newtheorem{remark}[theorem]{Remark}
\newtheorem{example}[theorem]{Example}

\newtheorem{conjecture}[theorem]{Conjecture}

\newcommand{\ZZ}{\mathbb{Z}}

\newcommand{\QQ}{\mathbb{Q}}

\newcommand{\gchoose}[2]{\left[\begin{array}{c}{#1 } \\{#2} \end{array}\right]  }
\newcommand{\gchooseround}[2]{{{#1 } \choose {#2}}}

\newcommand{\boxs}[1]
{ \multiput(#1)(20,0){2}
 {\line(0,20){20}}
\multiput(#1)(0,20){2}
 {\line(20,0){20}}
}

\newcommand{\exoneprime}{x_{1;t_2'}}

\begin{document}

\title[Positivity for rank 3 cluster algebras]{Positivity for cluster algebras of rank 3}
\author{Kyungyong Lee}\thanks{Communicated by H. Nakajima. Received August 28, 2012.}
\thanks{Revised December 11, 2012, January 29, 2013, February 16, 2013.}
\thanks{ 2010 {\it Mathematics Subject Classificaiton.} 13F60}
\thanks{ {\it Key works and phrases.} cluster algebra, positivity conjecture.}
\thanks{{The first author is  partially
   supported by the NSF grant DMS-0901367.}
   \thanks{The second author is partially
   supported by the NSF grant DMS-1001637.}}
   \thanks{Kyungyong Lee, Department of Mathematics,
Wayne State University,
656 West Kirby,
Detroit, MI 48202, USA, {klee@math.wayne.edu; }}
\author{Ralf Schiffler}
\thanks{Ralf Schiffler, Department of Mathematics,
196 Auditorium Road,
University of Connecticut, U-3009
Storrs, CT 06269-3009,
USA, schiffler@math.uconn.edu}

\dedicatory{Dedicated to Robert Lazarsfeld on the occasion of his sixtieth birthday}

\maketitle
\begin{abstract}  We prove the positivity conjecture for  skew-symmetric coefficient-free cluster algebras of rank 3.
\end{abstract}

\section{Introduction}

Cluster algebras have been introduced by Fomin and Zelevinsky in \cite{FZ} in the context of total positivity and canonical bases in Lie theory. Since then cluster algebras have been shown to be related to various fields in mathematics including representation theory of finite dimensional algebras, Teichm\"uller theory, Poisson geometry, combinatorics, Lie theory, tropical geometry and mathematical physics.

A cluster algebra is a subalgebra of a field of rational functions in $n$ variables $x_1,x_2,\ldots,x_n$, given by specifying a set of generators, the so-called \emph{cluster variables}. These generators are constructed in a recursive way, starting from the initial variables $x_1,x_2,\ldots,x_n$, by a procedure called \emph{mutation}, which is determined by the choice of a skew symmetric $n\times n$ integer matrix $B$ or, equivalently, by a quiver $Q$.
Although each mutation is an elementary operation, it is very difficult to compute cluster variables in general, because of the recursive character of the construction.

Finding explicit computable direct formulas for the cluster variables is one of the main open problems in the theory of cluster algebras and has been studied by many mathematicians. In 2002, Fomin and Zelevinsky showed that every cluster variable is a Laurent polynomial in the initial variables $x_1,x_2,\ldots, x_n$, and they conjectured that this Laurent polynomial has positive coefficients \cite{FZ}.

 This \emph{positivity conjecture} has been proved in the following special cases
 \begin{itemize}
 \item \emph{Acyclic cluster algebras.} These are cluster algebras given by a quiver that is mutation equivalent to a quiver without oriented cycles. In this case, positivity has been shown in \cite{KQ} building on \cite{BZ,HL,N,Q} using monoidal categorifications of quantum cluster algebras and  perverse sheaves over graded quiver varieties. If the initial seed itself is acyclic, the conjecture has also been shown in \cite{Ef} using Donaldson-Thomas theory.
\item \emph{Cluster algebras from surfaces.}  In this case, positivity has been shown in \cite{MSW} building on \cite{S2,ST,S3}, using the fact that each cluster variable in such a cluster algebra corresponds to a curve in an oriented Riemann surface and the Laurent expansion of the cluster variable is determined by the crossing pattern of the curve with a fixed triangulation of the surface \cite{FG,FST}. The construction and the proof of the positivity conjecture has been generalized to non skew-symmetric cluster algebras from orbifolds in \cite{FeShTu}.
\end{itemize}

Our approach in this paper is different. We prove positivity almost exclusively by elementary algebraic computation. The advantage of this approach is that we do not need to restrict to a special type of cluster algebras but can work in the setting of an arbitrary  cluster algebra. The drawback of our approach is that because of the sheer complexity of the computation, we need to restrict ourselves in this paper to the rank three. The rank three is crucial since it is the smallest rank in which non-acyclic cluster algebras exist.
Our main result is the following.

\begin{theorem} The positivity conjecture holds in every skew-symmetric coefficient-free cluster algebra of rank 3.
\end{theorem}

Our argument provides a method for the computation of the Laurent expansions of cluster variables, and we include some examples of explicit calculation.
We point out that
direct formulas for the Laurent polynomials have been obtained in several special cases. The most general results are the following:
\begin{itemize}
\item a formula involving the Euler-Poincar\'e characteristic of quiver Grassmannians obtained in \cite{FK,DWZ} using categorification and generalizing results in \cite{CC,CK1}. While this formula shows a very interesting connection between cluster algebras and geometry, it is of limited computational use, since the Euler-Poincar\'e characteristics of quiver Grassmannians are hard to compute. In particular, this formula does not show positivity. On the other hand, the positivity result in this paper proves the positivity of the Euler-Poincar\'e characteristics of the quiver Grassmannians involved.
\item an elementary combinatorial formula for cluster algebras from surfaces given in \cite{MSW}.
\item a formula for cluster variables corresponding to string modules as a product of $2\times 2$ matrices obtained in \cite{ADSS}, generalizing a result in \cite{ARS}.
\end{itemize}

The main tools of the proof are modified versions of two formulas for the rank two case, one obtained by the first author in \cite{L} and the other obtained by both authors in \cite{LS}. These formulas allow for the computation of  the Laurent expansions of a given cluster variable with respect to any seed which is close enough to the variable in the sense that there is a sequence of mutations in only two vertices which links seed and variable. The general result then follows by inductive reasoning.

 If the cluster algebra is not skew-symmetric, it is shown in \cite{R,LLZ} that (an adaptation of)  the second rank two formula still holds.
 We therefore expect that our argument can be generalized to prove the positivity conjecture for non skew-symmetric cluster algebras of rank 3.

The article is organized as follows. We start by recalling some definitions and results from the theory of cluster algebras in section \ref{sect 2}. In section \ref{sect rank 2}, we present several formulas for the rank 2 case when considered inside a cluster algebra of rank 3. We use each of these formulas in the proof of the positivity conjecture for rank 3 in  section \ref{sect 3}.
An example is given in section \ref{sect example}.

\noindent \emph{Acknowledgements.} We are grateful to  Andrei Zelevinksy and anonymous referees for their valuable suggestions.  Shortly after the original version of this paper \cite{LS3arXiv} became available, a new parametrization of our rank 2 formula has been given in \cite{LLZ}. Following the referee's suggestion, we use this new parametrization here to present our result in this final version.


\section{Cluster algebras}\label{sect 2}
In this section, we review some notions from the theory of cluster algebras.
\subsection{Definition and Laurent phenomenon}\label{sect cluster algebras}
We begin by reviewing the definition of cluster algebra,
first introduced by Fomin and Zelevinsky in \cite{FZ}.
Our definition follows the exposition in \cite{FZ4}.

To define  a cluster algebra~$\mathcal{A}$ we must first fix its
ground ring.
Let $(\mathbb{P},\oplus, \cdot)$ be a \emph{semifield}, i.e.,
an abelian multiplicative group endowed with a binary operation of
\emph{(auxiliary) addition}~$\oplus$ which is commutative, associative, and
distributive with respect to the multiplication in~$\mathbb{P}$.
The group ring~$\ZZ\mathbb{P}$ will be
used as a \emph{ground ring} for~$\mathcal{A}$.

As an \emph{ambient field} for
$\mathcal{A}$, we take a field $\mathcal{F}$
isomorphic to the field of rational functions in $n$ independent
variables (here $n$ is the \emph{rank} of~$\mathcal{A}$),
with coefficients in~$\QQ \mathbb{P}$.
Note that the definition of $\mathcal{F}$ does not involve
the auxiliary addition
in~$\mathbb{P}$.

\begin{definition}
\label{def:seed}
A \emph{labeled seed} in~$\mathcal{F}$ is
a triple $(\mathbf{x}, \mathbf{y}, B)$, where
\begin{itemize}
\item
$\mathbf{x} = (x_1, \dots, x_n)$ is an $n$-tuple
from $\mathcal{F}$
forming a \emph{free generating set} over $\QQ \mathbb{P}$,
\item
$\mathbf{y} = (y_1, \dots, y_n)$ is an $n$-tuple
from $\mathbb{P}$, and
\item
$B = (b_{ij})$ is an $n\!\times\! n$ integer matrix
which is \emph{skew-symmetrizable}.
\end{itemize}
That is, $x_1, \dots, x_n$
are algebraically independent over~$\QQ \mathbb{P}$, and
$\mathcal{F} = \QQ \mathbb{P}(x_1, \dots, x_n)$.
We refer to~$\mathbf{x}$ as the (labeled)
\emph{cluster} of a labeled seed $(\mathbf{x}, \mathbf{y}, B)$,
to the tuple~$\mathbf{y}$ as the \emph{coefficient tuple}, and to the
matrix~$B$ as the \emph{exchange matrix}.
\end{definition}

We  use the notation
$[x]_+ = \max(x,0)$,
$[1,n]=\{1, \dots, n\}$, and
\begin{align*}
\textup{sgn}(x) &=
\begin{cases}
-1 & \text{if $x<0$;}\\
0  & \text{if $x=0$;}\\
 1 & \text{if $x>0$.}
\end{cases}
\end{align*}

\begin{definition}
\label{def:seed-mutation}
Let $(\mathbf{x}, \mathbf{y}, B)$ be a labeled seed in $\mathcal{F}$,
and let $k \in [1,n]$.
The \emph{seed mutation} $\mu_k$ in direction~$k$ transforms
$(\mathbf{x}, \mathbf{y}, B)$ into the labeled seed
$\mu_k(\mathbf{x}, \mathbf{y}, B)=(\mathbf{x}', \mathbf{y}', B')$ defined as follows:
\begin{itemize}
\item
The entries of $B'=(b'_{ij})$ are given by
\begin{equation}
\label{eq:matrix-mutation}
b'_{ij} =
\begin{cases}
-b_{ij} & \text{if $i=k$ or $j=k$;} \\[.05in]
b_{ij} + \textup{sgn}(b_{ik}) \ [b_{ik}b_{kj}]_+
 & \text{otherwise.}
\end{cases}
\end{equation}
\item
The coefficient tuple $\mathbf{y}'=(y_1',\dots,y_n')$ is given by
\begin{equation}
\label{eq:y-mutation}
y'_j =
\begin{cases}
y_k^{-1} & \text{if $j = k$};\\[.05in]
y_j y_k^{[b_{kj}]_+}
(y_k \oplus 1)^{- b_{kj}} & \text{if $j \neq k$}.
\end{cases}
\end{equation}
\item
The cluster $\mathbf{x}'=(x_1',\dots,x_n')$ is given by
$x_j'=x_j$ for $j\neq k$,
whereas $x'_k \in \mathcal{F}$ is determined
by the \emph{exchange relation}
\begin{equation}
\label{exchange relation}
x'_k = \frac
{y_k \ \prod x_i^{[b_{ik}]_+}
+ \ \prod x_i^{[-b_{ik}]_+}}{(y_k \oplus 1) x_k} \, .
\end{equation}
\end{itemize}
\end{definition}

We say that two exchange matrices $B$ and $B'$ are {\it mutation-equivalent}
if one can get from $B$ to $B'$ by a sequence of mutations.
\begin{definition}
\label{def:patterns}
Consider the \emph{$n$-regular tree}~$\mathbb{T}_n$
whose edges are labeled by the numbers $1, \dots, n$,
so that the $n$ edges emanating from each vertex receive
different labels.
A \emph{cluster pattern}  is an assignment
of a labeled seed $\Sigma_t=(\mathbf{x}_t, \mathbf{y}_t, B_t)$
to every vertex $t \in \mathbb{T}_n$, such that the seeds assigned to the
endpoints of any edge $t {k\over\quad} t'$ are obtained from each
other by the seed mutation in direction~$k$.
The components of $\Sigma_t$ are written as:
\begin{equation}
\label{eq:seed-labeling}
\mathbf{x}_t = (x_{1;t}\,,\dots,x_{n;t})\,,\quad
\mathbf{y}_t = (y_{1;t}\,,\dots,y_{n;t})\,,\quad
B_t = (b^t_{ij})\,.
\end{equation}
\end{definition}

Clearly, a cluster pattern  is uniquely determined
by an arbitrary  seed.

\begin{definition}
\label{def:cluster-algebra}
Given a cluster pattern, we denote
\begin{equation}
\label{eq:cluster-variables}
\mathcal{X}
= \bigcup_{t \in\mathbb{T}_n} \mathbf{x}_t
= \{ x_{i,t}\,:\, t \in \mathbb{T}_n\,,\ 1\leq i\leq n \} \ ,
\end{equation}
the union of clusters of all the seeds in the pattern.
The elements $x_{i,t}\in \mathcal{X}$ are called \emph{cluster variables}.
The
\emph{cluster algebra} $\mathcal{A}$ associated with a
given pattern is the $\ZZ \mathbb{P}$-subalgebra of the ambient field $\mathcal{F}$
generated by all cluster variables: $\mathcal{A} = \ZZ \mathbb{P}[\mathcal{X}]$.
We denote $\mathcal{A} = \mathcal{A}(\mathbf{x}, \mathbf{y}, B)$, where
$(\mathbf{x},\mathbf{y},B)$
is any seed in the underlying cluster pattern.
\end{definition}

The cluster algebra is called \emph{skew-symmetric} if the matrix $B$ is skew-symmetric. In this case, it is often convenient to represent the $n\times n$ matrix $B$ by a quiver $Q_B$ with vertices $1,2,\ldots, n$ and $[b_{ij}]_+$ arrows from vertex $i$ to vertex $j$.

If $\mathbb{P}=1$ then the cluster algebra is said to be \emph{coefficient-free}.

The main result in this paper is on coefficient-free cluster algebras. However, we need cluster algebras with coefficients in section \ref{sect rank 2}.

In \cite{FZ}, Fomin and Zelevinsky proved the remarkable {\it Laurent phenomenon} and posed the following {\it positivity conjecture}.
\begin{theorem} [Laurent Phenomenon]
\label{Laurent}
 For any cluster algebra $\mathcal{A}$ and any seed $\Sigma_t$, each cluster variable $x$  is a Laurent polynomial over $\mathbb{ZP}$ in the cluster variables from $\mathbf{x}_t = (x_{1;t} , . . . , x_{n;t} )$.\end{theorem}

\begin{conjecture}[Positivity Conjecture] For any cluster algebra $\mathcal{A}$, any seed $\Sigma$, and any cluster variable $x$, the Laurent polynomial  has coefficients which are nonnegative integer linear combinations of elements in $\mathbb{P}$.
	\end{conjecture}

\subsection{Cluster algebras with principal coefficients}\label{sect principal coefficients}
 One important choice for $\mathbb{P}$ is the tropical semifield; in this case we say that the corresponding cluster algebra is of geometric type.

\begin{definition}[Tropical semifield]
Let $\text{Trop}(u_1,\cdots,u_m)$ be an abelian group (written multiplicatively) freely generated by the $u_j$. We define $\oplus$ in $\text{Trop}(u_1,\cdots,u_m)$ by
$$
\prod_j u_j^{a_j} \oplus \prod_j u_j^{b_j} = \prod_j u_j^{\min(a_j,b_j)},
$$and call $(\text{Trop}(u_1,\cdots,u_m), \oplus, \cdot)$ a tropical semifield.
\end{definition}

\begin{remark}\label{rectangular}
In cluster algebras whose ground ring is
$\mathrm{Trop}(u_1,\dots, u_{m})$ (the tropical semifield), it is convenient to replace the
matrix $B$ by an $(n+m)\times n$ matrix $\tilde B=(b_{ij})$ whose upper part
is the $n\times n$ matrix $B$ and whose lower part is an $m\times
n$ matrix that encodes the coefficient tuple via
\begin{equation}\label{eq 20}
y_k = \prod_{i=1}^{m} u_i^{b_{(n+i)k}}.
\end{equation}
Then the mutation of the coefficient tuple in equation (\ref{eq:y-mutation})
is determined by the mutation
of the matrix $\tilde B$ in equation (\ref{eq:matrix-mutation}) and the formula (\ref{eq 20}); and the
exchange relation (\ref{exchange relation}) becomes
\begin{equation}\label{geometric exchange}
 x_k'=x_k^{-1} \left( \prod_{i=1}^n x_i^{[b_{ik}]_+}
\prod_{i=1}^{m} u_i^{[b_{(n+i)k}]_+}
+\prod_{i=1}^n x_i^{[-b_{ik}]_+}
\prod_{i=1}^{m} u_i^{[-b_{(n+i)k}]_+}
\right).
\end{equation}
\end{remark}

Recall from \cite{FZ4} that a
cluster algebra~$\mathcal{A}$  is said to have
\emph{principal coefficients at a vertex~$t$} if
$\mathbb{P}= \textup{Trop}(y_1, \dots, y_n)$ and
$\mathbf{y}_{t}= (y_1, \dots, y_n)$.

\begin{definition}\label{def X}
Let~$\mathcal{A}$ be the cluster algebra with principal coefficients at
$t$, defined by the initial seed
$(\mathbf{x}_{t}\,,\mathbf{y}_{t}\,B_t)$ with
\[\mathbf{x}_{t} = (x_1, \dots, x_n), \quad
\mathbf{y}_{t} = (y_1, \dots, y_n).
\]
\begin{enumerate}
\item Let
$X_{\ell;t} $
be the
Laurent expansion of the cluster variable $x_{\ell;t}$ in $x_1, \dots, x_n, y_1, \dots, y_n$.
\item The
\emph{F-polynomial} of the cluster variable $x_{\ell,t}$ is defined as
 $F_{\ell;t}= X_{\ell;t}(1, \dots, 1; y_1, \dots, y_n).$
\item The \emph{$g$-vector} $\mathbf{g}_{\ell;t}$ of the cluster variable $x_{\ell,t}$ is defined as the degree vector of the monomial $X_{\ell;t}(x_1, \dots, x_n; 0, \dots, 0).$
\end{enumerate}
\end{definition}

The following theorem shows that expansion formulas in principal coefficients can be used to compute expansions in arbitrary coefficient systems.
\begin{theorem}\cite[Theorem 3.7]{FZ4}
\label{thm X}
Let $\mathcal{A}$ be a cluster algebra over an arbitrary semifield $\mathbb{P}$
with initial seed
$$((x_1, \dots, x_n), (\hat y_1, \dots, \hat y_n), B).$$
Then the cluster variables in~$\mathcal{A}$ can be expressed as follows:
\begin{equation}
\label{eq:xjt-reduction-principal}
x_{\ell;t} = \frac{X_{\ell;t} (x_1, \dots, x_n;\hat y_1, \dots,\hat y_n)}
{F_{\ell;t}|_\mathbb{P} (\hat y_1, \dots, \hat y_n)} \, .
\end{equation}
where ${F_{\ell;t}|_\mathbb{P} (\hat y_1, \dots, \hat y_n)}$ is the $F$-polynomial evaluated at $\hat y_1, \dots, \hat y_n$ inside the semifield $\mathbb{P}$.
\end{theorem}

\section{Rank 2 considerations}\label{sect rank 2}
In this section, we use the rank 2 formula from \cite{LS} (in the parametrization of \cite{LLZ}) to compute in a non-acyclic cluster algebra of rank three the  Laurent expansions of those cluster variables which are obtained from the initial cluster by a mutation sequence involving only two vertices.

\subsection{Rank 2 formula} We start by recalling  from \cite{LLZ} the formula for the Laurent expansion of an arbitrary cluster variable in the cluster algebra of rank 2 given by the initial quiver with $r$ arrows
\[ \xymatrix{1\ar^r[rr] &&2}\]
 where $r\ge 2$ is a positive integer. Then the cluster variables $x_n$ in this cluster algebra are defined by  the following recursion:
$$ x_{n+1} =(x_n^r +1)/{x_{n-1}} \text{ for any integer }n.
$$
Let $\{c_n^{[r]}\}_{n\in\mathbb{Z}}$ be the sequence  defined by the recurrence relation $$c_n^{[r]}=rc_{n-1}^{[r]} -c_{n-2}^{[r]},$$ with the initial condition $c^{[r]}_1=0$, $c^{[r]}_2=1$. The $c_n^{[r]}$ are Chebyshev polynomials.
For example, if $r=2$ then $c^{[r]}_n=n-1$;
if $r=3$, the sequence $c_n^{[r]}$ takes the following values:
\[\ldots,-3,-1,0,1,3,8,21,55,144,\ldots\]
The pair of absolute values of the integers $(c_{n-1}^{[r]},c^{[r]}_{n-2})$ is the degree of the denominator of the cluster variable $x_n$.

If the value of $r$ is clear from the context, we usually write $c_n$ instead of $c_n^{[r]}$.
\begin{lemma}
 \label{lem cn} Let $n\geq 3$. We have
 $c_{n-1}c_{n+k-3} - c_{n+k-2}c_{n-2}=c_k$ for $k\in \mathbb{Z}$. In particular, we have
 $c_{n-1}^2 -c_nc_{n-2}=1$.
\end{lemma}
\begin{proof}
 The result holds for $n=3$. Suppose that $n\ge 4$, then
 \[\begin{array}{rcl}
c_{n+k-2}c_{n-2} &=& rc_{n+k-3}c_{n-2}-c_{n+k-4}c_{n-2} \\
&\stackrel{*}{=}&  rc_{n+k-3}c_{n-2}-(c_k+c_{n+k-3}c_{n-3})\\
&=&  c_{n+k-3}(rc_{n-2} -c_{n-3}) -c_k\\
&=& c_{n+k-3}c_{n-1} -c_k,
\end{array}\] where equation $*$ holds by induction.
\end{proof}

Let $(a_1, a_2)$ be a pair of nonnegative integers.
A \emph{Dyck path} of type $a_1\times a_2$ is a lattice path
from $(0, 0)$  to $(a_1,a_2)$ that
never goes above the main diagonal joining $(0,0)$ and $(a_1,a_2)$.
Among the Dyck paths of a given type $a_1\times a_2$, there is a (unique) \emph{maximal} one denoted by
$\mathcal{D} = \mathcal{D}^{a_1\times a_2}$.
It is defined by the property that any lattice point strictly above $\mathcal{D}$ is also strictly above the main diagonal.

Let $\mathcal{D}=\mathcal{D}^{a_1\times a_2}$.  Let $\mathcal{D}_1=\{u_1,\dots,u_{a_1}\}$ be the set of horizontal edges of $\mathcal{D}$ indexed from left to right, and $\mathcal{D}_2=\{v_1,\dots, v_{a_2}\}$ the set of vertical edges of $\mathcal{D}$ indexed from bottom to top.
Given any points $A$ and $B$ on $\mathcal{D}$, let $AB$ be the subpath starting from $A$, and going in the Northeast direction until it reaches $B$ (if we reach $(a_1,a_2)$ first, we continue from $(0,0)$). By convention, if $A=B$, then $AA$ is the subpath that starts from $A$, then passes $(a_1,a_2)$ and ends at $A$. If we represent a subpath of $\mathcal{D}$ by its set of edges, then for $A=(i,j)$ and $B=(i',j')$, we have
$$
AB=
\begin{cases}
\{u_k, v_\ell: i < k \leq i', j < \ell \leq j'\}, \quad\textrm{if $B$ is to the Northeast of $A$};\\
\mathcal{D} - \{u_k, v_\ell: i' < k \leq i, j' < \ell \leq j\}, \quad\textrm{otherwise}.
\end{cases}
$$
We denote by $(AB)_1$ the set of horizontal edges in $AB$, and by $(AB)_2$ the set of vertical edges in $AB$.
Also let $AB^\circ$ denote the set of lattice points on the subpath $AB$ excluding the endpoints $A$ and $B$ (here $(0,0)$ and $(a_1,a_2)$ are regarded as the same point).

Here is an example for $(a_1,a_2)=(6,4)$.

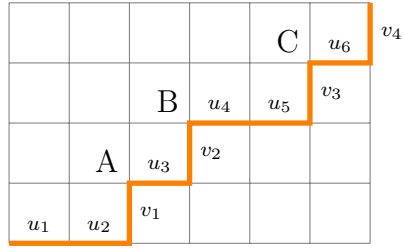
\begin{figure}[h]
\begin{tikzpicture}[scale=.8]
\draw[step=1,color=gray] (0,0) grid (6,4);
\draw[line width=2,color=orange] (0,0)--(2,0)--(2,1)--(3,1)--(3,2)--(5,2)--(5,3)--(6,3)--(6,4);
%
\draw (0.5,0) node[anchor=south]  {\tiny$u_1$};
\draw (1.5,0) node[anchor=south]  {\tiny$u_2$};
\draw (2.5,1) node[anchor=south]  {\tiny$u_3$};
\draw (3.5,2) node[anchor=south]  {\tiny$u_4$};
\draw (4.5,2) node[anchor=south]  {\tiny$u_5$};
\draw (5.5,3) node[anchor=south]  {\tiny$u_6$};
\draw (6,3.5) node[anchor=west]  {\tiny$v_4$};
\draw (5,2.5) node[anchor=west]  {\tiny$v_3$};
\draw (3,1.5) node[anchor=west]  {\tiny$v_2$};
\draw (2,.5) node[anchor=west]  {\tiny$v_1$};
\draw (2,1) node[anchor=south east] {A};
\draw (3,2) node[anchor=south east] {B};
\draw (5,3) node[anchor=south east] {C};
\end{tikzpicture}
\caption{A maximal Dyck path.}
\label{fig:Dyck-path}
\end{figure}

Let $A=(2,1)$, $B=(3,2)$ and $C=(5,3)$. Then
$$
(AB)_1=\{u_3\}, \,\, (AB)_2=\{v_2\}, \,\, 
(BA)_1=\{u_4,u_5,u_6,u_1,u_2\}, \, \, (BA)_2=\{v_3,v_4,v_1\} \ . 
$$
The point $C$ is in $BA^\circ$ but not in $AB^\circ$. The subpath $AA$ has length 10 (not 0).

\begin{definition}
\label{df:compatible}
For $S_1\subseteq \mathcal{D}_1$, $S_2\subseteq \mathcal{D}_2$, we say that the pair $(S_1,S_2)$ is compatible if for every $u\in S_1$ and $v\in S_2$, denoting by $E$ the left endpoint of $u$ and $F$ the upper endpoint of $v$, there exists a lattice point $A\in EF^\circ$ such that
\begin{equation}
\label{0407df:comp}
|(AF)_1|=r|(AF)_2\cap S_2|\textrm{\; or\; }|(EA)_2|=r|(EA)_1\cap S_1|.\end{equation}
\end{definition}

With all this terminology in place we are ready to present the combinatorial expression for greedy elements.  The following Theorem has been proved in \cite{LS3arXiv,LLZ}.

\begin{theorem}
\label{th:greedy-combinatorial}
For $n\geq 3$, we have
\begin{equation}
\label{eq:greedy-Dyck-expression}
x_n = x_1^{-c_{n-1}}x_2^{-c_{n-2}}\sum_{(S_1,S_2)}x_1^{r|S_2|}x_2^{r|S_1|}
\end{equation}
and
\begin{equation}
\label{eq:greedy-Dyck-expression2}
x_{3-n} = x_2^{-c_{n-1}}x_1^{-c_{n-2}}\sum_{(S_1,S_2)}x_2^{r|S_2|}x_1^{r|S_1|},
\end{equation}
where the sum is over all compatible pairs $(S_1,S_2)$ in $\mathcal{D}^{c_{n-1}\times c_{n-2}}$.
\end{theorem}
\begin{remark}
For $n=2$, the formula is consistent if we impose the additional convention that $\mathcal{D}^{c_1\times c_0}$ is the empty set.
\end{remark}

\begin{example}\label{mainexmp} Let $r=3$ and $n=5$.  Then
$\mathcal{D}^{8\times 3}$ is the following path.
$$\begin{array}{cc}
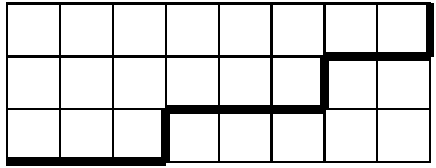
\hspace{16pt} \begin{picture}(140,60)
\boxs{0,0}\boxs{20,0}\boxs{40,0}\boxs{60,0}\boxs{80,0}\boxs{100,0}\boxs{120,0}\boxs{140,0}
\boxs{0,20}\boxs{20,20}\boxs{40,20}\boxs{60,20}\boxs{80,20}\boxs{100,20}\boxs{120,20}\boxs{140,20}
\boxs{0,40}\boxs{20,40}\boxs{40,40}\boxs{60,40}\boxs{80,40}\boxs{100,40}\boxs{120,40}\boxs{140,40}
\linethickness{3pt}\put(0,0){\line(1,0){60}}
\linethickness{3pt}\put(60,0){\line(0,1){20}}
\linethickness{3pt}\put(60,20){\line(1,0){60}}
\linethickness{3pt}\put(120,20){\line(0,1){20}}
\linethickness{3pt}\put(120,40){\line(1,0){40}}
\linethickness{3pt}\put(160,40){\line(0,1){20}}
\linethickness{3pt}{\line(1,0){6}}
\end{picture}
\end{array}$$
The illustrations
below show the possible configurations for compatible pairs in
$\mathcal{D}^{8\times 3}$. If the edge $u_i$ is marked
$\begin{picture}(25,10)\linethickness{3pt}\put(2,2){\line(1,0){20}}\color{white}\put(9,2){\line(1,0){6}}
\end{picture}$, then $u_i$ can occur in $S_1$.

$$
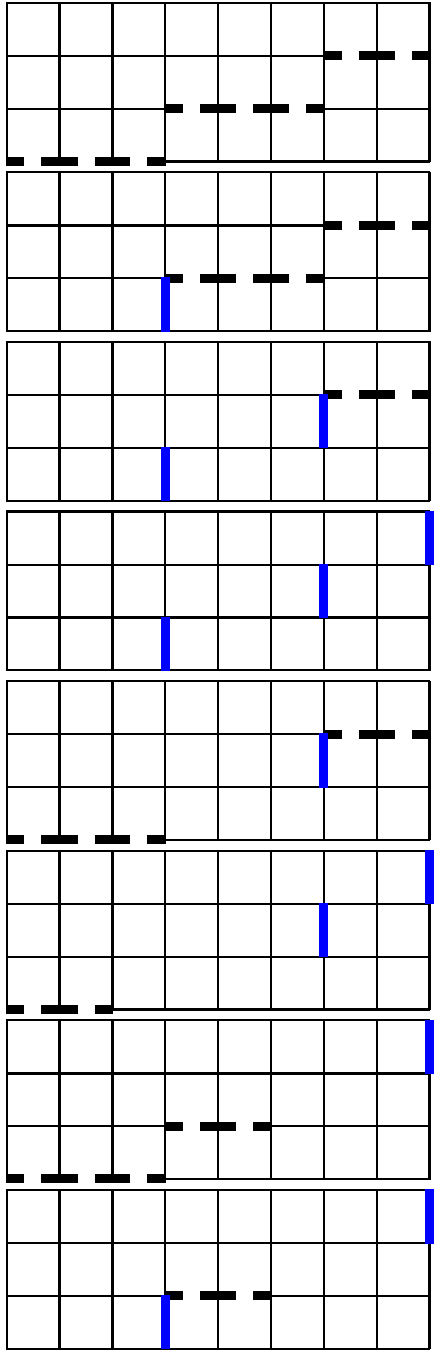
\begin{array}{cc}
\hspace{16pt} \begin{picture}(100,60)
\boxs{0,0}\boxs{20,0}\boxs{40,0}\boxs{60,0}\boxs{80,0}\boxs{100,0}\boxs{120,0}\boxs{140,0}
\boxs{0,20}\boxs{20,20}\boxs{40,20}\boxs{60,20}\boxs{80,20}\boxs{100,20}\boxs{120,20}\boxs{140,20}
\boxs{0,40}\boxs{20,40}\boxs{40,40}\boxs{60,40}\boxs{80,40}\boxs{100,40}\boxs{120,40}\boxs{140,40}
\linethickness{3pt}\put(0,0){\line(1,0){60}}
\linethickness{3pt}\put(60,20){\line(1,0){60}}
\linethickness{3pt}\put(120,40){\line(1,0){40}}
\color{white}\put(67,20){\line(1,0){6}}\linethickness{3pt}\color{white}\put(107,20){\line(1,0){6}}\color{white}\put(7,0){\line(1,0){6}}\color{white}\put(27,0){\line(1,0){6}}\color{white}\put(87,20){\line(1,0){6}}\color{white}\put(147,40){\line(1,0){6}}
\linethickness{3pt}\color{white}\put(47,0){\line(1,0){6}}\color{white}\put(127,40){\line(1,0){6}}
\end{picture}
&\,\,\,\,\,\,\,\,\,\,\,\,\,\,\,\,\,\,\,\,\,\,\,\,\,\,\,\,\,\,\,\,\,\,\,\,\,\,\,\,\,\,\,\,\,\,\,\,\,\,\,\,\,\,\,\,\,\,\,\,\,\,\,\,\,\,\,\,\,\,\,\,\,\,\,\,\,\,\,\,\,\begin{picture}(300,60)\put(-80,35){$\tiny{\sum_{\begin{array}{l}S_1\subset
\{u_1,\cdots,u_{8} \}, S_2=\emptyset
\end{array}}x_1^{r|S_2|}x_2^{r|S_1|}}$}\put(-80,15){$=(1+x_2^3)^8$}\end{picture}
\\
\hspace{16pt} \begin{picture}(100,60)
\boxs{0,0}\boxs{20,0}\boxs{40,0}\boxs{60,0}\boxs{80,0}\boxs{100,0}\boxs{120,0}\boxs{140,0}
\boxs{0,20}\boxs{20,20}\boxs{40,20}\boxs{60,20}\boxs{80,20}\boxs{100,20}\boxs{120,20}\boxs{140,20}
\boxs{0,40}\boxs{20,40}\boxs{40,40}\boxs{60,40}\boxs{80,40}\boxs{100,40}\boxs{120,40}\boxs{140,40}
\linethickness{3pt}\put(60,20){\line(1,0){60}}
\linethickness{3pt}\put(120,40){\line(1,0){40}}
\color{white}\put(67,20){\line(1,0){6}}\linethickness{3pt}\color{white}\put(107,20){\line(1,0){6}}\color{white}\color{white}\put(87,20){\line(1,0){6}}\color{white}\put(147,40){\line(1,0){6}}
\linethickness{3pt}\color{white}\color{white}\put(127,40){\line(1,0){6}}
\linethickness{3pt}\color{blue}\put(60,0){\line(0,1){20}}
\end{picture}
&\,\,\,\,\,\,\,\,\,\,\,\,\,\,\,\,\,\,\,\,\,\,\,\,\,\,\,\,\,\,\,\,\,\,\,\,\,\,\,\,\,\,\,\,\,\,\,\,\,\,\,\,\,\,\,\,\,\,\,\,\,\,\,\,\,\,\,\,\,\,\,\,\,\,\,\,\,\,\,\,\,\begin{picture}(300,60)\put(-80,35){$\tiny{\sum_{\begin{array}{l}S_1\subset
\{u_4,\cdots,u_{8} \}, S_2=\{v_1\}
\end{array}}x_1^{r|S_2|}x_2^{r|S_1|}}$}\put(-80,15){$=
x_1^3(1+x_2^3)^5$}\end{picture}\\
\hspace{16pt} \begin{picture}(100,60)
\boxs{0,0}\boxs{20,0}\boxs{40,0}\boxs{60,0}\boxs{80,0}\boxs{100,0}\boxs{120,0}\boxs{140,0}
\boxs{0,20}\boxs{20,20}\boxs{40,20}\boxs{60,20}\boxs{80,20}\boxs{100,20}\boxs{120,20}\boxs{140,20}
\boxs{0,40}\boxs{20,40}\boxs{40,40}\boxs{60,40}\boxs{80,40}\boxs{100,40}\boxs{120,40}\boxs{140,40}
\linethickness{3pt}\put(120,40){\line(1,0){40}}
\color{white}\put(147,40){\line(1,0){6}}
\linethickness{3pt}\color{white}\color{white}\put(127,40){\line(1,0){6}}
\linethickness{3pt}\color{blue}\put(60,0){\line(0,1){20}}\put(120,20){\line(0,1){20}}
\end{picture}
&\,\,\,\,\,\,\,\,\,\,\,\,\,\,\,\,\,\,\,\,\,\,\,\,\,\,\,\,\,\,\,\,\,\,\,\,\,\,\,\,\,\,\,\,\,\,\,\,\,\,\,\,\,\,\,\,\,\,\,\,\,\,\,\,\,\,\,\,\,\,\,\,\,\,\,\,\,\,\,\,\,\begin{picture}(300,60)\put(-80,35){$\tiny{\sum_{\begin{array}{l}S_1\subset\{u_7,u_8\},
S_2=\{v_1,v_2\} \end{array}}x_1^{r|S_2|}x_2^{r|S_1|}}$}\put(-80,15){$=
x_1^6(1+x_2^3)^2$}\end{picture}\\
\hspace{16pt} \begin{picture}(100,60)
\boxs{0,0}\boxs{20,0}\boxs{40,0}\boxs{60,0}\boxs{80,0}\boxs{100,0}\boxs{120,0}\boxs{140,0}
\boxs{0,20}\boxs{20,20}\boxs{40,20}\boxs{60,20}\boxs{80,20}\boxs{100,20}\boxs{120,20}\boxs{140,20}
\boxs{0,40}\boxs{20,40}\boxs{40,40}\boxs{60,40}\boxs{80,40}\boxs{100,40}\boxs{120,40}\boxs{140,40}
\linethickness{3pt}\color{blue}\put(60,0){\line(0,1){20}}\put(120,20){\line(0,1){20}}\put(160,40){\line(0,1){20}}
\end{picture}
&\,\,\,\,\,\,\,\,\,\,\,\,\,\,\,\,\,\,\,\,\,\,\,\,\,\,\,\,\,\,\,\,\,\,\,\,\,\,\,\,\,\,\,\,\,\,\,\,\,\,\,\,\,\,\,\,\,\,\,\,\,\,\,\,\,\,\,\,\,\,\,\,\,\,\,\,\,\,\,\,\,\begin{picture}(300,60)\put(-80,35){$\tiny{\sum_{\begin{array}{l}S_1=\emptyset,
S_2=\{v_1,v_2,v_3\}
\end{array}}x_1^{r|S_2|}x_2^{r|S_1|}}$}\put(-80,15){$=
x_1^9$}\end{picture}\\
\hspace{16pt} \begin{picture}(100,60)
\boxs{0,0}\boxs{20,0}\boxs{40,0}\boxs{60,0}\boxs{80,0}\boxs{100,0}\boxs{120,0}\boxs{140,0}
\boxs{0,20}\boxs{20,20}\boxs{40,20}\boxs{60,20}\boxs{80,20}\boxs{100,20}\boxs{120,20}\boxs{140,20}
\boxs{0,40}\boxs{20,40}\boxs{40,40}\boxs{60,40}\boxs{80,40}\boxs{100,40}\boxs{120,40}\boxs{140,40}
\linethickness{3pt}\put(0,0){\line(1,0){60}}
\linethickness{3pt}\put(120,40){\line(1,0){40}}
\linethickness{3pt}\color{white}\put(7,0){\line(1,0){6}}\color{white}\put(27,0){\line(1,0){6}}\color{white}\put(147,40){\line(1,0){6}}
\linethickness{3pt}\color{white}\put(47,0){\line(1,0){6}}\color{white}\put(127,40){\line(1,0){6}}
\linethickness{3pt}\color{blue}\put(120,20){\line(0,1){20}}
\end{picture}
&\,\,\,\,\,\,\,\,\,\,\,\,\,\,\,\,\,\,\,\,\,\,\,\,\,\,\,\,\,\,\,\,\,\,\,\,\,\,\,\,\,\,\,\,\,\,\,\,\,\,\,\,\,\,\,\,\,\,\,\,\,\,\,\,\,\,\,\,\,\,\,\,\,\,\,\,\,\,\,\,\,\begin{picture}(300,60)\put(-80,35){$\tiny{\sum_{\begin{array}{l}S_1\subset
\{u_1,u_2,u_{3},u_{7},u_{8} \}, S_2=\{v_2\}
\end{array}}x_1^{r|S_2|}x_2^{r|S_1|}}$}\put(-80,15){$=
x_1^3(1+x_2^3)^5$}\end{picture}\\
\hspace{16pt} \begin{picture}(100,60)
\boxs{0,0}\boxs{20,0}\boxs{40,0}\boxs{60,0}\boxs{80,0}\boxs{100,0}\boxs{120,0}\boxs{140,0}
\boxs{0,20}\boxs{20,20}\boxs{40,20}\boxs{60,20}\boxs{80,20}\boxs{100,20}\boxs{120,20}\boxs{140,20}
\boxs{0,40}\boxs{20,40}\boxs{40,40}\boxs{60,40}\boxs{80,40}\boxs{100,40}\boxs{120,40}\boxs{140,40}
\linethickness{3pt}\put(0,0){\line(1,0){40}}
\linethickness{3pt}\color{white}\put(7,0){\line(1,0){6}}\color{white}\put(27,0){\line(1,0){6}}
\linethickness{3pt}\color{blue}\put(120,20){\line(0,1){20}}\put(160,40){\line(0,1){20}}
\end{picture}
&\,\,\,\,\,\,\,\,\,\,\,\,\,\,\,\,\,\,\,\,\,\,\,\,\,\,\,\,\,\,\,\,\,\,\,\,\,\,\,\,\,\,\,\,\,\,\,\,\,\,\,\,\,\,\,\,\,\,\,\,\,\,\,\,\,\,\,\,\,\,\,\,\,\,\,\,\,\,\,\,\,\begin{picture}(300,60)\put(-80,35){$\tiny{\sum_{\begin{array}{l}S_1\subset
\{u_1,u_2\}, S_2=\{v_2,v_3\}
\end{array}}x_1^{r|S_2|}x_2^{r|S_1|}}$}\put(-80,15){$=
x_1^6(1+x_2^3)^2$}\end{picture}\\
\hspace{16pt} \begin{picture}(100,60)
\boxs{0,0}\boxs{20,0}\boxs{40,0}\boxs{60,0}\boxs{80,0}\boxs{100,0}\boxs{120,0}\boxs{140,0}
\boxs{0,20}\boxs{20,20}\boxs{40,20}\boxs{60,20}\boxs{80,20}\boxs{100,20}\boxs{120,20}\boxs{140,20}
\boxs{0,40}\boxs{20,40}\boxs{40,40}\boxs{60,40}\boxs{80,40}\boxs{100,40}\boxs{120,40}\boxs{140,40}
\linethickness{3pt}\put(0,0){\line(1,0){60}}
\linethickness{3pt}\put(60,20){\line(1,0){40}}
\color{white}\put(67,20){\line(1,0){6}}\linethickness{3pt}\color{white}\put(7,0){\line(1,0){6}}\color{white}\put(27,0){\line(1,0){6}}\color{white}\put(87,20){\line(1,0){6}}
\linethickness{3pt}\color{white}\put(47,0){\line(1,0){6}}
\linethickness{3pt}\color{blue}\put(160,40){\line(0,1){20}}
\end{picture}
&\,\,\,\,\,\,\,\,\,\,\,\,\,\,\,\,\,\,\,\,\,\,\,\,\,\,\,\,\,\,\,\,\,\,\,\,\,\,\,\,\,\,\,\,\,\,\,\,\,\,\,\,\,\,\,\,\,\,\,\,\,\,\,\,\,\,\,\,\,\,\,\,\,\,\,\,\,\,\,\,\,\begin{picture}(300,60)\put(-80,35){$\tiny{\sum_{\begin{array}{l}S_1\subset
\{u_1,u_2,u_3,u_4,u_5\}, S_2=\{v_3\}
\end{array}}x_1^{r|S_2|}x_2^{r|S_1|}}$}\put(-80,15){$=
x_1^3(1+x_2^3)^5$}\end{picture}\\
\hspace{16pt} \begin{picture}(100,60)
\boxs{0,0}\boxs{20,0}\boxs{40,0}\boxs{60,0}\boxs{80,0}\boxs{100,0}\boxs{120,0}\boxs{140,0}
\boxs{0,20}\boxs{20,20}\boxs{40,20}\boxs{60,20}\boxs{80,20}\boxs{100,20}\boxs{120,20}\boxs{140,20}
\boxs{0,40}\boxs{20,40}\boxs{40,40}\boxs{60,40}\boxs{80,40}\boxs{100,40}\boxs{120,40}\boxs{140,40}
\linethickness{3pt}\put(60,20){\line(1,0){40}}
\color{white}\put(67,20){\line(1,0){6}}\linethickness{3pt}\color{white}\put(87,20){\line(1,0){6}}
\linethickness{3pt}\color{blue}\put(160,40){\line(0,1){20}}\put(60,0){\line(0,1){20}}
\end{picture}
&\,\,\,\,\,\,\,\,\,\,\,\,\,\,\,\,\,\,\,\,\,\,\,\,\,\,\,\,\,\,\,\,\,\,\,\,\,\,\,\,\,\,\,\,\,\,\,\,\,\,\,\,\,\,\,\,\,\,\,\,\,\,\,\,\,\,\,\,\,\,\,\,\,\,\,\,\,\,\,\,\,\begin{picture}(300,60)\put(-80,35){$\tiny{\sum_{\begin{array}{l}S_1\subset
\{u_4,u_5\}, S_2=\{v_1,v_3\}
\end{array}}x_1^{r|S_2|}x_2^{r|S_1|}}$}\put(-80,15){$=
x_1^6(1+x_2^3)^2$}\end{picture}\end{array}
$$Adding the above 8 polynomials together gives
\begin{equation}\label{8eqtoge}\aligned  &{x_2}^{24}+8 {x_2}^{21}+3
{x_1}^{3} {x_2}^{15}+28
     {x_2}^{18}+15 {x_1}^{3} {x_2}^{12}+56 {x_2}^{15}+3
     {x_1}^{6} {x_2}^{6}+30 {x_1}^{3} {x_2}^{9}+70
     {x_2}^{12}\\&+{x_1}^{9}+6 {x_1}^{6} {x_2}^{3}+30
     {x_1}^{3} {x_2}^{6}+56 {x_2}^{9}+3 {x_1}^{6}+15
     {x_1}^{3} {x_2}^{3}+28 {x_2}^{6}+3 {x_1}^{3}+8
     {x_2}^{3}+1.\endaligned\end{equation}
     Then $x_5$ is obtained by dividing (\ref{8eqtoge}) by $x_1^8 x_2^3$.
\end{example}

The following Corollary can be adapted from results of \cite{LS}.
Let $g_\ell$ be the $g$-vector and let $F_\ell$ be the $F$-polynomial of $x_\ell$, for all integers $\ell$.
{Then $g_3=(-1,r), \,g_0=(0,-1),\,F_3=y_1+1$ and $F_0=y_2+1$, and all other cases are described in the following result.}
\begin{corollary}\label{01022012} Let $n\ge 3$. Then
\[ \begin{array}{lrclcccrcl}
&g_n&=&(-c_{n-1},c_n)& \quad &,&\quad& g_{3-n}&=&(-c_{n-2},c_{n-3}) ,\\
and \quad \\
& F_n&=& \sum_{(S_1,S_2)}y_1^{c_{n-1}-|S_1|}y_2^{|S_2|}&&,& &
F_{3-n}&=&\sum_{(S_1,S_2)}y_1^{c_{n-2}-|S_2|}y_2^{|S_1|},
\end{array}\]
where the sum is over all compatible pairs $(S_1,S_2)$ in $\mathcal{D}^{c_{n-1}\times c_{n-2}}$.
\end{corollary}

Recall from Definition \ref{def X} that
 $X_n$ is the Laurent polynomial in $x_1,x_2,y_1,y_2$ corresponding to the expansion of $x_n$ in the cluster algebra with principal coefficients in the seed $((x_1,x_2),(y_1,y_2),Q)$.
\begin{corollary}\label{cor X}Let $n\ge 3$. Then
 \[X_n=  x_1^{-c_{n-1}} x_2^{-c_{n-2}}\sum_{(S_1,S_2)}x_1^{r|S_2|}x_2^{r|S_1|}y_1^{c_{n-1}-|S_1|}y_2^{|S_2|}.\]
\end{corollary}
\begin{proof}
 This follows directly from Theorem \ref{th:greedy-combinatorial} and Corollary \ref{01022012}.
\end{proof}

\subsection{A preliminary lemma} Let $\mathcal{A}(Q)$ be a
coefficient-free cluster algebra of rank 3 with initial quiver $Q$ equal to
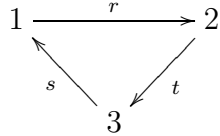
\[\xymatrix{1\ar[rr]^r&&2\ar[ld]^t\\&3\ar[lu]^s}\]
 where $r,s,t$ denote the number of arrows.

Let $x_n$ be the cluster variable  obtained from the  seed $((x_1,x_2,z_3),Q)$  by the sequence of $n-2$ mutations 1,2,1,2,1,\ldots

From Theorem \ref{thm X}, we have
\[x_n=\frac{X_n(x_1,x_2,\hat y_1,\hat y_2)}{F_n|_{\hat{\mathbb{P}}}(\hat y_1,\hat y_2)},\]
where $X_n$ is as in Corollary \ref{cor X}, $(\hat y_1,\hat y_2)=(z_3^s,z_3^{-t})$,
and the denominator is the $F$-polynomial evaluated in $\hat y$ and with tropical addition. Thus, by Corollary \ref{01022012}, the denominator is
$\sum^\oplus_\beta z_3^{s(c_{n-1}-|S_1|)-t|S_2|}$, which is equal to $z_3^{m}$ where $m$ is the smallest power occurring in this sum.
Therefore we get
\begin{equation}\label{eq thm33}x_n = x_1^{-c_{n-1}} x_2^{-c_{n-2}}\sum_{(S_1,S_2)}x_1^{r|S_2|}x_2^{r|S_1|} z_3^{s(c_{n-1}-|S_1|)-t|S_2|-m} \end{equation}
where $m=\min _{(S_1,S_2)} (s(c_{n-1}-|S_1|)-t|S_2|)$.

We shall need a precise value for $m$. As a first step, we determine which compatible pair $(S_1,S_2)$ in $\mathcal{D}^{c_{n-1}\times c_{n-2}}$  can realize the minimum $m$.

\begin{lemma}\label{12312011}
Let $s$ and $t$ be nonzero integers such that there are nonzero arrows between all pairs of the three vertices in any seed between the initial and terminal seeds inclusive.  Then there is a unique compatible pair $(S_1,S_2)$ in $\mathcal{D}^{c_{n-1}\times c_{n-2}}$ which achieves $s(c_{n-1}-|S_1|)-t|S_2|=m$. Such $(S_1,S_2)$ is either $(\mathcal{D}_1,\emptyset)$, $(\emptyset, \mathcal{D}_2)$, or $(\emptyset, \emptyset)$. Moreover, if $s$ and $t$ are positive, then such $(S_1,S_2)$ is either $(\mathcal{D}_1,\emptyset)$ or $(\emptyset, \mathcal{D}_2)$.
\end{lemma}
\begin{proof}
We use induction on $n$. Consider first the case $n=3$. Since $\mathcal {D}^{c_2\times c_1}=\mathcal{D}^{1\times 0}$,  there are exactly two compatible pairs $(\emptyset,\mathcal{D}_2)=(\emptyset,\emptyset)$ and $(\mathcal{D}_1,\emptyset)$. Thus $s(c_{n-1}-|S_1|)-t|S_2|= -s|S_1| $ achieves its minimum in $(S_1,S_2)=(\mathcal{D}_1,\emptyset)$ if $s$ is positive, and in $(S_1,S_2)=(\emptyset,\mathcal{D}_2)$ if $s$ is negative.

Suppose now that $n\ge3$.
Consider the expression given  in  equation (\ref{eq thm33}).
By induction, there is a unique $(\bar S_1,\bar S_2)$ in $\mathcal{D}^{c_{n-1}\times c_{n-2}}$  such that ${s(c_{n-1}-|S_1|)-t|S_2|}=m$.

Suppose first that  $(\bar S_1,\bar S_2)=(\mathcal{D}_1,\emptyset)$. Then the term in (\ref{eq thm33})  that is not divisible by $z_3$ is \begin{equation}\label{01022012eq1}x_1^{-c_{n-1}} x_2^{-c_{n-2}} x_2^{rc_{n-1}} = x_1^{-c_{n-1}} x_2^{c_{n}}.\end{equation}    If $t>0$ then $x_{n+1}$ is obtained from $x_n$ by substituting $$\left\{\begin{array}{l} x_1 \mapsto x_2 \\ x_2 \mapsto \frac{x_2^r + z_3^{t}}{x_1}.  \end{array}\right.$$
So (\ref{01022012eq1}) becomes  $$x_2^{-c_{n-1}} \left(\frac{x_2^r + z_3^t}{x_1}\right)^{c_{n}}=x_1^{-c_n}x_2^{-c_{n-1}}x_1^{0}x_2^{rc_n} z_3^0+\textup{terms divisible by $z_3$}.$$
Thus the term in the expression for $x_{n+1} $ that is not divisible by $z_3$ corresponds to
a compatible pair $(S_1,S_2)$ in $\mathcal{D}^{c_{n}\times c_{n-1}}$ with $|S_1|=c_{n}$, thus
$(S_1,S_2)=(\mathcal{D}_1,\emptyset)$.

On the other hand, if $t<0$ then we substitute $$\left\{\begin{array}{l} x_1 \mapsto x_2 \\ x_2 \mapsto \frac{x_2^r  z_3^{-t} + 1}{x_1},  \end{array}\right.$$ so (\ref{01022012eq1}) becomes  $$x_2^{-c_{n-1}} \left(\frac{x_2^r z_3^{-t}+1}{x_1}\right)^{c_{n}}=x_1^{-c_n}x_2^{-c_{n-1}}x_1^{0}x_2^0z_3^0 +\textup{terms divisible by $z_3$} .$$ Thus the term in the expression for $x_{n+1} $ that is not divisible by $z_3$ corresponds to
a compatible pair  $(S_1,S_2)$ in $\mathcal{D}^{c_{n}\times c_{n-1}}$ with $|S_1|=0$ and $|S_2|=0$, thus
$(S_1,S_2)=(\emptyset, \emptyset)$.

Next suppose that $(\bar S_1,\bar S_2)=(\emptyset, \emptyset)$. Then $c_{n-1}-|\bar S_1|=c_{n-1}$, $|\bar S_2|=0$, $m=sc_{n-1}$ and $t<0$ since $(\emptyset, \emptyset)$ realizes the minimum. Then
$$\aligned \sum_{(S_1,S_2) : |S_1|=0}  x_1^{-c_{n-1}}x_2^{-c_{n-2}}x_1^{r|S_2|}x_2^{r|S_1|}z_3^{s(c_{n-1} - |S_1|)-t|S_2| -m} &=\sum_{(S_1,S_2) : |S_1|=0}  x_1^{-c_{n-1}}x_2^{-c_{n-2}}x_1^{r|S_2|}z_3^{-t|S_2|}\\ &= x_1^{-c_{n-1}}\left(\frac{x_1^r z_3^{-t}+1}{x_2}\right)^{c_{n-2}},\endaligned$$
where the last identity holds because the condition $|S_1|=0$ means that  every subset $S_2$ of $\{1,2,\ldots,c_{n-2}\}$ is compatible with $S_1$.

Applying the map $$\left\{\begin{array}{l} x_1 \mapsto x_2 \\ x_2 \mapsto \frac{x_2^r  z_3^{-t} + 1}{x_1}  \end{array}\right.$$ yields
$x_2^{-c_{n-1}} x_1^{c_{n-2}}= x_1^{-c_{n}}x_2^{-c_{n-1}}  x_1^{rc_{n-1}} x_2^{0}$, which corresponds to a compatible pair $(S_1,S_2)$ in $\mathcal{D}^{c_{n}\times c_{n-1}}$ with $|S_1|=0$ and $|S_2|=c_{n-1}$, thus
 $(S_1,S_2)=(\emptyset,\mathcal{D}_2)$.

Finally, if  $(\bar S_1,\bar S_2)=(\emptyset, \mathcal{D}_2)$, then $|\bar S_1|=0$, $|\bar S_2|=c_{n-2}$, $m=sc_{n-1}-tc_{n-2}$ and $t$ must be positive  since $(\emptyset, \mathcal{D}_2)$ realizes the minimum.
Then $$\aligned \sum_{(S_1,S_2) : |S_1|=0}  x_1^{-c_{n-1}}x_2^{-c_{n-2}}x_1^{r|S_2|}x_2^{r|S_1|}z_3^{s(c_{n-1} - |S_1|)-t|S_2| -m} &=\sum_{(S_1,S_2) : |S_1|=0}   x_1^{-c_{n-1}}x_2^{-c_{n-2}}x_1^{r|S_2|}z_3^{tc_{n-2}-t|S_2|}\\ &= x_1^{-c_{n-1}}\left(\frac{x_1^r+z_3^t}{x_2}\right)^{c_{n-2}}\endaligned$$
and applying the map $$\left\{\begin{array}{l} x_1 \mapsto x_2 \\ x_2 \mapsto \frac{x_2^r + z_3^{t}}{x_1}  \end{array}\right.$$
yields
$x_2^{-c_{n-1}} x_1^{c_{n-2}}= x_1^{-c_{n}}x_2^{-c_{n-1}}  x_1^{rc_{n-1}} x_2^{0}$, which corresponds to a compatible pair $(S_1,S_2)$ in $\mathcal{D}^{c_{n}\times c_{n-1}}$ with $|S_1|=0$ and $|S_2|=c_{n-1}$, thus
 $(S_1,S_2)=(\emptyset,\mathcal{D}_2)$.
\end{proof}

\subsection{Rank 2 inside rank 3: Dyck path formula} Let $\mathcal{A}(Q)$ be a
 {\em non-acyclic }
coefficient-free cluster algebra of rank 3 with initial quiver $Q$ equal to
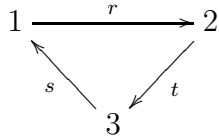
\[\xymatrix{1\ar[rr]^r&&2\ar[ld]^t\\&3\ar[lu]^s}\]
 where $r,s,t$ denote the number of arrows.
Now suppose that $r,s$ and $t$ are such that the cluster algebra $\mathcal{A}(Q)$ is non-acyclic. Then we can show that $m$ is always zero.
\begin{corollary}\label{12312011a}
If the cluster algebra is non-acyclic then there is a unique compatible pair $(S_1,S_2)$ in $\mathcal{D}^{c_{n-1}\times c_{n-2}}$ which achieves $s(c_{n-1}-|S_1|)-t|S_2|=m$. Moreover $(S_1,S_2)=(\mathcal{D}_1,\emptyset)$ and $m=0$.
\end{corollary}
\begin{proof}
We use induction on $n$.
For $n=3$ and $n=4$, computing $x_n$ directly by mutation yields
\[x_3= (x_2^r+z_3^s)/x_1,\qquad x_4=(x_3^r+z_3^{rs-t})/x_2 ,\]
and hence the term in the expression for $x_4$ that is not divisible by $z_3$ is $x_2^{r^2}x_1^{-r} x_2^{-1}$ which corresponds to a compatible pair $(S_1,S_2)$ in $\mathcal{D}^{c_{n-1}\times c_{n-2}}$  such that $|S_2|=c_{n-1}-|S_1|=0$, thus $(S_1,S_2)=(\mathcal{D}_1,\emptyset)$.

Now the result follows by the same argument as in the proof of Lemma \ref{12312011} using the fact that $t>0$.
\end{proof}

We have proved the following theorem.
\begin{theorem}\label{thm 33} For $n\ge 3$
\[x_n = x_1^{-c_{n-1}} x_2^{-c_{n-2}}\sum_{(S_1,S_2)}x_1^{r|S_2|}x_2^{r|S_1|} z_3^{s(c_{n-1}-|S_1|)-t|S_2|},  \]
where the sum is over all compatible pairs $(S_1,S_2)$ in $\mathcal{D}^{c_{n-1}\times c_{n-2}}$.
 \end{theorem}

\begin{remark}
 It follows from results of \cite{R} and \cite{LLZ} that an adapted version of this theorem still holds if the cluster algebra is not skew-symmetric.
\end{remark}


Our next goal is to describe cluster monomials of the form $x_{n+1}^px_n^q$ with $p,q\ge 0$. In order to simplify the notation we define
$A_i=p c_{i+1} +q c_{i}$. The following Lemma is a straightforward consequence of Lemma \ref{lem cn}.

\begin{lemma}\label{negone}
For any $i$, we have
\begin{itemize}
\item [(a)] $ A_i=rA_{i-1}-A_{i-2},$
\item [(b)] $A_{i}^2 - A_{i+1}A_{i-1}= p^2 + q^2 + rpq.$\qed
\end{itemize}
\end{lemma}

\begin{theorem}\label{mainthm12052011}
\[x_{n+1}^p x_n^q = x_1^{-A_{n-1}} x_2^{-A_{n-2}}\sum_{(S_1,S_2)}x_1^{r|S_2|}x_2^{r|S_1|} z_3^{s(c_{n-1}-|S_1|)-t|S_2|}, \]
where the sum is over all $(S_1=\cup_{i=1}^{p+q}\, S_{1}^i,S_2=\cup_{i=1}^{p+q}\, S_{2}^i)$ such that  
$$(S_{1}^i,S_{2}^i) \textup{ is a compatible pair in }\left\{\begin{array}{ll}\mathcal{D}^{c_{n-1}\times c_{n-2}}  &\textup{ if $1\le i \le q$}; \\  
 \mathcal{D}^{c_{n}\times c_{n-1}} &\textup{ if $q+1\le i\le p+q$.}
\end{array}\right.$$
\end{theorem}
\begin{proof} This follows immediately from Theorem \ref{thm 33}. 
\end{proof}
\begin{remark}
 It can be shown that the summation on the right hand side in Theorem \ref{mainthm12052011} can be taken over all compatible pairs in $\mathcal{D}^{A_{n-1}\times A_{n-2}}$ instead, without changing the sum, see \cite[Theorem 1.11]{LLZ}.
  \end{remark}

\subsection{Rank 2 inside rank 3: Mixed formula}
We keep the setup of the previous subsection. In particular, $\mathcal{A}(Q)$ is non-acyclic. We present another formula for the Laurent expansion of the cluster monomial $x_{n+1}^px_n^q$, which is parametrized by a certain sequence of integers $\tau_0,\tau_1,\ldots,\tau_{n-2}$. This formula is a generalization of a formula given in \cite[Theorem 2.1]{L}. Combining it with the formula of Theorem~\ref{mainthm12052011} yields the mixed formula of Theorem \ref{mainthm12062011} below, which is a key ingredient for the proof of the positivity conjecture in section \ref{sect 3}.

For arbitrary (possibly negative) integers $A, B$, we define the modified binomial coefficient as follows.
$$\left[\begin{array}{c}{A } \\{B} \end{array}\right] := \left\{ \begin{array}{ll}  \prod_{i=0}^{A-B-1} \frac{A-i}{A-B-i}, & \text{ if }A > B\\ \, & \,  \\   1, & \text{ if }A=B \\ \, & \, \\  0, & \text{ if }A<B.  \end{array}  \right.$$

If $A\geq 0$ then $\left[\begin{array}{c}{A } \\{B} \end{array}\right]=\gchoose{A}{A-B}$ is just the usual binomial coefficient. In particular  $\left[\begin{array}{c}{A } \\{B} \end{array}\right]= 0$ if $A\ge0$ and $B<0$.

For a sequence of integers $(\tau_j)$ (respectively $(\tau'_j)$), we define a sequence  of weighted partial sums $(s_i)$ (respectively $(s_i'))$ as follows:
$$ \begin{array}{rcccl}
s_0=0,&\quad &s_i=\sum_{j=0}^{i-1} c_{i-j+1} \tau_j = c_{i+1}\tau_0+c_i\tau_1+\cdots+c_2\tau_{i-1};\\ \\
s'_0=0,&\quad &s'_i=\sum_{j=0}^{i-1} c_{i-j+1} \tau_j' = c_{i+1}\tau'_0+c_i\tau'_1+\cdots+c_2\tau'_{i-1}.\end{array}
$$
For example, $s_1=c_2\tau_0=\tau_0$, $s_2= c_3\tau_0+c_2\tau_1=r\tau_0+\tau_1$.

\begin{lemma}
 \label{lem sn}
 $s_n=r s_{n-1} -s_{n-2} +\tau_{n-1}$.
\end{lemma}
\begin{proof} Since $c_{n-j+1}=rc_{n-j}-c_{n-j-1}$, we see that $s_n$ is equal to
 $$\begin{array}{rcl}\displaystyle\sum_{j=0}^{n-1} (rc_{n-j}-c_{n-j-1})\tau_j
 & =&\displaystyle
r\left( \sum_{j=0}^{n-2} c_{n-j}\tau_j \right) +rc_{1}\tau_{n-1}
-\left(\sum_{j=0}^{n-3} c_{n-j-1}\tau_j\right) - c_{1} \tau_{n-2} -c_0 \tau_{n-1}
\\ &=&
 rs_{n-1}-s_{n-2}+\tau_{n-1},\end{array}$$
 where the last identity holds because $c_1=0, c_0=-1$.
\end{proof}
\begin{definition}\label{20120121}
Let $\mathcal{L}( \tau_0,\tau_1,\cdots,\tau_{n-2} ) $
denote the set of all $( \tau'_0,\tau'_1,\cdots,\tau'_{n-2} )\in \mathbb{Z}^{n-1} $ satisfying the conditions
$$\begin{array}{ll} \textup{(1)} & 0\leq\tau'_i\leq \tau_i \text{ for }0\leq i\leq n-3, \\ \\
\textup{(2)}&   s'_{n-2}=k c_{n-1}\textup{ and }
     s'_{n-1} =k c_{n} \text{ for some integer }0\leq k\leq p.\end{array}
$$
%
\end{definition}

We define a partial order on $\mathcal{L}( \tau_0,\tau_1,\cdots,\tau_{n-2} )$ by
$$( \tau'_0,\tau'_1,\cdots,\tau'_{n-2} ) \leq_{\mathcal{L}} ( \tau''_0,\tau''_1,\cdots,\tau''_{n-2} )\text{ if and only if }\tau'_i\leq \tau''_i \text{ for }0\leq i\leq n-3.
$$
Then let $
\mathcal{L}_{\max}( \tau_0,\tau_1,\cdots,\tau_{n-2} )$ be the set of the maximal elements of $
\mathcal{L}( \tau_0,\tau_1,\cdots,\tau_{n-2} )$ with respect to $\leq_{\mathcal{L}}$.

\bigskip

We are ready to state the main result of this subsection.
\begin{theorem}\label{mainthm2} Let $n\geq 3$.  
  Then
\begin{equation}\label{mainformula2}\aligned
&x_{n+1}^p x_n^q= x_1^{-A_{n-1}} x_2^{-A_{n-2}} \sum_{ \tau_0,\tau_1,\cdots,\tau_{n-2}}\aligned &\left( \prod_{i=0}^{n-2} \left[\begin{array}{c}{{A_{i+1} - rs_i    } } \\{\tau_i} \end{array}\right]   \right)x_1^ {rs_{n-2}} x_2^{r(A_{n-1}-s_{n-1})}   z_3^{ss_{n-1}-ts_{n-2}},
\endaligned
\endaligned\end{equation}
where  the summation runs over all integers $ \tau_0,...,\tau_{n-2}$ satisfying
\begin{equation}\label{cond502}
\begin{array}{ll} \bullet \quad 0\leq \tau_i \leq A_{i+1} - rs_i \, (0\leq i\leq n-3),
\\
\bullet \quad \tau_{n-2} \leq A_{n-1} - rs_{n-2} , \text{ and } \\
\bullet \quad \left(s_{n-1}-s'_{n-1} \right) A_{n-2} \geq \left(s_{n-2} -s'_{n-2} \right) A_{n-1} \text{ for any }(\tau'_0,...,\tau'_{n-2})\in\mathcal{L}_{\max}( \tau_0,\cdots,\tau_{n-2} ).
 \end{array} \end{equation}

\end{theorem}
\begin{proof}The theorem is proved in section \ref{pfthm}
\end{proof}

\begin{example}
 Let $Q$ be the quiver $$\xymatrix{1\ar[rr]^2&&2\ar[ld]^2\\&3\ar[lu]^2}$$ and let $n=5, p=1, q=0$. Thus our formula computes the cluster variable obtained from the initial cluster by mutating in directions 1,2,1 and 2.

 First note that in this case $c_i=i-1$, $A_i=i$ and $s_i= i\tau_0+(i-1)\tau_1+\cdots +2\tau_{i-2}+\tau_{i-1}$. The first condition in (\ref{cond502}) is $0\le \tau_i\le i+1-2 s_i$. From this  we see that  $ \tau_0$ is either 0 or 1. If $\tau_0=1$, then $s_1=1$, hence $\tau_1=0$ by (\ref{cond502}), whence $s_2=2$ and $0\le \tau_2\le (2+1)-2(2)$, again by (\ref{cond502}), a contradiction. Thus $\tau_0=0$ and the conditions on $\tau_i$ in  (\ref{cond502}) become
  \[
\begin{array}{ccccccccc}
 \tau_0= 0& 0\le \tau_1 \le 2 & 0\le \tau_2\le 3-2\tau_1 & \tau_3\le 4-4\tau_1-2\tau_2 ,
\end{array}
 \]
 From this we conclude that there are the following 11 possibilities for $(\tau_0,\tau_1,\tau_2,\tau_3)$
 \[
\begin{array}{cccccccc}
(0,0,0,0)&(0,0,0,1)&(0,0,0,2)&(0,0,0,3)&(0,0,0,4)&(0,0,1,0) \\
&(0,0,1,1)&
(0,0,1,2)&(0,0,2,0)&(0,0,3,-2)&(0,1,0,0)
\end{array}
 \]
 Observe that each of these tuples satisfies the second condition in (\ref{cond502}). Indeed, the integer $k\in\{0,1\}$ in Definition \ref{20120121} must satisfy
 \[ 3k= 2\tau'_1+\tau'_2 \qquad \textup{and} \qquad 4k=3\tau'_1+2\tau'_2+\tau_3',\]
 so for example, if $(\tau_0,\tau_1,\tau_2,\tau_3)=(0,0,1,1)$ then
 $\tau_0'=\tau_1'=0$, $0\le \tau_2' \le 1$, $\tau_3'\le 1$ and
 \[3k=\tau_2'\le 1 \qquad \textup{and} \qquad 4k=2\tau'_2+\tau_3'.\]
 Thus $k=0$,  $
 \mathcal{L}(0,0,1,1)=\{(0,0,0,0)\}$, and
  the second condition in (\ref{cond502}) becomes
  \[(s_4-s_4')A_3\ge (s_3-s_3')A_4 \quad \Leftrightarrow  \quad(2+1 -0)3 \ge (1-0)4  \quad\Leftrightarrow  \quad9\ge 4.\]
On the other hand, the eleven 4-tuples above are the only ones that satisfy all conditions in (\ref{cond502}). For example, for the tuple $(0,1,1,-2)$, we get $k=1$, $\tau'=(0,1,1,-1)\in\mathcal{L}_{\textup{max}}(0,1,1,-2)$ and  the condition
\[(s_4-s_4')A_3\ge (s_3-s_3')A_4 \quad \Leftrightarrow  \quad(3+2-2-3-2+1)3 \ge (2+1-2-1)4  \quad\Leftrightarrow  \quad -3 \ge 0\]
is not satisfied.
Therefore  Theorem \ref{mainthm2} yields
\[
\begin{array}{ccccccccccl}
x_6= (x_2^8
&+& 4 x_2^6z_3^2
&+ &6 x_2^4z_3^4
&+ &4 x_2^2z_3^6
&+  &z_3^8
&+& 3 x_1^2 x_2^4 z_3^2\\
&+& 6x_1^2 x_2^2 z_3^4
&+& 3x_1^2  z_3^6
&+& 3 x_1^4 z_3^ 4
&+& x_1^6  z_3^2
&+& 2x_1^4 x_2^2 z_3^ 2    \,    ) / x_1^4x_2^3.
\end{array}
\]
\end{example}

\begin{remark}\label{20120121a} When comparing the formula of Theorem \ref{mainthm2} with the Dyck path formula of Theorem \ref{mainthm12052011}, we have the following interpretation for the integer $k$ in Definition \ref{20120121}.
 Let $\mathcal{D}_2$ be the set of all vertical edges in $\mathcal{D}^{c_n\times c_{n-1}}$, and fix a  pair $(S_1,S_2)$ as in Theorem \ref{mainthm12052011}. Then $k$ in Definition 3.15 is equal to the number of times $\mathcal{D}_2$ appears in $S_2$.
Moreover,
 if $(\tau_0',\tau_1',\ldots,\tau_{n-2}')\in \mathcal{L}_{\textup max}$ is such that $ \prod_{i=0}^{n-2} \left[\begin{array}{c}{{A_{i+1} - rs_i    } } \\{\tau_i} \end{array}\right]   \ne 0$ then $$k=\textup{min}\left(\left\lfloor \frac{s_{n-2}}{c_{n-1}}\right\rfloor, p \right).$$
\end{remark}

\begin{corollary}\label{cor 26}
 \noindent Let $x_3= (x_2^r+z_3^s)/x_1 $ and let $t'=s$ and $s'=rs-t$ be the number of arrows from 1 to 3, or 3 to 2 respectively, in the quiver obtained from $Q$ by mutating in the vertex 1. Then
 \begin{equation}\label{mainformula2b0303}\aligned
&x_{n+1}^p x_n^q= x_2^{-A_{n-2}} x_3^{-A_{n-3}} \sum_{ \tau_0,\tau_1,\cdots,\tau_{n-3}}\aligned &\left( \prod_{i=0}^{n-3} \left[\begin{array}{c}{{A_{i+1} - rs_i    } } \\{\tau_i} \end{array}\right]   \right)x_2^ {rs_{n-3}} x_3^{r(A_{n-2}-s_{n-2})}   z_3^{s's_{n-2}-t's_{n-3}}, \endaligned
\endaligned\end{equation}
where  the summation runs over all integers $ \tau_0,...,\tau_{n-3}$ satisfying
\begin{equation}\label{cond502b0303}\left\{
\begin{array}{l} 0\leq \tau_i \leq A_{i+1} - rs_i \, (0\leq i\leq n-4),  \ \tau_{n-3} \leq A_{n-2} - rs_{n-3} , \text{ and } \\
\left(s_{n-2}-s'_{n-2}\right) A_{n-3} \geq \left(s_{n-3} -s_{n-3}'\right) A_{n-2} \text{ for any }(\tau'_0,...,\tau'_{n-3})\in\mathcal{L}_{\max}( \tau_0,\cdots,\tau_{n-3} ).
 \end{array} \right.\end{equation}

\end{corollary}

\begin{proof}
 This follows directly from Theorem \ref{mainthm2}. \end{proof}

Combining
Theorem \ref{mainthm2} with Theorem~\ref{mainthm12052011}  we get the following mixed formula.
\begin{theorem}\label{mainthm12062011} Let $n\geq 3$.   Then
\begin{equation}\label{mainformula22}\aligned
x_{n+1}^p x_n^q= & \sum_{\begin{array}{c} \scriptstyle\tau_0,\tau_1,\cdots,\tau_{n-3} \\ \scriptstyle s_{n-2}\leq  A_{n-1}/r \end{array}  }\left( \prod_{i=0}^{n-3}{{A_{i+1} - rs_i    }  \choose{\tau_i} }   \right){x_3}^ {A_{n-1}-rs_{n-2}} x_2^{rs_{n-3}-A_{n-2}}   z_3^{s's_{n-2}-t's_{n-3}}
\\
&+ \sum_{\begin{array}{c} \scriptstyle (S_1,S_2)\\ \scriptstyle r|S_2|-A_{n-1}>0 \end{array}}x_1^{r|S_2|-A_{n-1}}x_2^{r|S_1|-A_{n-2}} z_3^{s(c_{n-1}-|S_1|)-t|S_2|}.
\endaligned\end{equation}
\end{theorem}

\begin{remark} The exponents of $x_3, x_1$ and $z_3$ are nonnegative, which is important for the proof of Theorem~\ref{mainthm11102011_00}.  The modified binomial coefficients can be replaced by the usual binomial coefficients, because the condition $s_{n-2}\le A_{n-1}/r$ implies that $A_{n-2}-rs_{n-3}$ is nonnegative, see Lemma \ref{lem 34} below.
\end{remark}

\begin{proof} This first sum of the statement is obtained from
Corollary \ref{cor 26} using Lemma \ref{negone}(a) for the exponent of $x_3$.  Observe that the new condition $s_{n-2}\le A_{n-1}/r$ in the summation precisely means  that the exponent of $x_3$  is   nonnegative. On the other hand, the  sum of all terms  in which the exponent of $x_3$ is negative in the expression in Corollary \ref{cor 26} is equal to the second sum in the statement of Theorem \ref{mainformula22}. This follows from  the formula  of Theorem~\ref{mainthm12052011}.
\end{proof}

In \cite{BFZ} the upper cluster algebra was defined as the intersection of the rings of Laurent polynomials in the $n+1$ clusters consisting of the initial cluster and all clusters obtained from it by one single mutation.
The following corollary gives a different ``upper bound"   for the cluster monomials in the rank 2 direction. This new upper bound is defined as the semi-ring of polynomials in the variables in the initial cluster, the first mutation of one initial variable and the inverse of another initial cluster variable.
\begin{corollary} Let $\widetilde{x}_1$ denote the cluster variable obtained from the initial seed by mutating in $x_1$. Then
 $$x_{n+1}^p x_n^q\in \mathbb{Z}_{\ge 0}[x_1,\widetilde{x}_1, z_3, x_2^{\pm 1}].$$
\end{corollary}
\begin{proof}
 This follows from Theorem \ref{mainthm12062011} because $\widetilde{x}_1=x_3.$
\end{proof}

\subsection{Proof of Theorem~\ref{mainthm2}}\label{pfthm}
We use induction on $n$. Suppose first that $n=3$. Since $x_3=(x_2^r+z_3^s)/x_1$, we have
\[ x_3^a = x_1^{-a}\sum_{\tau_1=0}^a {a    \choose \tau_1  }  x_2^{r(a-\tau_1)}z_3^{s\tau_1},\]
and since $x_4=(x_3^r+z_3^{rs-t})/x_2$, we have
\[ \begin{array}{rcl}\displaystyle x_4^p &=&\displaystyle x_2^{-p}\sum_{\tau_0=0}^p {p    \choose \tau_0  }  x_3^{r(p-\tau_0)}z_3^{(rs-t)\tau_0}.
\end{array}
\]
Therefore
\[\begin{array}{rcl}
x_4^px_3^q &=&\displaystyle x_2^{-p}\sum_{\tau_0=0}^p {p    \choose \tau_0  }  x_3^{r(p-\tau_0)+q}z_3^{(rs-t)\tau_0}\\
&=&\displaystyle x_2^{-p}\sum_{\tau_0=0}^p {p    \choose \tau_0 } \sum_{\tau_1=0}^{r(p-\tau_0)+q} {r(p-\tau_0)+q\choose \tau_1} x_1^{-r(p-\tau_0)-q} x_2^{r(r({p-\tau_0})+q-\tau_1)} z_3^{s\tau_1}z_3^{(rs-t)\tau_0} \\
&=&\displaystyle x_1^{-rp-q} x_2^{-p}\sum_{\tau_0=0}^p  \sum_{\tau_1=0}^{r(p-\tau_0)+q}  {p    \choose \tau_0  }{r(p-\tau_0)+q\choose \tau_1}
x_1^{r\tau_0} x_2^{r(r({p-\tau_0})+q-\tau_1)} z_3^{s\tau_1+(rs-t)\tau_0} ,
\end{array}
\]
and the statement follows from $A_1= p$, $A_{2}=rp+q$,  $s_0=0$, $s_1=\tau_0$, $ s_2= r\tau_0+\tau_1$.

Suppose now that $n\geq 4$, and assume
 that the statement holds for $n$ or less. Then by the obvious shift, we have
\Small{ $$
x_{n+2}^p x_{n+1}^q= x_2^{-A_{n-1}} x_3^{-A_{n-2}}  \sum_{ \tau_0,\tau_1,\cdots,\tau_{n-2}} \left[ \left( \prod_{i=0}^{n-2} \left[\begin{array}{c}{{A_{i+1} - rs_i    } } \\{\tau_i} \end{array}\right]   \right)  x_2^ {rs_{n-2}} x_3^{r(A_{n-1} - s_{n-1})}z_3^T \right],
 $$}
 \normalsize{where the summation runs over all integers $ \tau_0,...,\tau_{n-2}$ satisfying (\ref{cond502}) and }
 $$\aligned T= {s's_{n-1}-t's_{n-2}}= (rs-t) s_{n-1} -s s_{n-2} \endaligned$$

 Using Lemma \ref{negone}(a), we see that the exponent of $x_3 $ is equal to $A_n-rs_{n-1}$. Then substituting
 $\frac{x_2^r+z_3^s}{x_1}$ into $x_3$, we get
$$ x_{n+2}^px_{n+1}^q =
 x_2^{-A_{n-1}}  \sum_{ \tau_0,\tau_1,\cdots,\tau_{n-2}} \left[ \left( \prod_{i=0}^{n-2} \left[\begin{array}{c}{{A_{i+1} - rs_i    } } \\{\tau_i} \end{array}\right]   \right)  x_2^ {rs_{n-2}} \left(\frac{x_2^r+z^{s}}{x_1}\right)^{A_{n} - rs_{n-1}}z_3^T\right].
$$
Expanding $(x_2^r+z^{s})^{A_{n} - rs_{n-1}}$ yields
$$
 x_1^{-A_n} x_2^{-A_{n-1}}  \sum_{ \tau_0,\cdots,\tau_{n-2}} \!\! \left( \prod_{i=0}^{n-2} \left[\begin{array}{c}{{\!\!\!\scriptstyle A_{i+1} - rs_i \! \!\!  } } \\{\scriptstyle\tau_i} \end{array}\right]   \right)x_2^ {rs_{n-2}} \!\! \sum_{\tau_{n-1}\in \mathbb{Z}}    \left[\begin{array}{c}{\!\!\!\scriptstyle{A_{n} - r s_{n-1} \!\!\!\!  } } \\{\scriptstyle\tau_{n-1}} \end{array}\right]
 x_1^{rs_{n-1} }  (x_2^r)^{A_{n} - rs_{n-1}-\tau_{n-1}}  z_3^{T+s\tau_{n-1}} .
$$
Note that ${T+s\tau_{n-1}}={s s_{n} - t s_{n-1}}$, by Lemma~\ref{lem sn}.
Combining the sums, we get
$$ x_1^{-A_n}  x_2^{-A_{n-1}}  \sum_{ \tau_0,\tau_1,\cdots,\tau_{n-2};\tau_{n-1}\in\mathbb{Z}}  \left( \prod_{i=0}^{n-1} \left[\begin{array}{c}{{A_{i+1} - rs_i    } } \\{\tau_i} \end{array}\right]   \right)x_1^{rs_{n-1}}x_2^ {rs_{n-2}}    (x_2^r)^{A_{n} - rs_{n-1}-\tau_{n-1}}   z_3^{s s_{n} - t s_{n-1}}
$$
and, by Lemma \ref{lem sn}, this is equal to
\begin{equation}
\label{3.4.1}
 x_1^{-A_n} x_2^{-A_{n-1}} \sum_{ \tau_0,\tau_1,\cdots,\tau_{n-2};\tau_{n-1}\in\mathbb{Z}}  \left( \prod_{i=0}^{n-1} \left[\begin{array}{c}{{A_{i+1} - rs_i    } } \\{\tau_i} \end{array}\right]   \right) x_1^{rs_{n-1}} x_2^{r(A_{n} - s_{n})}z_3^{s s_{n} - t s_{n-1}} .
\end{equation}


Remember that $\tau_0,...,\tau_{n-2}$  satisfy (\ref{cond502}). 

Proposition~\ref{eanegineqcor} below implies that even if we impose the additional condition on $\tau_{n-2}$ and $\tau_{n-1}$  that
\begin{equation}\label{extra condition}
 \left(s_{n}-s_n'\right) A_{n-1} - \left(s_{n-1} -s_{n-1} '\right) A_{n}  \geq 0,\textup{ for all $(\tau'_0,...,\tau'_{n-1}) \in\mathcal{L}_{\max}( \tau_0,\cdots,\tau_{n-2} )$,}
\end{equation}
 the value of the expression for $x_{n+2}^px_{n+1}^q$ in (\ref{3.4.1}) does not change. From now on we impose the condition (\ref{extra condition}). On the other hand, in order to prove that Theorem~\ref{mainthm2} holds for $n+1$, we only need to show that $\tau_{n-2}$ is enough to run over $0\leq \tau_{n-2}\le  A_{n-1}-rs_{n-2}$. So we want to show that, for a fixed sequence $(\tau'_0,...,\tau'_{n-2},\tau'_{n-1})$,
\begin{equation}\label{zeroeq701}
 \sum_{ \tau_0,\tau_1,\cdots,\tau_{n-1}} \left[ \left( \prod_{i=0}^{n-1} \left[\begin{array}{c}{{A_{i+1} - rs_i}} \\{\tau_i} \end{array}\right]   \right) {x_1}^{rs_{n-1}} x_2^{r(A_{n} - s_{n})}    \right] =0,
\end{equation}
where the summation runs over all integers $ \tau_0,...,\tau_{n-1}$ satisfying
\begin{equation}\label{cond503}\left\{
\begin{array}{ll}
(a) & (\tau'_0,...,\tau'_{n-2})\in\mathcal{L}_{\max}( \tau_0,\cdots,\tau_{n-2} ), \\
&\\
(b) & 0\leq \tau_i \leq A_{i+1} - rs_i \, (0\leq i\leq n-3),\\
&\\
(c) & \left(s_{n-1}-s_{n-1}' \right) A_{n-2} - \left(s_{n-2} -s_{n-2}'\right) A_{n-1}\geq 0,  \\
&\\
(d) & (\tau'_0,...,\tau'_{n-1})\in\mathcal{L}_{\max}( \tau_0,\cdots,\tau_{n-1} ), \\
&\\
(e) & \tau_{n-2}\leq A_{n-1} - rs_{n-2}<0,  \text{ and } \\
&\\
(f) & \left(s_{n}-s_{n}' \right) A_{n-1} - \left(s_{n-1} -s_{n-1}'\right) A_{n} \geq 0.
 \end{array} \right.\end{equation}


To do so, it is sufficient to show that
$ \left[\begin{array}{c}{{A_{n} - rs_{n-1}}} \\{\tau_{n-1}} \end{array}\right]  =0$ for every $\tau_{n-1}$, because then each summand in equation (\ref{zeroeq701}) is zero.
This purely algebraic result is proved in Lemma~\ref{lem 34} below.

Assuming Lemma \ref{lem 34} and Proposition~\ref{eanegineqcor}, we have proved that  $$x_{n+2}^p x_{n+1}^q=x_1^{-A_n} x_2^{-A_{n-1}} \sum_{ \tau_0,\tau_1,\cdots,\tau_{n-1}} \left[ \left( \prod_{i=0}^{n-1} \left[\begin{array}{c}{{A_{i+1} - rs_i    } } \\{\tau_i} \end{array}\right]   \right) {x_1}^{rs_{n-1}} x_2^{r(A_{n} - s_{n})} z_3^{s s_{n} - t s_{n-1}}   \right],$$where the summation runs over all integers $ \tau_0,...,\tau_{n-1}$ satisfying\begin{equation}\label{cond504}\left\{
\begin{array}{l} 0\leq \tau_i \leq A_{i+1} - rs_i \, (0\leq i\leq n-2),\\ \\
(s_{n-1}-s'_{n-1}) A_{n-2} -(s_{n-2} -s'_{n-2}) A_{n-1}  \geq 0,  \text{ and }\\ \\
(s_{n}-s'_{n}) A_{n-1} -(s_{n-1} -s'_{n-1}) A_{n}  \geq 0,
 \end{array} \right.\end{equation}
for all $(\tau'_0,...,\tau'_{n-1}) \in\mathcal{L}_{\max}( \tau_0,\cdots,\tau_{n-2} )$.
Therefore it only remains to show that we do not need to require the second condition in (\ref{cond504}).
Using Lemma \ref{lem sn} and Lemma \ref{negone} (a), we see that
$$\aligned &s_{n-1} A_{n-2} -s_{n-2} A_{n-1} \stackrel{(\ref{lem sn})}{=} (rs_{n-2}-s_{n-3}+\tau_{n-2}) A_{n-2} -s_{n-2} A_{n-1}\\
& \stackrel{(\ref{negone})}{=}  (s_{n-2}A_{n-3}-s_{n-3}A_{n-2}) + \tau_{n-2}A_{n-2}.\endaligned$$
Iterating this argument, we get
$$
s_{n-1} A_{n-2} -s_{n-2} A_{n-1}=(s_{2}A_{1}-s_{1}A_{2}) +\sum_{i=2}^{n-2} \tau_i A_i ,$$
which is equal to
 $$\tau_1p-\tau_0 q+\sum_{i=1}^{n-2} \tau_i A_i\geq \tau_1p-\tau_0 q,$$
 because $s_1=\tau_0$, $s_2=r\tau_0+\tau_1$, $A_1=p$ and $A_2=rp+q$.
 Thus
 \begin{equation}
\label{eq 33}
s_{n-1} A_{n-2} -s_{n-2} A_{n-1}\geq \tau_1p-\tau_0 q.
\end{equation}
 Our next goal is to estimate $-s'_{n-1} A_{n-2} +s'_{n-2} A_{n-1} $.
  Let $k$ be as in Definition \ref{20120121}, so that $0\le k\le p$ and  $s'_{n-1}=kc_{n}$ and $s'_{n-2}=kc_{n-1}$.
Then
 \begin{eqnarray}\label{equation 262626}
-s'_{n-1} A_{n-2} +s'_{n-2} A_{n-1} &=&k(-c_n A_{n-2} +c_{n-1}A_{n-1}) \\
&=&k(-pc_nc_{n-1}-qc_nc_{n-2}+pc_nc_{n-1}+qc_{n-1}^2)\nonumber
\\
&=& kq(-c_nc_{n-2}+c_{n-1}^2)\nonumber
\\
&=&kq,\nonumber
\end{eqnarray}
where the second equality follows from the definition $A_i=pc_{i+1}+qc_i$, and the last equality holds by Lemma  \ref{lem cn}.
On the other hand, $s'_{n-2}$ is defined as $s'_{n-2}=c_{n-1}\tau_0' + \sum_{j=1}^{n-3} c_{n-1-j} \tau'_j$, which implies that $kc_{n-1}=s'_{n-2} \ge c_{n-1}\tau_0'$, and thus $k\ge\tau_0'$. Moreover, $\tau_0=\tau_0'$ by  definition of $
\mathcal{L}_{\max}( \tau_0,\tau_1,\cdots,\tau_{n-2} )$, and thus equation (\ref{equation 262626}) implies
$$
-s'_{n-1} A_{n-2} +s'_{n-2} A_{n-1} \geq \tau_0 q.
$$
  Adding the last inequality to inequality (\ref{eq 33}) we get
$$(s_{n-1}-s'_{n-1}) A_{n-2} -(s_{n-2} -s'_{n-2}) A_{n-1} \geq 0,$$
hence the second condition in  (\ref{cond504}) is always satisfied.

This completes the proof  of Theorem \ref{mainthm2} modulo Lemma \ref{lem 34} and Proposition~\ref{eanegineqcor}.

\begin{lemma} \label{lem 34} Assume conditions (\ref{cond503}). Then
$$ \left[\begin{array}{c}{{A_{n} - rs_{n-1}}} \\{\tau_{n-1}} \end{array}\right]  =0.$$
\end{lemma}
\begin{proof}
  By definition of the modified binomial coefficient, it  suffices to show that $A_n-rs_{n-1}<\tau_{n-1}.$

On the one hand, we have
\begin{equation}\label{eq070301}\aligned
&A_{n-3}(s_{n-2} -s'_{n-2}) - A_{n-2}(A_{n-1}-s_{n-1}+s'_{n-1}) \\
&=rA_{n-2}(s_{n-2}-s'_{n-2}) - A_{n-1}(s_{n-2}-s'_{n-2})- A_{n-2}(A_{n-1}-s_{n-1}+s'_{n-1})\\
&\quad(\text{by Lemma~\ref{negone}(a)})\\
&=A_{n-2}rs_{n-2}-rA_{n-2}s'_{n-2} - A_{n-1}(s_{n-2}-s'_{n-2})- A_{n-2}(A_{n-1}-s_{n-1}+s'_{n-1})\\
&>A_{n-2}A_{n-1}-r A_{n-2}s'_{n-2}- A_{n-1}(s_{n-2}-s'_{n-2}) - A_{n-2}(A_{n-1}-s_{n-1}+s'_{n-1})\\
& \quad(\text{since }A_{n-1}-rs_{n-2}<0 )\\
&=-r A_{n-2}s'_{n-2}+(s_{n-1}-s'_{n-1}) A_{n-2} -(s_{n-2} -s'_{n-2}) A_{n-1}\\
& \geq -r A_{n-2}s'_{n-2} \quad ({\text{by }(\ref{cond503})(c)}).
\endaligned\end{equation}

Thus
\begin{equation}\label{eq 3.15}(s_{n-2} -s'_{n-2})> (A_{n-2}(A_{n-1}-s_{n-1}+s'_{n-1}) -r A_{n-2}s'_{n-2} ) /A_{n-3}\end{equation}
Then \begin{equation}\label{eq070302}\aligned
& A_{n-2}(s_{n-1}-s'_{n-1})  - A_{n-1} (s_{n-2} -s'_{n-2})  \\
&\underset{\text{by }(\ref{eq 3.15})}< A_{n-2}(s_{n-1}-s'_{n-1}) - A_{n-1}\frac{A_{n-2}(A_{n-1}-s_{n-1}+s'_{n-1}) - r A_{n-2}s'_{n-2}}{A_{n-3}} \\
&=A_{n-2} \left(  s_{n-1}-s'_{n-1} -  \frac{A_{n-1}}{A_{n-3}}(A_{n-1}-s_{n-1}+s'_{n-1})\right)+\frac{rA_{n-1}A_{n-2}s'_{n-2}}{A_{n-3}} \\
&=A_{n-2}\left( A_{n-1} - \left(1+\frac{A_{n-1}}{A_{n-3}}\right)(A_{n-1}-s_{n-1}+s'_{n-1})   \right) + \frac{r A_{n-1}A_{n-2}s'_{n-2}}{A_{n-3}}\\
&=A_{n-2}\left( A_{n-1} - \left(1+\frac{A_{n-1}}{A_{n-3}}\right)\frac{A_{n-2}+A_n-rs_{n-1}+r s'_{n-1}}{r}   \right)+ \frac{r A_{n-1}A_{n-2}s'_{n-2}}{A_{n-3}}\\&\quad(\textup{by Lemma \ref{negone} (a)})
\endaligned\end{equation}
Now, aiming for contradiction, suppose that $A_n-rs_{n-1} \ge 0$. Then
 \begin{equation}\label{eq070302b}\aligned(\ref{eq070302})&\leq A_{n-2}\left( A_{n-1} - \frac{A_{n-3}+A_{n-1}}{A_{n-3}}\frac{A_{n-2}+r s'_{n-1}}{r}   \right)+ \frac{r A_{n-1}A_{n-2}s'_{n-2}}{A_{n-3}} \\
&= A_{n-2}\left( A_{n-1} - \frac{rA_{n-2}}{A_{n-3}}\frac{A_{n-2}+r s'_{n-1}}{r}   \right)+ \frac{r A_{n-1}A_{n-2}s'_{n-2}}{A_{n-3}} \quad(\textup{by Lemma \ref{negone} (a)})\\
&= A_{n-2}\left( A_{n-1} - \frac{A_{n-2}^2}{A_{n-3}} \right) - r\frac{A_{n-2}^2}{A_{n-3}}s'_{n-1}+ \frac{r A_{n-1}A_{n-2}s'_{n-2}}{A_{n-3}}\\
&= \frac{A_{n-2}}{A_{n-3}} \left(A_{n-1}A_{n-3} - A_{n-2}^2\right)- r\frac{A_{n-2}}{A_{n-3}}\left(A_{n-2}s'_{n-1}- { A_{n-1} s'_{n-2}}\right) .
\endaligned\end{equation}
Using Lemma~\ref{negone}(b) and the definition of $A_i$, this is equal to
\[
 \frac{A_{n-2}}{A_{n-3}} \left(-p^2-q^2-rpq\right)- r\frac{A_{n-2}}{A_{n-3}}\left((pc_{n-1}+qc_{n-2})s'_{n-1} - { (pc_{n}+qc_{n-1}) s'_{n-2}}\right).\]
 Let $k$ be as in Definition \ref{20120121}, then $kc_{n}=s'_{n-1}$ and $kc_{n-1}=s'_{n-2}$, and using Lemma \ref{lem cn}, we get
\[  \frac{A_{n-2}}{A_{n-3}} \left(-p^2-q^2-rpq\right)+ rk q\frac{A_{n-2}}{A_{n-3}},\]
and, since $k\le p$, this is less than or equal to
\[ \frac{A_{n-2}}{A_{n-3}} \left(-p^2-q^2-rpq\right)+rpq\frac{A_{n-2}}{A_{n-3}}
\leq 0,\]
which contradicts $(s_{n-1}-s'_{n-1}) A_{n-2} -(s_{n-2} -s'_{n-2}) A_{n-1} \geq 0$. Hence \begin{equation}\label{ancsnminus1leq0}A_n - rs_{n-1} <0.\end{equation}

Next we show that $s_{n-2}>A_n-s_n$. Suppose to the contrary that $s_{n-2}\leq A_n-s_n$. Then it follows from condition (\ref{cond503})(f) that
$$\aligned &A_{n-1}-rs_{n-2}  \geq A_{n-1}-r(A_n-s_n) \underset{(\ref{cond503})(f) }{\geq} A_{n-1}-r \frac{(A_{n-1}-s_{n-1}+s'_{n-1}) A_n}{A_{n-1}} +rs'_n \\
&\underset{\text{by Lemma }\ref{negone}(b)}=\frac{\left(p^2+q^2+rpq\right)}{A_{n-1}}+ \frac{A_n A_{n-2}  }{A_{n-1}}-r \frac{(A_{n-1}-s_{n-1}+s'_{n-1}) A_n}{A_{n-1}} +rs'_{n}\\
&\underset{\text{by Lemma }\ref{negone}(a)}=\frac{\left(p^2+q^2+rpq\right)}{A_{n-1}}+\frac{A_n}{A_{n-1}}\left(rs_{n-1}-A_n-rs'_{n-1}  + \frac{A_{n-1}}{A_{n}} rs'_{n}\right) \\
&=\frac{p^2+q^2}{A_{n-1}}+\frac{A_n}{A_{n-1}}\left(rs_{n-1}-A_n\right) +\frac{r}{A_{n-1}}(pq-A_n s'_{n-1}+A_{n-1}s_n')\\
&=\frac{p^2+q^2}{A_{n-1}}+\frac{A_n}{A_{n-1}}\left(rs_{n-1}-A_n\right) +\frac{r}{A_{n-1}}(pq-A_n kc_n+A_{n-1}kc_{n+1}) \quad (\textup{by Def \ref{20120121}}).
\endaligned$$\\
Now
$ kc_nA_{n} - kc_{n+1}A_{n-1}=kc_n(pc_{n+1}+qc_{n}) - kc_{n+1}(pc_n+qc_{n-1})=kq(c_n^2-c_{n+1}c_{n-1})=kq$,
where the last equation holds because of Lemma \ref{lem cn}.
Thus

$$\aligned &A_{n-1}-rs_{n-2} \geq \frac{p^2+q^2}{A_{n-1}}
+\frac{A_n}{A_{n-1}}\left(rs_{n-1}-A_n\right)
+\frac{r}{A_{n-1}}(q(p-k))
\\
& \geq \frac{p^2+q^2}{A_{n-1}}
+\frac{A_n}{A_{n-1}}\left(rs_{n-1}-A_n\right) \quad(\textup{by Def \ref{20120121}} \textup{ and } (\ref{cond503})(d) )\\
&\underset{\text{by }(\ref{ancsnminus1leq0})}>0,
\endaligned$$
which contradicts $A_{n-1}-rs_{n-2}<0$ in (\ref{cond503}). Thus $s_{n-2}>A_n-s_n$, so we have $$A_n-rs_{n-1}< s_n+s_{n-2}- rs_{n-1}\underset{}=\tau_{n-1},$$which gives $\gchoose{A_n-rs_{n-1}}{\tau_{n-1}}=0.$
\end{proof}

\begin{proposition}\label{eanegineqcor}
Let $a$ and $b$ be any two nonnegative integers satisfying $$ \sum_{ \begin{array}{c}\scriptstyle\tau_0,\tau_1,\cdots,\tau_{n-2}\\\scriptstyle s_{n-1}=a, s_{n-2}=b \end{array}}  \prod_{i=0}^{n-2} \left[\begin{array}{c}{{A_{i+1} - rs_i    } } \\{\tau_i} \end{array}\right]   \neq 0.$$
 Let $\tau_0,\cdots,\tau_{n-2}$ satisfy  $s_{n-1}=a, s_{n-2}=b$.
For any $(\tau'_0,...,\tau'_{n-2})\in\mathcal{L}_{\max}( \tau_0,\cdots,\tau_{n-2} ),$ we have
\begin{equation}\label{12312011eq1}\left(s_{n-1}-s_{n-1}'\right) A_{n-2} - \left(s_{n-2} -s_{n-2}'\right) A_{n-1}\geq 0.\end{equation}
\end{proposition}
\begin{proof} We use Theorem \ref{mainthm12052011}. 
Let $\mathcal{D}_{i,1}$ (respectively $\mathcal{D}_{i,2}$) be the set of horizontal (resp. vertical) edges in the $i$-th Dyck path in $(\mathcal{D}^{c_{n}\times c_{n-1}})^p\times (\mathcal{D}^{c_{n-1}\times c_{n-2}})^q$.

Choose a collection of compatible pairs $$\beta=(\beta_1, ..., \beta_{p+q})=((S_{1,1},S_{1,2}), ..., (S_{p+q,1},S_{p+q,2}))\text{ in }(\mathcal{D}^{c_{n}\times c_{n-1}})^p\times (\mathcal{D}^{c_{n-1}\times c_{n-2}})^q$$ such that
$\sum_{i=1}^{p+q} |S_{i,2}|=s_{n-2}$ and $pc_n+qc_{n-1}-\sum_{i=1}^{p+q} |S_{i,1}|=s_{n-1}$, and such that $\beta$ has the maximal number, say $w$, of copies of $(\emptyset, \mathcal{D}_2)$ in $\mathcal{D}^{c_{n}\times c_{n-1}}$. Say $\beta_{i_1}=\cdots=\beta_{i_w}=(\emptyset, \mathcal{D}_2)$ for some $1\leq i_1<\cdots<i_w \leq p$.
Since  $|\mathcal{D}_2|=c_{n-1}$ we see that
\[w=\textup{min} \left(\left\lfloor \frac{s_{n-2}}{c_{n-1}}\right\rfloor,p\right).\]
On the other hand, by Definition \ref{20120121} and Remark \ref{20120121a}, $s'_{n-2}=kc_{n-1}$ and $s'_{n-1}=kc_{n}$
with $k=w$ because $(\tau_0',\tau_1',\ldots,\tau_{n-2}')\in \mathcal{L}_{\textup{max}}$.
Therefore
$$\sum_{i\in \{1,2,...,p+q\} \setminus \{i_1,...,i_w\}} |\beta_i|_2 = s_{n-1}-wc_n = s_{n-1}-s'_{n-1},$$
$$\sum_{i\in \{1,2,...,p+q\} \setminus \{i_1,...,i_w\}} |\beta_i|_1 = s_{n-2}-wc_{n-1} = s_{n-2}-s'_{n-2},$$
where $|\beta_i|_2$ denotes $|\mathcal{D}_{i,1}|-|S_{i,1}|$ and $|\beta_i|_1$ denotes $|S_{i,2}|$.
Hence
 (\ref{12312011eq1}) is equivalent to
\begin{equation}\label{eq 3.20}
\sum_{i\in \{1,2,...,p+q\} \setminus \{i_1,...,i_w\}} A_{n-2}|\beta_i|_2-A_{n-1}|\beta_i|_1 \geq 0.
\end{equation}

First we show that $$\sum_{i=p+1}^{p+q} A_{n-2}|\beta_i|_2-A_{n-1}|\beta_i|_1 \geq 0.
 $$
 Due to Lemma~\ref{12312011}, if $p<i\leq p+q$ then either $\beta_i=(\mathcal{D}_{i,1}, \emptyset)$ or $\beta_i=(\emptyset, \mathcal{D}_{i,2})$ gives the minimum of $A_{n-2}|\beta_i|_2-A_{n-1}|\beta_i|_1$. If $\beta_i=(\mathcal{D}_{i,1}, \emptyset)$ then clearly $A_{n-2}|\beta_i|_2-A_{n-1}|\beta_i|_1=0$. If  $\beta_i=(\emptyset, \mathcal{D}_{i,2})$ then
 $$\aligned
 A_{n-2}|\beta_i|_2-A_{n-1}|\beta_i|_1 &= A_{n-2}c_{n-1}-A_{n-1}c_{n-2} \\
                                                             &= \left(p c_{n-1} +q c_{n-2}\right)c_{n-1}-\left(p c_{n} +q c_{n-1}\right)c_{n-2}\\
                                                              &=p \left(c_{n-1}^2-c_n c_{n-2}\right)\\
                                                              &=p\geq 0.
 \endaligned$$

 Next we show that
  $$\sum_{i\in \{1,2,...,p\} \setminus \{i_1,...,i_w\}} A_{n-2}|\beta_i|_2-A_{n-1}|\beta_i|_1 \geq 0.
 $$

 If $p=0$ then there is nothing to show. Suppose that $p\geq 1$.
 Again using Lemma~\ref{12312011}, we see that if $1\leq i\leq p$ and $\beta_i\neq (\emptyset, \mathcal{D}_{i,2})$ then $$\aligned
 c_{n-2}|\beta_i|_2-c_{n-1}|\beta_i|_1 > c_{n-2} c_n -c_{n-1}^2=-1,
 \endaligned$$so  $c_{n-2}|\beta_i|_2-c_{n-1}|\beta_i|_1\geq 0.$
  Also $c_{n-1}|\beta_i|_2-c_{n}|\beta_i|_1 \geq  c_{n-1}c_{n}-c_n c_{n-1}=0$.
 Thus we have
$$\aligned
 A_{n-2}|\beta_i|_2-A_{n-1}|\beta_i|_1   &= \left(p c_{n-1} +q c_{n-2}\right)|\beta_i|_2-\left(p c_{n} +q c_{n-1}\right)|\beta_i|_1\\
            &=p\left(c_{n-1}|\beta_i|_2-c_{n}|\beta_i|_1\right) + q \left(c_{n-2}|\beta_i|_2-c_{n-1}|\beta_i|_1\right)\geq 0.
 \endaligned$$
 \end{proof}

\subsection  {}We end this section with the following rank 2 result which we will need in  section \ref{sect 3} for the rank 3 case.
\begin{theorem}\label{thm01312012}Let $a\geq \frac{A_n}{r}$ be an integer. Then
$$\sum_{ \begin{array}{c}\scriptstyle \tau_0,\tau_1,\cdots,\tau_{n-2}\\ \scriptstyle s_{n-1}=a \end{array} }\prod_{i=0}^{n-2} \left[\begin{array}{c}{{A_{i+1} - rs_i    } } \\{\tau_i} \end{array}\right]   x_1^ {rs_{n-2}-A_{n-1}} x_2^{r(A_{n-1}-a)-A_{n-2}}$$is divisible by $(1+x_1^r)^{ra-A_n}$ and the resulting quotient has nonnegative coefficients.
\end{theorem}

\begin{proof} Using Theorem \ref{mainthm2} with $z_3=1$, we get $x_{n+1}^px_n^q$ is equal to
$$
\sum_{ \tau_0,\tau_1,\cdots,\tau_{n-2}}\aligned &\left( \prod_{i=0}^{n-2} \left[\begin{array}{c}{{A_{i+1} - rs_i    } } \\{\tau_i} \end{array}\right]   \right)x_1^ {rs_{n-2}-A_{n-1}} x_2^{A_n-rs_{n-1}}.\endaligned$$
On the other hand, using Theorem \ref{mainthm2} to express $x_{n+1}^px_n^q$ in the cluster $(x_0=\frac{x_1^r+1}{x_2},x_1)$, we get
$$\sum_{ \tau_0,\tau_1,\cdots,\tau_{n-2},\tau_{n-1}}\aligned &\left( \prod_{i=0}^{n-1} \left[\begin{array}{c}{{A_{i+1} - rs_i    } } \\{\tau_i} \end{array}\right]   \right)\left(\frac{x_1^r+1}{x_2}\right)^ {rs_{n-1}-A_{n}} x_1^{A_{n+1}-rs_{n}}\endaligned
$$
Since  the positivity conjecture is known for rank 2, it follows that all the sums of products of modified binomial coefficients in this expression are positive.
Now the result follows by fixing  $s_{n-1}=a$.
\end{proof}

\section{Rank 3}\label{sect 3}

\subsection{Non-acyclic mutation classes of rank 3}
In this subsection we collect some basic results on quivers of rank 3 which are not mutation equivalent to an acyclic quiver. First let us recall how mutations act on a rank 3 quiver. Given a quiver
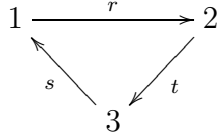
\[\xymatrix{1\ar[rr]^r&&2\ar[ld]^t\\&3\ar[lu]^s}\]
 where $r,s,t\ge 1$ denote the number of arrows, then its mutation in 1 is the quiver
 
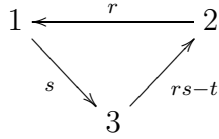
\[\xymatrix{1\ar[rd]_s&&2\ar[ll]_r \\ & 3\ar[ru]_{rs-t}}\]
 where we agree that if $rs-t<0$ then there are $|rs-t|$ arrows from 2 to 3.
\begin{lemma}
 \label{lem nonacyclic}
Let $\mathcal{A}$ be a non-acyclic cluster algebra of rank 3 with initial quiver $Q$ equal to
 
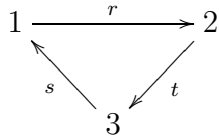
\[\xymatrix{1\ar[rr]^r&&2\ar[ld]^t\\&3\ar[lu]^s}\]
 where $r,s,t\ge 1$ denote the number of arrows.
 Then
 \begin{enumerate}
\item $r,s,t \ge 2$;
\item applying to $Q$  a mutation sequence in the vertices $1,2,1,2,1,2,\ldots$  consisting of $n$ mutations yields  the  quiver
 
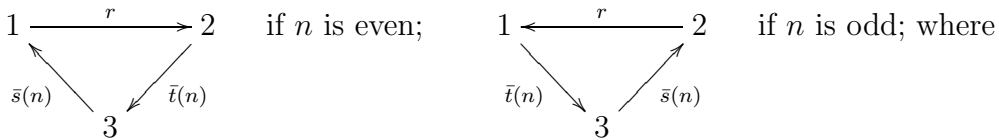
\[\xymatrix{1\ar[rr]^r&&2\ar[ld]^{\bar t(n)}\\&3\ar[lu]^{\bar s(n)}} \textup{ \quad if $n$ is even;}\qquad
 \xymatrix{1\ar[rd]_{\bar t(n)}&&2\ar[ll]_{r}\\&3\ar[ru]_{\bar s(n)}} \textup{ \quad if $n$ is odd;  where }\]
\[\bar s(n)= {c_{n+2}^{[r]}s-c_{n+1}^{[r]}t}
 \qquad \textup{and} \qquad \bar t(n)={c_{n+1}^{[r]}s-c_{n}^{[r]}t}.\]
 \item ${c_{n+1}^{[r]}s-c_{n}^{[r]}t}\ge 2$ for all $n\ge 1.$
\end{enumerate}

\end{lemma}
\begin{proof}
(1) This has been shown in \cite{BBH} \footnote{In loc. cit. it is shown that $A$ is non-acyclic if and only if  $r,s,t\ge 2 $ and $r^2+s^2+t^2-rst\le 4 $}, but we include a proof for convenience. Suppose that one of $r,s,t$ is less than 2. Without loss of generality, we may suppose $r<2$ and $s\le t$. Since $Q$ is not acyclic, $r$ cannot be zero, whence $r=1$. But then mutating $Q$ in the vertex 1 would produce the following acyclic quiver:
  $$
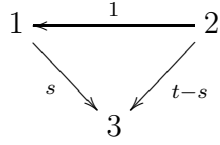
\xymatrix{ 1\ar[rd]_{s}&&2\ar[ld]^{t-s} \ar[ll]_{1} \\
  &3\\}$$
(2) We proceed by induction on $n$.
For $n=1$, the quiver obtained from $Q$ by mutation in 1 is the following:
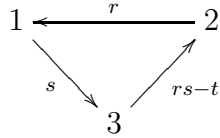
\[ \xymatrix{1\ar[rd]_{s}&&2\ar[ll]_{r}\\&3\ar[ru]_{rs- t}}\]
 and for $n=2$, the quiver obtained from $Q$ by mutation in 1 and 2 is the following:
 
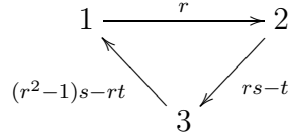
\[\xymatrix{1\ar[rr]^r&&2\ar[ld]^{rs-t}\\&3\ar[lu]^{(r^2-1)s-rt}} \]
In both cases, the result follows from
$c_1^{[r]}=0, c_2^{[r]}=1, c_3^{[r]}=r, c_4^{[r]}=r^2-1$.

Suppose that $n>2$. If $n $ is odd then by induction we know that the quiver we are considering is obtained by mutating the following quiver in  vertex 1:
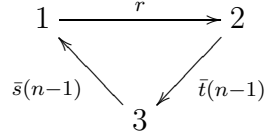
\[\xymatrix{1\ar[rr]^r&&2\ar[ld]^{\bar t(n-1)}\\&3\ar[lu]^{\bar s(n-1)}}\]
and the result follows from $\bar t(n)=\bar s(n-1) $ and
\[r \bar s(n-1)-\bar t(n-1)=
{rc_{n+1}^{[r]}s-rc_{n}^{[r]}t -c_{n}^{[r]}s+c_{n-1}^{[r]}t}
={c_{n+2}^{[r]}s-c_{n+1}^{[r]}t}
=
\bar s(n) .\]
If $n$ is even then the proof is similar. (3) follows from (1) and (2).
\end{proof}

\subsection{Positivity}
In this section, we prove the positivity conjecture in rank 3.

\begin{theorem}\label{11192011thm}
Let $\mathcal{A}(Q)$ be a skew-symmetric coefficient-free  cluster algebra of rank 3 with initial cluster $\mathbf{x}$ and let $\mathbf{x}_{t_0}$ be any cluster. Then the Laurent expansion of any cluster variable in $\mathbf{x}_{t_0}$  with respect to the cluster $\mathbf{x}$ has nonnegative coefficients.
\end{theorem}

The remainder of this section is devoted to the proof of this theorem.

If $\mathcal{A}$ is acyclic then the theorem has been proved by Kimura and Qin \cite{KQ}. We therefore
assume  that $\mathcal{A}$ is non-acyclic, but we point out that this is not a necessary  assumption for our argument. Our methods would work also in the acyclic case, but restricting to the non-acyclic case considerably simplifies the exposition.

Since $\mathcal{A}$ is non-acyclic, every quiver $Q'$ which is mutation equivalent to the initial quiver $Q$ has at least one oriented cycle. Since $Q'$ has 3 vertices and no 2-cycles, this implies that every arrow in $Q'$ lies in at least one oriented cycle.

Let $\mathbf{x}_{t_0}$ be an arbitrary cluster. Choose a finite sequence  of mutations $\mu$ which transforms $\mathbf{x}_{t_0}$ into the initial cluster $\mathbf{x}$ and which is of minimal length among all such sequences. Each mutation in this sequence $\mu$ is a mutation at one of the three vertices 1,2 or 3 of the quiver, thus the sequence $\mu$ induces a finite sequence of vertices  $(V_1,V_2,\ldots)$, where each $V_i$ equals 1,2 or 3. Since the sequence $\mu$ is of minimal length, it follows that $V_i\ne V_{i+1}$.

Let $e_{0,1}=V_1$ and $e_{0,2}=V_2$. Let ${k_1}>1$
be the least integer such that $V_{k_1}\ne e_{0,1}, e_{0,2}$, and denote by $t_1$ the seed obtained from $t_0$ after mutating at $V_1,V_2,\ldots,V_{k_1-1}$.
Then let $e_{1,1}=V_{k_1-1}$ and $e_{1,2}=V_{k_1}$.
Let ${k_2}>k_1$
be the least integer such that $V_{k_2}\ne e_{1,1}, e_{1,2}$, and denote by $t_2$ the seed obtained from $t_0$ after mutating at $V_1,V_2,\ldots,V_{k_2-1}$.
Recursively, we define a sequence of seeds $\Sigma_{t_0},\Sigma_{t_1},\Sigma_{t_2},\ldots,\Sigma_{t_m}$ along our sequence  $\mu$ such that $\mathbf{x}_{t_m}=\mathbf{x}$ is the initial cluster and, for each $i$, the subsequence of $\mu$ between $t_i$ and $t_{i+1}$ is a sequence of mutations in two vertices.

The sequence of mutations can be visualized in the following diagram.
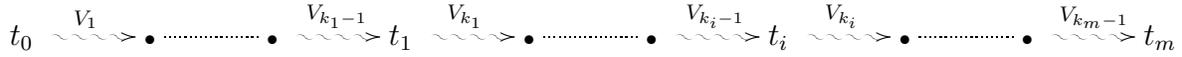
\[ \xymatrix@C35pt{t_0\ar@{~>}[r]^{V_1}&\scriptstyle \bullet\ar@{.}[r]&\scriptstyle\bullet \ar@{~>}[r]^{V_{k_1-1}} &t_1 \ar@{~>}[r]^{V_{k_1}} &\scriptstyle\bullet\ar@{.}[r]&\scriptstyle\bullet \ar@{~>}[r]^{V_{k_i-1}}& t_i \ar@{~>}[r]^{V_{k_i} }&\scriptstyle\bullet\ar@{.}[r]  & \scriptstyle\bullet \ar@{~>}[r]^{V_{k_m-1}}& t_m
}
\]

The main idea of our proof is to use our rank 2 formula from the previous section, to compute the Laurent expansion of a cluster variable of $t_0$ in the cluster at $t_1$, then replace the cluster variables of $t_1$ in this expression by their Laurent expansions in the cluster $t_2$, which we can compute again because of our rank 2 formulas. Continuing this procedure we obtain, at least theoretically, a Laurent expansion in the initial cluster.

\medskip

Fix an arbitrary $j\ge 0$, let $d=e_{j,1}$, $e=e_{j,2}$, and let $f$ be the integer such that $\{d,e,f\}=\{1,2,3\}$. Thus the sequence of mutations $\mu$ around the node $t_j$ is of the following form.
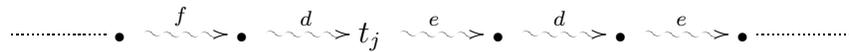
\[ \xymatrix@C35pt{ \ \ar@{.}[r]&\scriptstyle\bullet \ar@{~>}[r]^{f}&\scriptstyle\bullet \ar@{~>}[r]^{d}& t_j \ar@{~>}[r]^{e }&\scriptstyle\bullet \ar@{~>}[r]^{d}&\scriptstyle\bullet \ar@{~>}[r]^{e}&\scriptstyle\bullet\ar@{.}[r]  &
}
\]

Let
 
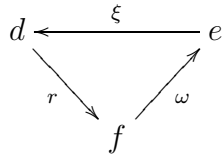
$$\xymatrix{d\ar[dr]_{r}&&e\ar[ll]_{\xi}\\&f\ar[ru]_{\omega}}$$ be the quiver at $t_j$,
and use the notation
$\{x_{1;t_j},x_{2;t_j},x_{3;t_j}\}$ for the cluster $\mathbf{x}_{t_j}$ at $t_j$.
Let $$\widetilde{x_{d;t_j}}= \frac{x_{e;t_j}^{\xi}+ x_{f;t_j}^r}{x_{d;t_j}},$$
and
 $$\widetilde{\widetilde{x_{e;t_j}}}=\frac{\widetilde{x_{d;t_j}}^{\xi} + x_{f;t_j}^{r\xi -\omega} }{x_{e;t_j}}.$$
Thus $\widetilde{x_{d;t_j}}=\mu_d({x_{d;t_j}})$ is the new cluster variable obtained by mutation  of the cluster at $t_j$ in direction $d$ and
$\widetilde{\widetilde{x_{e;t_j}}} =\mu_e\mu_d (
{{x_{e;t_j}}}) $ is the new cluster variable obtained by mutation  of the cluster at $t_j$ in direction $d$ and then $e$.

Theorem \ref{11192011thm} follows easily from the following result.
 \begin{theorem}\label{mainthm11102011_00}
For any $t_i,t_j$ with $i<j$, the Laurent expansion of the cluster monomial
$x_{e_{i,1};t_i}^p x_{e_{i,2};t_i}^q$ in the cluster $\mathbf{x}_{t_j}$ is of the form
$$\aligned
&\sum_{u_1\in\mathbb{Z}, p_1,q_1\geq 0} C_{1; \, i\rightarrow j}\, x_{f;t_j}^{u_1} \widetilde{x_{d;t_j}}^{p_1}\widetilde{\widetilde{x_{e;t_j}}}^{q_1}\\
& +   \sum_{u_2\in\mathbb{Z}, p_2,q_2\geq 0} C_{2; \, i\rightarrow j}\, x_{f;t_j}^{u_2} \widetilde{x_{d;t_j}}^{p_2} x_{e;t_j}^{q_2} \\
&+ \sum_{u_3\in\mathbb{Z}, p_3,q_3\geq 0} C_{3; \, i\rightarrow j}\, x_{f;t_j}^{u_3}  x_{d;t_j}^{p_3} x_{e;t_j}^{q_3},
\endaligned$$
where $$\aligned C_{1; \, i\rightarrow j}&=C_{1}(Q_{t_{j}}, V_{k_i}\cdots V_{k_j-1}, p,q; u_1,p_1,q_1),\\ C_{2; \, i\rightarrow j}&=C_{2}(Q_{t_{j}}, V_{k_i}\cdots V_{k_j-1}, p,q;u_2,p_2,q_2),\text{ and}\\ C_{3; \, i\rightarrow j}&=C_{3}(Q_{t_{j}}, V_{k_i}\cdots V_{k_j-1}, p,q;u_3,p_3,q_3)\\
\endaligned$$ are nonnegative integers which depend on $Q_{t_{j}}, V_{k_i}\cdots V_{k_j-1},  p,q,u_1,u_2,u_3,p_1,p_2,p_3,q_1,q_2,q_3.$
\end{theorem}

\subsection{Proof of positivity}
Before proving Theorem \ref{mainthm11102011_00}, let us show that it implies Theorem~\ref{11192011thm}.

Let $x$ be a cluster variable in $\mathbf{x}_{t_0}$.
Using Theorem \ref{mainthm11102011_00} with $i=0$ and $j=m-1$, we get an expression $LP(x,t_j)$ for $x$ as a Laurent polynomial with nonnegative coefficients in which each summand is a quotient of a certain monomial   in the variables $ x_{f;t_j},  \widetilde{x_{d;t_j}}, \widetilde{\widetilde{x_{e;t_j}}}, {x_{e;t_j}}$ by a power of $x_{f;t_j}$.  In other words, only the variable $x_{f;t_j}$ appears in the denominator of $LP(x,t_j)$.

Since the seed $t_j$ is obtained from the seed $t_m$ by a sequence of mutations at the vertices $d$ and $e$, we see  that $x_{f;t_j}=x_{f;t_m}$ is one of the initial cluster variables and that the variables $x_{d;t_j}, x_{e;t_j} \widetilde{x_{d;t_j}}, \widetilde{\widetilde{x_{e;t_j}}} $ are obtained from the initial cluster $\mathbf{x}_{t_{m}}$ by a sequence of mutations at the vertices $d$ and $e$.
Therefore, in order to write $x$ as a Laurent polynomial in the initial cluster $\mathbf{x}_{t_m}$, we only need to compute the expansions for the variables $ {x_{d;t_j}}, {x_{e;t_j}} , \widetilde{x_{d;t_j}}, \widetilde{\widetilde{x_{e;t_j}}} $ in the cluster $\mathbf{x}_{t_m}$ and substitute these expansions into   $LP(x,t_j)$.
But since the exponents $p_i,q_i$ of these variables are nonnegative, we know from Theorem \ref{thm 33} that these expansions are Laurent polynomials with nonnegative coefficients, and hence, after substitution into $LP(x,t_j)$, we get an expansion for $x$ as a Laurent polynomial with nonnegative coefficients in the initial cluster $\mathbf{x}_{t_m}$.\qed

\subsection{Proof of Theorem \ref{mainthm11102011_00}}
We use induction on $j$. Suppose that the statement holds true for $j$, and we prove it for $j+1$. Note that $f=e_{j+1,2}$, since the rank is 3. Without loss of generality, let $d=1$, so the sequence of mutations $\mu$ around the node $t_{j+1}$ is of the following form.
\[ \xymatrix@C35pt{ \ \ar@{.}[r]&\scriptstyle\bullet \ar@{~>}[r]^{e}&\scriptstyle\bullet \ar@{~>}[r]^{1}& t_{j+1} \ar@{~>}[r]^{f }&\scriptstyle\bullet \ar@{~>}[r]^{1}&\scriptstyle\bullet \ar@{~>}[r]^{f}&\scriptstyle\bullet\ar@{.}[r]  &
}
\]

 We analyze the three sums in the statement of Theorem \ref{mainthm11102011_00} separately. For the first sum, thanks to Theorem~\ref{mainthm12062011} with $x_1=x_{1;t_{j+1}}$,  $x_2=x_{e;t_{j+1}}$,    $x_3=\widetilde{x_{1;t_{j+1}}}$,
$z_3=x_{f;t_{j+1}}$,  there exist nonnegative coefficients $\overline{\overline{C_{2; \, j\rightarrow j+1}}} $ and $ \overline{\overline{C_{3; \, j\rightarrow j+1}}} $ such that

\begin{equation}\label{1217201101}\aligned & \sum_{u_1\in\mathbb{Z}, p_1,q_1\geq 0} C_{1; \, i\rightarrow j} \, x_{f;t_j}^{u_1} \widetilde{x_{1;t_j}}^{p_1}\widetilde{\widetilde{x_{e;t_j}}}^{q_1}\\
& = \sum_{u_1\in\mathbb{Z}, p_1,q_1\geq 0} C_{1; \, i\rightarrow j}\, x_{f;t_j}^{u_1} \\
&\,\,\,\,\,\,\,\,\,\,\,\,\,\, \times \left(\aligned&\sum_{v_e^{(1)}\in\mathbb{Z}, p_2^{(1)},q_2^{(1)}\geq 0} \overline{\overline{C_{2; \, j\rightarrow j+1}}} \,x_{e;t_{j+1}}^{v_e^{(1)}}\widetilde{x_{1;t_{j+1}}}^{p_2^{(1)}} x_{f;t_{j+1}}^{q_2^{(1)}} \\
&+ \sum_{w_e^{(1)}\in\mathbb{Z}, p_3^{(1)},q_3^{(1)}\geq 0} \overline{\overline{C_{3; \, j\rightarrow j+1}}}\,x_{e;t_{j+1}}^{w_e^{(1)}} x_{1;t_{j+1}}^{p_3^{(1)}} x_{f;t_{j+1}}^{q_3^{(1)}}\endaligned\right),\endaligned\end{equation}
where
$\overline{\overline{C_{2; \, j\rightarrow j+1}}} $ depends on $v_e^{(1)},p_2^{(1)},q_2^{(1)}$
and
$\overline{\overline{C_{3; \, j\rightarrow j+1}}} $ depends on $w_e^{(1)},p_3^{(1)},q_3^{(1)}.$

We can rewrite (\ref{1217201101}) as
\begin{equation}\label{1217201102}\aligned
& \sum   C_{1; \, i\rightarrow j} \overline{\overline{C_{2; \, j\rightarrow j+1}}} \, x_{e;t_{j+1}}^{v_e^{(1)}} \widetilde{x_{1;t_{j+1}}}^{p_2^{(1)}} x_{f;t_{j+1}}^{u_1+q_2^{(1)}} \\
&+ \sum C_{1; \, i\rightarrow j} \overline{\overline{C_{3; \, j\rightarrow j+1}}}\, x_{e;t_{j+1}}^{w_e^{(1)}}x_{1;t_{j+1}}^{p_3^{(1)}} x_{f;t_{j+1}}^{u_1+q_3^{(1)}}.\endaligned\end{equation}

For the second sum, we have

\begin{equation}\label{1217201103}\aligned & \sum_{u_2\in\mathbb{Z}, p_2,q_2\geq 0} C_{2; \, i\rightarrow j}  \,x_{f;t_{j}}^{u_2}\widetilde{x_{1;t_j}}^{p_2} x_{e;t_j}^{q_2}\\
& =  \sum_{u_2\in\mathbb{Z}, p_2,q_2\geq 0} C_{2; \, i\rightarrow j}  \,x_{f;t_{j}}^{u_2}\\
&\,\,\,\,\,\,\,\,\,\,\,\,\,\, \times \left(\aligned&\sum_{v_e^{(2)}\in\mathbb{Z}, p_2^{(2)},q_2^{(2)}\geq 0} \overline{C_{2; \, j\rightarrow j+1}} \,x_{e;t_{j+1}}^{v_e^{(2)}}\widetilde{x_{1;t_{j+1}}}^{p_2^{(2)}} x_{f;t_{j+1}}^{q_2^{(2)}} \\
&+ \sum_{w_e^{(2)}\in\mathbb{Z}, p_3^{(2)},q_3^{(2)}\geq 0} \overline{C_{3; \, j\rightarrow j+1}}\,x_{e;t_{j+1}}^{w_e^{(2)}}
 x_{1;t_{j+1}}^{p_3^{(2)}} x_{f;t_{j+1}}^{q_3^{(2)}}\endaligned\right).\endaligned\end{equation}
where
${\overline{C_{2; \, j\rightarrow j+1}}} $ depends on $v_e^{(2)},p_2^{(2)},q_2^{(2)}$ and
${\overline{C_{3; \, j\rightarrow j+1}}} $ depends on $w_e^{(2)},p_3^{(2)},q_3^{(2)}.$

We can rewrite (\ref{1217201103}) as
\begin{equation}\label{1217201104}\aligned
& \sum   C_{2; \, i\rightarrow j} {\overline{C_{2; \, j\rightarrow j+1}}}  \,x_{e;t_{j+1}}^{v_e^{(2)}} \widetilde{x_{1;t_{j+1}}}^{p_2^{(2)}} x_{f;t_{j+1}}^{u_2+q_2^{(2)}} \\
&+ \sum C_{2; \, i\rightarrow j} {\overline{C_{3; \, j\rightarrow j+1}}}\,x_{e;t_{j+1}}^{w_e^{(2)}}x_{1;t_{j+1}}^{p_3^{(2)}} x_{f;t_{j+1}}^{u_2+q_3^{(2)}}.\endaligned\end{equation}

Similarly, for the third sum

\begin{equation}\label{1217201105}\aligned & \sum_{u_3\in\mathbb{Z}, p_3,q_3\geq 0} C_{3; \, i\rightarrow j}  x_{f;t_j}^{u_3} x_{1;t_j}^{p_3} x_{e;t_j}^{q_3}\\
& =  \sum_{u_3\in\mathbb{Z}, p_3,q_3\geq 0} C_{3; \, i\rightarrow j}x_{f;t_j}^{u_3} \\
&\times \left(\aligned&\sum_{v_e^{(3)}\in\mathbb{Z}, p_2^{(3)},q_2^{(3)}\geq 0} C_{2; \, j\rightarrow j+1}\,x_{e;t_{j+1}}^{v_e^{(3)}}\widetilde{x_{1;t_{j+1}}}^{p_2^{(3)}} x_{f;t_{j+1}}^{q_2^{(3)}} \\
&+ \sum_{w_e^{(3)}\in\mathbb{Z}, p_3^{(3)},q_3^{(3)}\geq 0} C_{3; \, j\rightarrow j+1} \,x_{e;t_{j+1}}^{w_e^{(3)}} x_{1;t_{j+1}}^{p_3^{(3)}} x_{f;t_{j+1}}^{q_3^{(3)}}\endaligned\right),\endaligned\end{equation}
where
${C_{2; \, j\rightarrow j+1}} $ depends on $v_e^{(3)},p_2^{(3)},q_2^{(3)}$ and
${C_{3; \, j\rightarrow j+1}}$ depends on $w_e^{(3)},p_3^{(3)},q_3^{(3)}.$

We can rewrite (\ref{1217201105}) as
\begin{equation}\label{1217201106}\aligned
& \sum   C_{3; \, i\rightarrow j} {C_{2; \, j\rightarrow j+1}}  \,x_{e;t_{j+1}}^{v_e^{(3)}} \widetilde{x_{1;t_{j+1}}}^{p_2^{(3)}} x_{f;t_{j+1}}^{u_3+q_2^{(3)}} \\
&+ \sum C_{3; \, i\rightarrow j} {C_{3; \, j\rightarrow j+1}} \, x_{e;t_{j+1}}^{w_e^{(3)}} x_{1;t_{j+1}}^{p_3^{(3)}} x_{f;t_{j+1}}^{u_3+q_3^{(3)}}.\endaligned\end{equation}

So by induction we have an expression of $x_{e_{i,1};t_i}^p x_{e_{i,2};t_i}^q$ with respect to the cluster $t_{j+1}$, that is,
\begin{equation}\label{expression}x_{1;t_i}^p x_{e_{i+1};t_i}^q=(\ref{1217201102})+(\ref{1217201104})+(\ref{1217201106}).\end{equation}

This expression allows us to compute the new coefficients $C_{2;\, i\rightarrow j+1}$ and $C_{3;\, i\rightarrow j+1}$ by collecting terms with nonnegative exponents on $x_{f;t_{j+1}}$; we have
 $$C_{2;\, i\rightarrow j+1}=\sum C_{1; \, i\rightarrow j} \overline{\overline{C_{2; \, j\rightarrow j+1}}}
+\sum   C_{2; \, i\rightarrow j} {\overline{C_{2; \, j\rightarrow j+1}}}
+ \sum   C_{3; \, i\rightarrow j} {C_{2; \, j\rightarrow j+1}},
 $$where the three sums are over all possible variables satisfying $u_1 + q_2^{(1)}\geq 0$, $u_2+q_2^{(2)}\geq 0$, and $u_3+q_2^{(3)}\geq 0$ respectively, so that the exponent of $x_{f;t_{j+1}}$ is nonnegative.

 Similarly,
 $$C_{3;\, i\rightarrow j+1}=\sum C_{1; \, i\rightarrow j} \overline{\overline{C_{3; \, j\rightarrow j+1}}}
+\sum   C_{2; \, i\rightarrow j} {\overline{C_{3; \, j\rightarrow j+1}}}
+ \sum   C_{3; \, i\rightarrow j} {C_{3; \, j\rightarrow j+1}},
 $$where the three sums are over all possible variables satisfying  $u_1 + q_3^{(1)}\geq 0$, $u_2+q_3^{(2)}\geq 0$, and  $u_3+q_3^{(3)}\geq 0$ respectively. In particular, this shows that $C_{2;\, i\rightarrow j+1}$ and $C_{3;\, i\rightarrow j+1}$ are nonnegative integers.

 Since $u_1, u_2$ or $u_3$ can be negative, $x_{f;t_{j+1}}$ can have negative exponents. Now we analyze the terms in which $x_{f;t_{j+1}}$  appears with a negative exponent.
For every positive integer $\theta$, let $\mathcal{P}_\theta$ be the sum of all terms in (\ref{expression}) whose exponent of $x_{f,t_{j+1}}$ is equal to $-\theta$, that is, $\mathcal{P}_\theta=\mathcal{P}_{\theta,2}+\mathcal{P}_{\theta,3}$, where
$$\aligned \mathcal{P}_{\theta,2} &=
\sum
C_{1; \, i\rightarrow j} \overline{\overline{C_{2; \, j\rightarrow j+1}}} \,x_{e;t_{j+1}}^{v_e^{(1)}}  \widetilde{x_{1;t_{j+1}}}^{p_2^{(1)}} x_{f;t_{j+1}}^{u_1+q_2^{(1)}}  \\
&\\
&+  \sum C_{2; \, i\rightarrow j} {\overline{C_{2; \, j\rightarrow j+1}}}\,x_{e;t_{j+1}}^{v_e^{(2)}}  \widetilde{x_{1;t_{j+1}}}^{p_2^{(2)}} x_{f;t_{j+1}}^{u_2+q_2^{(2)}} \\
&\\
&+ \sum C_{3; \, i\rightarrow j} {C_{2; \, j\rightarrow j+1}} \, x_{e;t_{j+1}}^{v_e^{(3)}} \widetilde{x_{1;t_{j+1}}}^{p_2^{(3)}} x_{f;t_{j+1}}^{u_3+q_2^{(3)}},
\endaligned$$

 $$\aligned \mathcal{P}_{\theta,3} &=
 \sum C_{1; \, i\rightarrow j} \overline{\overline{C_{3; \, j\rightarrow j+1}}}\,x_{e;t_{j+1}}^{w_e^{(1)}} x_{1;t_{j+1}}^{p_3^{(1)}} x_{f;t_{j+1}}^{u_1+q_3^{(1)}} \\
&\\
&+   \sum   C_{2; \, i\rightarrow j} {\overline{C_{3; \, j\rightarrow j+1}}}\,x_{e;t_{j+1}}^{w_e^{(2)}} x_{1;t_{j+1}}^{p_3^{(2)}} x_{f;t_{j+1}}^{u_2+q_3^{(2)}} \\
&\\
&+ \sum C_{3; \, i\rightarrow j} {C_{3; \, j\rightarrow j+1}} \,x_{e;t_{j+1}}^{w_e^{(3)}} x_{1;t_{j+1}}^{p_3^{(3)}} x_{f;t_{j+1}}^{u_3+q_3^{(3)}},
\endaligned$$
where the first sum in the expression for $\mathcal{P}_{\theta,h}$ ($h=2,3$) is over all
$$
u_1,v_e^{(1)},w_e^{(1)}\in\mathbb{Z}, p_1,q_1,p_h^{(1)},q_h^{(1)}\geq 0\text{ satisfying } u_1+ q_h^{(1)}=-\theta,
$$ the second sum is over all  $$
u_2,v_e^{(2)},w_e^{(2)}\in\mathbb{Z}, p_2,q_2,p_h^{(2)},q_h^{(2)}\geq 0\text{ satisfying } u_2+q_h^{(2)}=-\theta, $$
the third sum is over all  $$
u_3,  v_e^{(3)},w_e^{(3)} \in\mathbb{Z}, p_3,q_3,p_h^{(3)},q_h^{(3)}\geq 0\text{ satisfying } u_3+q_h^{(3)}=-\theta.$$

To complete the proof, we shall compute
$\mathcal{P}_{\theta,2}$ and
$\mathcal{P}_{\theta,3}$ separately.
We show that
$\mathcal{P}_{\theta,3}=0$ in Lemma \ref{lem theta3}, and thus to complete the proof it suffices to show the following result on
$\mathcal{P}_{\theta,2}$.

\begin{lemma}\label{mainthm11102011_01}
 $\mathcal{P}_{\theta,2}$ is of the form
$$\aligned
 \widetilde{\widetilde{x_{f;t_{j+1}}}}^\theta \sum_{u_e\in\mathbb{Z}, p_1\geq 0} C_{1;i\rightarrow j+1}\,x_{e;t_{j+1}}^{u_e} \widetilde{x_{1;t_{j+1}}}^{p_1},
\endaligned$$
where
$$C_{1;i\rightarrow j+1}=C_{1}\left(B_{t_{j+1}}, V_{k_i}\cdots V_{k_{j+1}-1}, p,q; u_e,p_1,\theta\right)$$ are nonnegative integers.
\end{lemma}

\begin{proof}

$P_\theta$ is the sum of all terms in the Laurent expansion of $x_{1;t_i}^p x_{f;t_i}^q$ in the cluster $\mathbf{x}_{t_{j+1}}$ whose exponent of $x_{f;t_{j+1}}$ is equal to $-\theta$. Clearly $ P_{\theta,2}=0$, for $j=i.$ We shall start over and compute $P_\theta$ using Corollary \ref{cor 26}.  In trying to keep the notation simple, we give a detailed proof for the case $j=i+1$. The case $j>i+1$ uses the same argument.
Let $$Q_1=
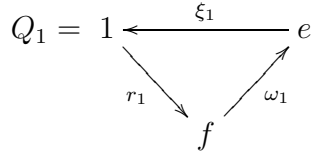
\xymatrix{1\ar[dr]_{r_1}&&e\ar[ll]_{\xi_1}\\&f\ar[ru]_{\omega_1}}$$ be the quiver at the seed $\mu_1(t_{i+1})$, where $r_1$ (respectively $\omega_1$ and $\xi_1$) is the number of arrows from 1 to $f$ (respectively from $f$ to $e$, and from  $e$ to $1$).

Applying Corollary \ref{cor 26} to $x_{1;t_i}^p x_{f;t_i}^q$, with $x_2=x_{f;t_{i+1}}$, $x_3=\widetilde{x_{1;t_{i+1}}}$, and $z_3= x_{e;t_{i+1}}$, we obtain
\begin{equation}\label{0303eq5}
 \sum_{\tau_{0;1},\tau_{1;1},\cdots,\tau_{n_1-3;1}}  \!\!\!\! \left( \prod_{w=0}^{n_1-3}\!\gchoose{\!\!A_{w+1;1} - r_1s_{w;1}\!\!}{\tau_{w;1}} \!\!  \right)
\widetilde{x_{1;t_{i+1}}}^ {A_{n_1-1;1}-r_1s_{n_1-2;1}}x_{f;t_{i+1}}^ {r_1s_{n_1-3;1}-A_{n_1-2;1}   }  x_{e;t_{i+1}}^{\omega_1 s_{n_1-2;1} - \xi_1 s_{n_1-3;1}   }.
\end{equation}

Let $n_2$ be the number of seeds between $\mu_1(t_{i+1})$ and $t_{i+2}$ inclusive. Suppose that $n_2$ is an even integer. The case that $n_2$ is odd  is similar, except that the roles of $x_{1;t_i+1}$ and $x_{e;t+1}$ are interchanged. Let $$Q_2=
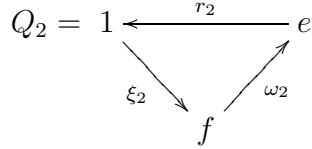
\xymatrix{1\ar[dr]_{\xi_2}&&e\ar[ll]_{r_2}\\&f\ar[ru]_{\omega_2}}$$ be the quiver at the seed $\mu_1(t_{i+2})$, where $r_2$ (respectively $\omega_2$ and $\xi_2$) is the number of arrows from $e$ to $1$ (respectively from $f$ to $e$, and from  $1$ to $f$). Since the mutation sequence relating the quivers $Q_1$ and $Q_2$ consists of mutations in the vertices 1 and $e$, we see from Lemma \ref{lem nonacyclic}

\begin{eqnarray} \label{eq 4.8.1}
 r_2&=&\xi_1,\nonumber\\
\omega_2&=&c_{n_2-1}^{[r_2]}r_1- c_{n_2-2}^{[r_2]}\omega_1,\\
 \xi_2&=&c_{n_2}^{[r_2]}r_1- c_{n_2-1}^{[r_2]}\omega_1\nonumber.
\end{eqnarray}

Now let
$p_2 ={A_{n_1-1;1}-r_1s_{n_1-2;1}}$, $q_2={\omega_1 s_{n_1-2;1} - \xi_1 s_{n_1-3;1}   }$ be the exponents of $\widetilde{x_{1;t_{i+1}}}$ and $x_{e;t_{i+1}}$ in (\ref{0303eq5}) respectively.
 Applying Corollary~\ref{cor 26} to $\widetilde{x_{1;t_{i+1}}}^{p_2}x_{e;t_{i+1}}^{q_2}$
  we see that (\ref{0303eq5}) is equal to
\begin{equation}\label{0421eq1}
\aligned
&\sum_{\tau_{0;1},\tau_{1;1},\cdots,\tau_{n_1-3;1}}   \left( \prod_{w=0}^{n_{1}-3}\gchoose{A_{w+1;1} - r_1s_{w;1}}{\tau_{w;1}}   \right) x_{f;t_{i+2}}^ {r_1s_{n_1-3;1}-A_{n_1-2;1} }  \\
&\sum_{\tau_{0;2},\tau_{1;2},\cdots,\tau_{n_{2}-3;2}}   \left( \prod_{w=0}^{n_{2}-3}\gchoose{A_{w+1;2} - r_{2}s_{w;2}}{\tau_{w;2}}   \right) \\
&\times  \widetilde{x_{1;t_{i+2}}}^ {A_{n_2-1;2}-r_{2} s_{n_{2}-2;2} }
 x_{e;t_{i+2}}^{ r_2s_{n_2-3;2}-A_{n_2-2;2}}x_{f;t_{i+2}}^{\omega_2s_{n_2-2;2}-\xi_2s_{n_2-3;2}}\\
  \\
 =
&\sum_{\tau_{0;1},\tau_{1;1},\cdots,\tau_{n_1-3;1}}   \left( \prod_{w=0}^{n_{1}-3}\gchoose{A_{w+1;1} - r_1s_{w;1}}{\tau_{w;1}}   \right)  \\
&\sum_{\tau_{0;2},\tau_{1;2},\cdots,\tau_{n_{2}-3;2}}   \left( \prod_{w=0}^{n_{2}-3}\gchoose{A_{w+1;2} - r_{2}s_{w;2}}{\tau_{w;2}}   \right) \\
&\times  \widetilde{x_{1;t_{i+2}}}^ {A_{n_2-1;2}-r_{2} s_{n_{2}-2;2} }
 x_{e;t_{i+2}}^{ r_2s_{n_2-3;2}-A_{n_2-2;2}}
 x_{f;t_{i+2}}^{\omega_2s_{n_2-2;2}-\xi_2s_{n_2-3;2}+r_1s_{n_1-3;1}-A_{n_1-2;1} }
 \endaligned\end{equation}
where $A_{i;2}$ and $s_{i;2}$ are as defined before Lemma~\ref{negone} and Lemma \ref{lem sn} but in terms of $p_2, q_2,$ and $r_2$, thus $A_{i;2}=p_2c_{i+1}^{[r_2]}+q_2c_i^{[r_2]}$ and $s_{i;2}=\sum_{j=0}^{i-1}c_{i-j+1}^{[r_2] }\tau_{j;2}$.

Let $\theta$ be a positive integer. We want to compute $P_\theta$, which is the sum of all the terms in the sum above for which  the exponent of $x_{f;t_{i+2}}$ is equal to $-\theta$, and show that it is divisible by $ \widetilde{\widetilde{x_{f;t_{j+1}}}}^\theta$. Thus $-\theta $ is equal to
\[{\omega_2s_{n_2-2;2}-\xi_2s_{n_2-3;2}+r_1s_{n_1-3;1}-A_{n_1-2;1} }.
\]

It is convenient to introduce $\varsigma$ such that $\tau_{0;2}=\varsigma -s_{n_1-3;1}$. Then
$$\aligned
s_{n_2-2;2}&=c_{n_2-1}^{[r_2]} (\varsigma -s_{n_1-3;1})+ \sum_{j=1}^{n_2-3} c_{n_2-1-j}^{[r_2]} \tau_{j;2},\text{ and }\\
s_{n_2-3;2}&=c_{n_2-2}^{[r_2]} (\varsigma -s_{n_1-3;1})+ \sum_{j=1}^{n_2-4} c_{n_2-2-j}^{[r_2]}\tau_{j;2}.
\endaligned$$

Using equation (\ref{eq 4.8.1}), the expressions for $s_{n_2-2;2}$ and $s_{n_2-3;2}$ and the fact that $c_1^{[\xi]}=0$, we have

\begin{equation*}
\aligned & \omega_2s_{n_2-2;2}-\xi_2s_{n_2-3;2 } \\
=&\ (c_{n_2-1}^{[\xi_1]}r_1-c_{n_2-2}^{[\xi_1]} \omega_1)
\left[c_{n_2-1}^{[\xi_1]}(\varsigma -s_{n_1-3;1}) +  \sum_{j=1}^{n_2-3} c_{n_2-1-j}^{[\xi_1]}\tau_{j;2} \right] \\
&-(c_{n_2}^{[\xi_1]}r_1-c_{n_2-1}^{[\xi_1]} \omega_1)
\left[c_{n_2-2}^{[\xi_1]}(\varsigma -s_{n_1-3;1}) +  \left(\sum_{j=1}^{n_2-3} c_{n_2-2-j}^{[\xi_1]}\tau_{j;2}\right) - c_1^{[\xi_1]}\tau_{n_2-3;2}\right] \\
=&\ (\varsigma -s_{n_1-3;1})\,r_1\left( (c_{n_2-1}^{[\xi_1]})^2-c_{n_2}^{[\xi_1]}c_{n_2-2}^{[\xi_1]}\right) \\
&+\sum_{j=1}^{n_2-3} \tau_{j;2} \left[ r_1\left(c_{n_2-1}^{[\xi_1]}c_{n_2-1-j}^{[\xi_1]} -c_{n_2}^{[\xi_1]}c_{n_2-2-j}^{[\xi_1]}\right)
+\omega_1\left(-c_{n_2-2}^{[\xi_1]}c_{n_2-1-j}^{[\xi_1]}+c_{n_2-1}^{[\xi_1]}c_{n_2-2-j}^{[\xi_1]}
\right)
\right] \\
{=}
&\ (\varsigma -s_{n_1-3;1})\,r_1
+\sum_{j=1}^{n_2-3} \tau_{j;2} \left[ r_1\left(-c_{-j}^{[\xi_1]}\right)
+\omega_1c_{1-j}^{[\xi_1]}\right] \qquad \textup{(by Lemma \ref{lem cn})}
\\
\endaligned
\end{equation*}
And since $-c_{-j}^{[\xi_1]}=c_{j+2}^{[\xi_1]}$, we get

\begin{equation}\label{theta}\aligned-\theta
&=r_1(\varsigma -s_{n_1-3;1}) + \sum_{j=1}^{n_2-3} \tau_{j;2}(c_{j+2}^{[\xi_1]}r_1 - c_{j+1}^{[\xi_1]}\omega_1)+ r_1s_{n_1-3;1}-A_{n_1-2;1} \\
&=-A_{n_1-2;1} +r_1\varsigma + \sum_{j=1}^{n_2-3} \tau_{j;2}(c_{j+2}^{[\xi_1]}r_1 - c_{j+1}^{[\xi_1]}\omega_1).
\endaligned\end{equation}
Also, the exponents of $\widetilde{x_{1;t_{i+2}}}$ and  $x_{e;t_{i+2}}$ in (\ref{0421eq1}) can be expressed as follows:
$$\aligned
A_{n_2-1;2}-r_{2} s_{n_{2}-2;2} = \ & c_{n_2}^{[\xi_1]}p_2 + c_{n_2-1}^{[\xi_1]}q_2
- \xi_1\left(c_{n_2-1}^{[\xi_1]} (\varsigma -s_{n_1-3;1})+ \sum_{j=1}^{n_2-3} c_{n_2-1-j}^{[\xi_1]} \tau_{j;2}\right)\\
= \ & c_{n_2}^{[\xi_1]}(A_{n_1-1;1}-r_1s_{n_1-2;1} ) + c_{n_2-1}^{[\xi_1]}(\omega_1 s_{n_1-2;1} - \xi_1 s_{n_1-3;1}) \\
& - \xi_1\left(c_{n_2-1}^{[\xi_1]} (\varsigma -s_{n_1-3;1})+ \sum_{j=1}^{n_2-3} c_{n_2-1-j}^{[\xi_1]} \tau_{j;2}\right)\\
 =\ & c_{n_2}^{[\xi_1]}(A_{n_1-1;1}-r_1s_{n_1-2;1} ) + c_{n_2-1}^{[\xi_1]}\omega_1 s_{n_1-2;1}\\
&- \xi_1\left(c_{n_2-1}^{[\xi_1]} \varsigma+ \sum_{j=1}^{n_2-3} c_{n_2-1-j}^{[\xi_1]} \tau_{j;2}\right),\endaligned$$
and similarly
\[\aligned \xi_1s_{n_2-3;2}-A_{n_2-2;2} 
= \ & \xi_1\left(c_{n_2-2}^{[\xi_1]} \varsigma + \sum_{j=1}^{n_2-4} c_{n_2-2-j}^{[\xi_1]}\tau_{j;2}\right) \\
& - \left( c_{n_2-1}^{[\xi_1]}(A_{n_1-1;1}-r_1s_{n_1-2;1} ) + c_{n_2-2}^{[\xi_1]}\omega_1 s_{n_1-2;1} \right).
\endaligned\]
By fixing $\varsigma,\tau_{1;2},\cdots,\tau_{n_{2}-3;2}$ in (\ref{0421eq1}), we have
$$\aligned
&\sum_{\tau_{0;1},\tau_{1;1},\cdots,\tau_{n_1-3;1}}   \left( \prod_{w=0}^{n_{1}-3}\gchoose{A_{w+1;1} - r_1s_{w;1}}{\tau_{w;1}}   \right)  \left( \prod_{w=0}^{n_{2}-3}\gchoose{A_{w+1;2} - r_{2}s_{w;2}}{\tau_{w;2}}   \right) \\
&\times  \widetilde{x_{1;t_{i+2}}}^ {c_{n_2}^{[\xi_1]}(A_{n_1-1;1}-r_1s_{n_1-2;1} ) + c_{n_2-1}^{[\xi_1]}\omega_1 s_{n_1-2;1}- \xi_1\left(c_{n_2-1}^{[\xi_1]} \varsigma+ \sum_{j=1}^{n_2-3} c_{n_2-1-j}^{[\xi_1]} \tau_{j;2}\right) }\\
&\times   x_{e;t_{i+2}}^{ \xi_1\left(c_{n_2-2}^{[\xi_1]} \varsigma + \sum_{j=1}^{n_2-4} c_{n_2-2-j}^{[\xi_1]}\tau_{j;2}\right) - \left( c_{n_2-1}^{[\xi_1]}(A_{n_1-1;1}-r_1s_{n_1-2;1} ) + c_{n_2-2}^{[\xi_1]}\omega_1 s_{n_1-2;1} \right)} \\
&\times   x_{f;t_{i+2}}^{-A_{n_1-2;1} +r_1\varsigma + \sum_{j=1}^{n_2-3} (c_{j+2}^{[t_1]}r_1 - c_{j+1}^{[t_1]}\omega_1)\tau_{j;2} }
 \endaligned$$
and this is equal to a product $\phi \varphi$ where $\phi$ is a Laurent monomial in $ \widetilde{x_{1;t_{i+2}}}, x_{e;t_{i+2}},x_{f;t_{i+2}}$ and $\varphi$ is equal to
\begin{equation*}\aligned
\sum_{\tau_{0;1},\tau_{1;1},\cdots,\tau_{n_1-3;1}}   \left( \prod_{w=0}^{n_{1}-3}\gchoose{A_{w+1;1} - r_1s_{w;1}}{\tau_{w;1}}   \right)   \left( \prod_{w=0}^{n_{2}-3}\gchoose{A_{w+1;2} - r_{2}s_{w;2}}{\tau_{w;2}}   \right) \\
\times \left(\frac{\widetilde{x_{1;t_{i+2}}}^{c_{n_2}^{[\xi_1]}r_1- c_{n_2-1}^{[\xi_1]}\omega_1} }{ x_{e;t_{i+2}}^{c_{n_2-1}^{[\xi_1]}r_1- c_{n_2-2}^{[\xi_1]}\omega_1}}\right)^{ \left\lfloor (A_{n_1-1;1}-\varsigma) \frac{A_{n_1-1;1}}{A_{n_1;1}}\right\rfloor -s_{n_1-2;1}}
\endaligned
\end{equation*} and transferring the 0-th term of the second product to a $(n_1-2)$-nd term in the first product, we get
\begin{equation}\label{eq goal2}\aligned
\sum_{\tau_{0;1},\tau_{1;1},\cdots,\tau_{n_1-2;1}}   \left( \prod_{w=0}^{n_{1}-2}\gchoose{A_{w+1;1} - r_1s_{w;1}}{\tau_{w;1}}   \right)   \left( \prod_{w=1}^{n_{2}-3}{A_{w+1;2} - r_{2}s_{w;2}\choose\tau_{w;2}}   \right)
\\
  \times \left(\frac{\widetilde{x_{1;t_{i+2}}}^{c_{n_2}^{[\xi_1]}r_1- c_{n_2-1}^{[\xi_1]}\omega_1} }{ x_{e;t_{i+2}}^{c_{n_2-1}^{[\xi_1]}r_1- c_{n_2-2}^{[\xi_1]}\omega_1}}\right)^{ \left\lfloor (A_{n_1-1;1}-\varsigma) \frac{A_{n_1-1;1}}{A_{n_1;1}}\right\rfloor -s_{n_1-2;1}},
 \endaligned
 \end{equation}
where $\tau_{n_1-2;1}=A_{n_1-1;1}-r_1s_{n_1-2;1}-\tau_{0;2}=A_{n_1-1;1}-r_1s_{n_1-2;1}-\varsigma +s_{n_1-3;1}$.
Moreover, using Lemma~\ref{lem sn}, we observe that
\begin{equation}\label{eq 4.10a}
 s_{n_1-1;1}=r_1s_{n_1-2;1} - s_{n_1-3;1} + \tau_{n_1-2;1} =A_{n_1-1;1}-\varsigma,
\end{equation}
and by Theorem~\ref{thm01312012},
$$
\sum_{\begin{array}{c}\scriptstyle \tau_{0;1},\tau_{1;1},\cdots,\tau_{n_1-2;1}\\ \scriptstyle s_{n_1-1;1}=A_{n_1-1;1}-\varsigma \end{array}}   \left( \prod_{w=0}^{n_1-2}\gchoose{A_{w+1;1} - r_1s_{w;1}}{\tau_{w;1}}   \right){x_{1;t_{i+1}}}^ {A_{n_1-1;1}-r_1s_{n_1-2;1}}
$$
is   divisible by $(1+{x_{1;t_{i+1}}}^{r_1})^{r_1(A_{n_1-1;1}-\varsigma) -A_{n_1;1}}$  in $\mathbb{Z}[{x_{1;t_{i+1}}}^{\pm 1}]$, and the resulting quotient has nonnegative coefficients.
Multiplying the sum with ${x_{1;t_{i+1}}}^{r_1  \left\lfloor (A_{n_1-1;1}-\varsigma) \frac{A_{n_1-1;1}}{A_{n_1;1}}\right\rfloor -A_{n_1-1;1}}$ shows that
\begin{equation}\label{double star}
\sum_{\begin{array}{c} \scriptstyle \tau_{0;1},\tau_{1;1},\cdots,\tau_{n_1-2;1}\\\scriptstyle  s_{n_1-1;1}=A_{n_1-1;1}-\varsigma\end{array}}   \left( \prod_{w=0}^{n_1-2}\gchoose{A_{w+1;1} - r_1s_{w;1}}{\tau_{w;1}}   \right)({x_{1;t_{i+1}}}^{r_1})^{ \left\lfloor (A_{n_1-1;1}-\varsigma) \frac{A_{n_1-1;1}}{A_{n_1;1}}\right\rfloor -s_{n_1-2;1}}
\end{equation}
is also divisible by $(1+{x_{1;t_{i+1}}}^{r_1})^{r_1(A_{n_1-1;1}-\varsigma) -A_{n_1;1}}$, and the resulting quotient has nonnegative coefficients.
Moreover, we shall show in Lemma~\ref{0302lem1} below that the exponents in the expression (\ref{double star}) are nonnegative,
which implies that the quotient is a polynomial.

Note that the statement about the divisibility of (\ref{double star}) also holds when we replace $({x_{1;t_{i+1}}}^r)$ with any other expression $X$.
We can write (\ref{eq goal2}) as follows:
\[ (\ref{eq goal2}) = \sum\sum_mq(m)p(m)X^{b-m},\]
where
\[
\begin{array}
 {rcl}
q(m) &=& \displaystyle\prod_{w=0}^{n_{1}-2}\gchoose{A_{w+1;1} - r_1s_{w;1}}{\tau_{w;1}} \\
& \\
p(m)&=&\displaystyle  \prod_{w=1}^{n_{2}-3}{{A_{w+1;2} - r_{2}s_{w;2}}\choose{\tau_{w;2}}}  \\
\\
b &=& \left\lfloor (A_{n_1-1;1}-\varsigma) \frac{A_{n_1-1;1}}{A_{n_1;1}}\right\rfloor \\
& \\
m&=& s_{n_1-2;1} \\
\\
X&=& \left(\frac{\widetilde{x_{1;t_{i+2}}}^{c_{n_2}^{[\xi_1]}r_1- c_{n_2-1}^{[\xi_1]}\omega_1} }{ x_{e;t_{i+2}}^{c_{n_2-1}^{[\xi_1]}r_1- c_{n_2-2}^{[\xi_1]}\omega_1}}\right)
\end{array}\]
Then Lemma \ref{0303lem945} below yields
\[p(m) =
\sum_{i=0}^{\sum_{w=1}^{n_2-3}\tau_{w;2}} d_i {\left\lfloor (A_{n_1-1;1}-\varsigma) \frac{A_{n_1-1;1}}{A_{n_1;1}}\right\rfloor-s_{n_1-2;1}\choose i}\]
and using Lemma~\ref{0303lem944}
with
$g=r_1(A_{n_1-1;1}-\varsigma) -A_{n_1;1}$ and
$h=\sum_{j=1}^{n_2-3} \tau_{j;2}$,
we get that the expression  in (\ref{eq goal2})
is divisible by
\begin{equation}\label{divisor}
\aligned
&\left(1+\left(\frac{\widetilde{x_{1;t_{i+2}}}^{c_{n_2}^{[\xi_1]}r_1- c_{n_2-1}^{[\xi_1]}\omega_1} }{ x_{e;t_{i+2}}^{c_{n_2-1}^{[\xi_1]}r_1- c_{n_2-2}^{[\xi_1]}\omega_1}}\right)\right)^{r_1(A_{n_1-1;1}-\varsigma) -A_{n_1;1}-\sum_{j=1}^{n_2-3} \tau_{j;2}}\\
&=\left(1+\left(\frac{\widetilde{x_{1;t_{i+2}}}^{\xi_2} }{ x_{e;t_{i+2}}^{\omega_2}}\right)\right)^{A_{n_1-2;1} -r_1\varsigma -\sum_{j=1}^{n_2-3} \tau_{j;2}},\endaligned\end{equation} and the resulting quotient has nonnegative coefficients. Finally, dividing (\ref{divisor}) by $x_{f;t_{i+2}}^\theta$ and using the fact that
\[ \widetilde{\widetilde{x_{f;t_{i+2}}}} =(x_{e;t_{i+2}}^{\omega_2}+\widetilde{x_{1;t_{i+2}}}^{\xi_2} )/x_{f;t_{i+2}}\]
we see that $\mathcal{P}_{\theta,2}$ is divisible by
 $\widetilde{\widetilde{x_{f;t_{i+2}}}}^\theta$.
\end{proof}

\begin{lemma}\label{0302lem1}
 $$\left\lfloor (A_{n_1-1;1}-\varsigma) \frac{A_{n_1-1;1}}{A_{n_1;1}}\right\rfloor -s_{n_1-2;1}\ge 0.$$
\end{lemma}
\begin{proof}
We have
\begin{equation*}
{s_{n_1-1;1}-s'_{n_1-1;1}} \geq \left(s_{n_1-2;1} -s'_{n_1-2;1}\right) \frac{A_{n_1-1;1}}{A_{n_1-2;1}}\geq \left(s_{n_1-2;1} -s'_{n_1-2;1}\right) \frac{A_{n_1;1}}{A_{n_1-1;1}},
\end{equation*}
where the first inequality follows from (\ref{cond502}) and  the second one follows from Lemma~\ref{negone}(b). Hence
\begin{equation}\label{0302eq1}
\left(s_{n_1-1;1}-s'_{n_1-1;1}\right)A_{n_1-1;1}\geq \left(s_{n_1-2;1} -s'_{n_1-2;1}\right) A_{n_1;1}.
\end{equation}

On the other hand, since $s'_{n_1-2;1}=kc^{[r_1]}_{n_1-1;1}$, $s'_{n_1-1;1}=kc^{[r_1]}_{n;1}$, by Definition \ref{20120121}, and $A_{i;1}=pc_{i+1;1}^{[r_1]}+qc^{[r_1]}_{i;1}$, we have

\begin{equation}\label{0302eq2}
s'_{n_1-1;1}A_{n_1-1;1}= kc_{n;1}^{[r_1]}(pc_{n;1}^{[r_1]}+qc_{n_1-1;1}^{[r_1]})\geq kc^{[r_1]}_{n_1-1;1}(pc^{[r_1]}_{n+1;1}+qc^{[r_1]}_{n;1})=s'_{n_1-2;1}A_{n_1;1},
\end{equation}
where the inequality follows from Lemma \ref{lem cn}.
Adding (\ref{0302eq1}) and (\ref{0302eq2}) we get
$$
s_{n_1-1;1}A_{n_1-1;1}\geq s_{n_1-2;1}A_{n_1;1}.
$$
Therefore equation (\ref{eq 4.10a}) yields
$$ (A_{n_1-1;1} -\varsigma)A_{n_1-1;1}\ge
s_{n_1-2;1}A_{n_1;1},
$$ and
we get
\[ (A_{n_1-1;1}-\varsigma) \frac{A_{n_1-1;1}}{A_{n_1;1}}-s_{n_1-2;1}
\ge s_{n_1-2;1} -s_{n_1-2;1}
 =0.\]
\end{proof}

\begin{lemma}\label{0303lem944}
Suppose that a polynomial of $x$ of the form $$\sum_{m\in I} q(m)x^{b-m}$$is divisible by $(1+x)^g$ and its quotient has nonnegative coefficients. Let
$$
p(m)=\sum_{i=0}^h d_i {b-m\choose i}
$$ be a polynomial of $m$ with $d_i\geq 0$.
Then  $$\sum_{m\in I} p(m)q(m)x^{b-m}$$is divisible by $(1+x)^{g-h}$ and its quotient has nonnegative coefficients.
\end{lemma}
\begin{proof}
This is because $x^i\frac{d^i}{dx^i}\sum_{m\in I} q(m)x^{b-m}$ is divisible by $(1+x)^{g-i}$ and its quotient has nonnegative coefficients.
\end{proof}

\begin{lemma}\label{0303lem945}
With assumptions in Proof of Lemma~\ref{mainthm11102011_01}, we have
$$
 \prod_{w=1}^{n_{2}-3}{{A_{w+1;2} - r_{2}s_{w;2}}\choose{\tau_{w;2}}}
=
\sum_{i=0}^{\sum_{w=1}^{n_2-3}\tau_{w;2}} d_i {\left\lfloor (A_{n_1-1;1}-\varsigma) \frac{A_{n_1-1;1}}{A_{n_1;1}}\right\rfloor-s_{n_1-2;1}\choose i}
$$ for some $d_i\in \mathbb{N}$, which are independent of $s_{n_1-2;1}$.
\end{lemma}
\begin{proof}
Once we know that there are nonnegative integers $a$ and $b$ such that $A_{w+1;2}-r_{2} s_{w;2}=a\left( \left\lfloor (A_{n_1-1;1}-\varsigma) \frac{A_{n_1-1;1}}{A_{n_1;1}}\right\rfloor-s_{n_1-2;1}\right) +b$, then it is clear, by Lemma~\ref{0423lem1}, that
$$
{{A_{w+1;2} - r_{2}s_{w;2}}\choose{\tau_{w;2}}} =
\sum_{i=0}^{\tau_{w;2}} d_i' {\left\lfloor (A_{n_1-1;1}-\varsigma) \frac{A_{n_1-1;1}}{A_{n_1;1}}\right\rfloor-s_{n_1-2;1}\choose i}
$$ for some $d_i'\in \mathbb{N}$, and by Lemma~\ref{0423lem2}, for any nonnegative integers $j$ and $k$,
$$\aligned
&{\left\lfloor (A_{n_1-1;1}-\varsigma) \frac{A_{n_1-1;1}}{A_{n_1;1}}\right\rfloor-s_{n_1-2;1}\choose j}
{\left\lfloor (A_{n_1-1;1}-\varsigma) \frac{A_{n_1-1;1}}{A_{n_1;1}}\right\rfloor-s_{n_1-2;1}\choose k}\\
&=\sum_{i=0}^{j+k} d_i'' {\left\lfloor (A_{n_1-1;1}-\varsigma) \frac{A_{n_1-1;1}}{A_{n_1;1}}\right\rfloor-s_{n_1-2;1}\choose i}
\endaligned$$ for some $d_i''\in \mathbb{N}$.
Then the desired statement easily follows.

Thus we need to show the existence of the nonnegative integers $a$ and $b$.
Using the definitions of $A_{w+1;2}$ and $\varsigma$  as well as  the fact that $r_2=\xi_1$, we get
$$\aligned
A_{w+1;2}-r_{2} s_{w;2}
 = c_{w+2}^{[\xi_1]}(A_{n_1-1;1}-r_1s_{n_1-2;1} ) + c_{w+1}^{[\xi_1]}\omega_1 s_{n_1-2;1}- \xi_1\left(c_{w+1}^{[\xi_1]} \varsigma+ \sum_{j=1}^{w-1} c_{w+1-j}^{[\xi_1]} \tau_{j;2}\right)\endaligned$$ which can be written as
$$ A_{w+1;2}-r_{2} s_{w;2}
 = (c_{w+2}^{[\xi_1]}r_1-c_{w+1}^{[\xi_1]}\omega_1)\left(\left\lfloor (A_{n_1-1;1}-\varsigma) \frac{A_{n_1-1;1}}{A_{n_1;1}}\right\rfloor-s_{n_1-2;1} \right)+C(w),
$$ where $C(w)$ is some function of $w$, which is independent of $s_{n_1-2;1}$. Note that $$c_{w+2}^{[\xi_1]}r_1-c_{w+1}^{[\xi_1]}\omega_1 >0,$$ because, by Lemma \ref{lem nonacyclic}, this is the number of arrows between some pair of vertices in some seed between $t_{i+1}$ and $t_{i+2}$. Thus is suffices to show that $C(w)$ is nonnegative.

$$\aligned
C(w)=&(c_{w+2}^{[\xi_1]}r_1-c_{w+1}^{[\xi_1]}\omega_1)\left( (A_{n_1-1;1}-\varsigma) \frac{A_{n_1-1;1}}{A_{n_1;1}}-\left\lfloor (A_{n_1-1;1}-\varsigma) \frac{A_{n_1-1;1}}{A_{n_1;1}}\right\rfloor\right)+ \tilde{C}(w) \theta(w), \endaligned$$
where
$$\aligned
 \tilde{C}(w)&=c_{w+2}^{[\xi_1]} - (c_{w+2}^{[\xi_1]}r_1-c_{w+1}^{[\xi_1]}\omega_1)\frac{A_{n_1-1;1}}{A_{n_1;1}} \text{ and}  \\
\theta(w) &= A_{n_1-1;1} - \frac{\xi_1c_{w+1}^{[\xi_1]} - (c_{w+2}^{[\xi_1]}r_1-c_{w+1}^{[\xi_1]}\omega_1)\frac{A_{n_1-1;1}}{A_{n_1;1}}}{c_{w+2}^{[\xi_1]} - (c_{w+2}^{[\xi_1]}r_1-c_{w+1}^{[\xi_1]}\omega_1)\frac{A_{n_1-1;1}}{A_{n_1;1}}}\varsigma- \sum_{j=1}^{w-1}  \frac{\xi_1c_{w+1-j}^{[\xi_1]}}{c_{w+2}^{[\xi_1]} - (c_{w+2}^{[\xi_1]}r_1-c_{w+1}^{[\xi_1]}\omega_1)\frac{A_{n_1-1;1}}{A_{n_1;1}}} \tau_{j;2}.
  \endaligned$$
We want to show that $C(w)$ is nonnegative for $w\geq 1$, for which it suffices to show that $\tilde{C}(w)$ and $\theta(w)$ are nonnegative for $w\geq 1$.

First we show that $\tilde{C}(w)$ are nonnegative for $w\geq 1$. Note that $\tilde{C}(w)=\xi_1\tilde{C}(w-1)-\tilde{C}(w-2)$. Moreover $\xi_1\ge 2$, by Lemma \ref{lem nonacyclic}. Hence
if we show $\tilde{C}(1)>0\geq \tilde{C}(0)$ then the induction on $w$ will show that $\tilde{C}(w)$ is increasing with $w$. It is easy to see that $\tilde{C}(0)=1-r_1\frac{A_{n_1-1;1}}{A_{n_1;1}} \leq 0$. On the other hand,
$$\begin{array}{rcccl}\tilde{C}(1)&=&
\xi_1-(\xi_1r_1-\omega_1)\frac{A_{n_1-1;1}}{A_{n_1;1}}&=&
\xi_1(\frac{{A_{n_1;1}}-r_1A_{n_1-1;1}}{A_{n_1;1}}) + \omega_1\frac{A_{n_1-1;1}}{A_{n_1;1}}\\
&=&
\xi_1(\frac{-{A_{n_1-2;1}}}{A_{n_1;1}}) + \omega_1\frac{A_{n_1-1;1}}{A_{n_1;1}}\end{array} $$
which is positive because
\begin{eqnarray} \label{eq 44.1}
\omega_1A_{n_1-1;1} -\xi_1A_{n_1-2;1}&=&\omega_1(pc_{n_1}+qc_{n_1-1}) -\xi_1(pc_{n_1-1} + qc_{n_1-2})\\
&=& p(\omega_1 c_{n_1} -\xi_1c_{n_1-1}) + q (\omega_1 c_{n_1-1} -\xi_1 c_{n_1-2}))>0,\nonumber
\end{eqnarray}
where the last inequality follows from
$$\aligned
&\{\omega_1 c_{n_1} -\xi_1c_{n_1-1}, \omega_1 c_{n_1-1} -\xi_1 c_{n_1-2}\}\\&=\{(\text{the number of arrows between } e\text{ and } 1\text{ in the seed }t_i ), \\
&\,\,\,\,\,\,\,\,\,\,\,(\text{the number of arrows between }e \text{ and } f\text{ in the seed }t_i)\}.
\endaligned$$

Next we show that $\theta(w)$ are nonnegative for all $w$ such that $1\le w\le n_2-3$.   Recall from (\ref{theta}) that
\begin{equation*}\label{02292012eq2}\theta=A_{n_1-2;1} -r_1\varsigma - \sum_{j=1}^{n_2-3} (c_{j+2}^{[t_1]}r_1 - c_{j+1}^{[t_1]}\omega_1)\tau_{j;2}  >0,\end{equation*}which implies that
\begin{equation*}\label{04222012eq5}A_{n_1-1;1} -\frac{r_1 A_{n_1-1;1} }{A_{n_1-2;1} }\varsigma - \sum_{j=1}^{w-1} \frac{(c_{j+2}^{[t_1]}r_1 - c_{j+1}^{[t_1]}\omega_1) A_{n_1-1;1} }{A_{n_1-2;1} }\tau_{j;2}  >0.\end{equation*}
So it is enough to show
$$\frac{r_1 A_{n_1-1;1} }{A_{n_1-2;1} } >  \frac{\xi_1c_{w+1}^{[\xi_1]} - (c_{w+2}^{[\xi_1]}r_1-c_{w+1}^{[\xi_1]}\omega_1)\frac{A_{n_1-1;1}}{A_{n_1;1}}}{c_{w+2}^{[\xi_1]} - (c_{w+2}^{[\xi_1]}r_1-c_{w+1}^{[\xi_1]}\omega_1)\frac{A_{n_1-1;1}}{A_{n_1;1}}}$$and
$$\frac{(c_{j+2}^{[t_1]}r_1 - c_{j+1}^{[t_1]}\omega_1) A_{n_1-1;1} }{A_{n_1-2;1} } > \frac{\xi_1c_{w+1-j}^{[\xi_1]}}{c_{w+2}^{[\xi_1]} - (c_{w+2}^{[\xi_1]}r_1-c_{w+1}^{[\xi_1]}\omega_1)\frac{A_{n_1-1;1}}{A_{n_1;1}}},$$ but these inequalities can be proved by induction on $w$.
\end{proof}

\begin{lemma}\label{0423lem1}
Let $a,b,c$ be any nonnegative integers. Then there are nonnegative integers $d_0,...,d_c$ such that
$$
{aX+b\choose c} = \sum_{i=0}^c d_i{X\choose i}
$$for all nonnegative integers $X$.
\end{lemma}
\begin{proof}
The Vandermonde identity shows
$$
{aX+b\choose c} = \sum_{w_0+w_1+\cdots+w_a=c \atop w_0,...,w_a\in \mathbb{N} } {b\choose w_0}\prod_{v=1}^a {X\choose w_v},
$$and then the statement follows from Lemma~\ref{0423lem2}.
\end{proof}

\begin{lemma}\label{0423lem2}
Let $a,b$ be any nonnegative integers. Then there are nonnegative integers $e_0,...,e_{a+b}$ such that
$$
{X\choose a} {X\choose b} = \sum_{i=0}^{a+b} e_i{X\choose i}
$$for all nonnegative  integers $X$.
\end{lemma}

\begin{proof} There are many proofs. This proof is due to Qiaochu Yuan and Brendan McKay.
It is enough to prove for large enough integers $X$, because both sides can be regarded as polynomials of $X$.
$
{X\choose a} {X\choose b}$ is the number of ways to choose a subset of size $a$ and a subset of size $b$ from a set of size $X$. The union of these two subsets is a subset of size anywhere from $\max(a,b)$ to $a+b$, so $e_i$ is the number of different ways a subset of size $i$ can be realized as the union of a subset of size $a$ and a subset of size $b$.
\end{proof}

This completes the proof of Lemma \ref{mainthm11102011_01}.

\bigskip

\begin{lemma}\label{lem theta3}
$\mathcal{P}_{\theta,3}=0$.\end{lemma}

\begin{proof}
In trying to keep the notation simple, we give a detailed proof for the case $j=i+1$. The case $j>i+1$ uses the same argument.

Let $$Q_1=
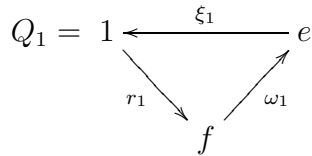
\xymatrix{1\ar[dr]_{r_1}&&e\ar[ll]_{\xi_1}\\&f\ar[ru]_{\omega_1}}$$ be the quiver at the seed $\mu_1(t_{i+1})$, where $r_1$ (respectively $\omega_1$ and $\xi_1$) is the number of arrows from 1 to $f$ (respectively from $f$ to $e$, and from  $e$ to $1$). For a compatible pair $\beta=(S_1,S_2)$, let
$|\beta|_2$ denote $|\mathcal{D}_{1}|-|S_{1}|$ and $|\beta|_1$ denote $|S_{2}|$.

Applying Theorem \ref{mainthm12062011} to $x_{1;t_i}^p x_{f;t_i}^q$, we obtain
\begin{equation}\label{0303eq5a}
 \sum_{\begin{array}{c}\scriptstyle\tau_{0;1},\cdots,\tau_{n_1-3;1}\\ \scriptstyle A_{n_1-1;1}-r_1s_{n_1-2;1}\geq 0\end{array}}  \!\!\!\!\!\!\!\!\!\!\!\! \left( \prod_{w=0}^{n_1-3}\gchoose{\!\!\!A_{w+1;1} - r_1s_{w;1}\!\!\!}{\tau_{w;1}} \!\!  \right){\widetilde{x_{1;t_{i+1}}}}^ {A_{n_1-1;1}-r_1s_{n_1-2:1}}x_{f;t_{i+1}}^ {r_1s_{n_1-3;1}-A_{n_1-2;1}   }  x_{e;t_{i+1}}^{\omega_1 s_{n_1-2;1} - \xi_1  s_{n_1-3;1}   }
 \end{equation}
\begin{equation}\label{0303eq5b}
+ \,\, x_{1;t_{i+1}}^{-A_{n_1-1;1}} x_{f;t_{i+1}}^{-A_{n_1-2;1}}\sum_{\scriptstyle \beta : r_1|\beta|_1-A_{n_1-1;1}> 0} x_{1;t_{i+1}}^{r_1|\mathbf{\beta}|_1}x_{f;t_{i+1}}^{r_1(A_{n_1-1;1}-|\mathbf{\beta}|_2)} x_{e;t_{i+1}}^{\xi_1|\beta|_2-(\xi_1r_1-\omega_1)|\beta|_1},
\end{equation}
where the second sum is over all $\beta=(S_1=\cup_{i=1}^{p+q}\, S_{1}^i,S_2=\cup_{i=1}^{p+q}\, S_{2}^i)$ such that  
$$(S_{1}^i,S_{2}^i) \textup{ is a compatible pair in }\left\{\begin{array}{ll}\mathcal{D}^{c_{n_1-1;1}\times c_{{n_1-2;1}}}  &\textup{ if $1\le i \le q$}; \\  
 \mathcal{D}^{c_{n_1;1}\times c_{n_1-1;1}} &\textup{ if $q+1\le i\le p+q$.}
\end{array}\right.$$
Let $n_2$ be the number of seeds between $\mu_1(t_{i+1})$ and $t_{i+2}$ inclusive. Suppose that $n_2$ is an even integer. The case that $n_2$ is odd  is similar, except that the roles of $x_{1;t_i+1}$ and $x_{e;t+1}$ are interchanged. Let $$Q_2=
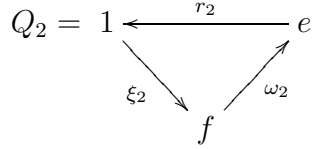
\xymatrix{1\ar[dr]_{\xi_2}&&e\ar[ll]_{r_2}\\&f\ar[ru]_{\omega_2}}$$ be the quiver at the seed $\mu_1(t_{i+2})$, where $r_2$ (respectively $\omega_2$ and $\xi_2$) is the number of arrows from $e$ to $1$ (respectively from $f$ to $e$, and from  $1$ to $f$).

Here we show that if $C_{2; \, i\rightarrow j} {\overline{C_{3; \, j\rightarrow j+1}}}\neq 0$ then $u_2+q_3^{(2)}$ (the exponent of $x_{f})$ can never be negative, so that the second sum in $\mathcal{P}_{\theta,3}$ is equal to 0. A similar argument can be applied to show that the other sums are 0.

Let $p_2= {A_{n_1-1;1}-r_1s_{n_1-2;1}}$
and $q_2= {\omega_1 s_{n_1-2;1} - \xi_1  s_{n_1-3;1}   }$ be the exponents of
$\widetilde{x_{1;t_{i+1}}} $
and
$x_{e;t_{i+1}} $
in (\ref{0303eq5a}), respectively.
Applying Theorem~\ref{mainthm12052011} to $\widetilde{x_{1;t_{i+1}}}^ {p_2} x_{e;t_{i+1}}^{q_2}
$ in (\ref{0303eq5a}), we have
$$
\aligned
& \sum_{\begin{array}{c}\scriptstyle\tau_{0;1},\cdots,\tau_{n_1-3;1}\\ \scriptstyle A_{n_1-1;1}-r_1s_{n_1-2;1}\geq 0\end{array}}   \left( \prod_{w=0}^{n_1-3}\gchoose{A_{w+1;1} - r_1s_{w;1}}{\tau_{w;1}}   \right) x_{f;t_{i+1}}^ {r_1s_{n_1-3;1}-A_{n_1-2;1}   }    \\
&\times \sum_{\mathbf{\beta}} x_{1;t_{i+2}}^{r_2|\mathbf{\beta}|_1-A_{n_2-1;2}}x_{e;t_{i+2}}^{A_{n_2;2}-r_2|\mathbf{\beta}|_2} x_{f;t_{i+2}}^{\xi_2|\beta|_2-(\xi_2 r_2 - \omega_2)|\beta|_1}\\
&= \sum_{\begin{array}{c}\scriptstyle\tau_{0;1},\cdots,\tau_{n_1-3;1}\\ \scriptstyle A_{n_1-1;1}-r_1s_{n_1-2;1}\geq 0\end{array}}   \left( \prod_{w=0}^{n_1-3}\gchoose{A_{w+1;1} - r_1s_{w;1}}{\tau_{w;1}}   \right) \\
&\times \sum_{\mathbf{\beta}} x_{1;t_{i+2}}^{r_2|\mathbf{\beta}|_1-A_{n_2-1;2}}x_{e;t_{i+2}}^{A_{n_2;2}-r_2|\mathbf{\beta}|_2}
 x_{f;t_{i+2}}^{\xi_2|\beta|_2-(\xi_2 r_2 - \omega_2)|\beta|_1 + r_1s_{n_1-3;1}-A_{n_1-2;1}},
\endaligned$$
where each second sum is over all $\beta= (S_1=\cup_{i=1}^{p_2+q_2}\, S_{1}^i,S_2=\cup_{i=1}^{p_2+q_2}\, S_{2}^i)$ such that  
$$(S_{1}^i,S_{2}^i) \textup{ is a compatible pair in }\left\{\begin{array}{ll}\mathcal{D}^{c_{n_2-1;2}\times c_{{n_2-2;2}}}  &\textup{ if $1\le i \le q_2$}; \\  
 \mathcal{D}^{c_{n_2;2}\times c_{n_2-1;2}} &\textup{ if $q_2+1\le i\le p_2+q_2$.}
\end{array}\right.$$

The exponent of $x_1$ is positive by definition of $\mathcal{P}_{\theta,3}$. Therefore
$A_{n_2-1;2} <r_2|\mathbf{\beta}|_1$ hence
$A_{n_2-1;2} / r_2|\mathbf{\beta}|_1 <1$ and thus
\[ \frac{r_1A_{n_2-1;2}}{c_{n_2-1}^{[r_2]} \,r_2} =
 \frac{r_1|\mathbf{\beta}|_1}{c_{n_2-1}^{[r_2]} } \frac{\,A_{n_2-1;2}}{r_2|\mathbf{\beta}|_1}
<
\frac{r_1|\mathbf{\beta}|_1}{c_{n_2-1}^{[r_2]}}
\le \xi_2|\beta|_2-(\xi_2 r_2 - \omega_2)|\beta|_1,\]
where the last inequality is proved in \cite[Lemma 4.10]{LS3arXiv} and \cite[Proposition 4.1]{LLZ}.
Using $r_2=\xi_1$ and the definition of $A_{n_2-1;2}$, we get
$$\aligned
& r_1s_{n_1-3;1}-A_{n_1-2;1} + \xi_2|\beta|_2-(\xi_2 r_2 - \omega_2)|\beta|_1 \\
&\geq  r_1s_{n_1-3;1}-A_{n_1-2;1} + \frac{r_1\left( c_{n_2}^{[\xi_1]}(A_{n_1-1;1}-r_1s_{n_1-2;1} ) + c_{n_2-1}^{[\xi_1]}(\omega_1 s_{n_1-2;1} - \xi_1 s_{n_1-3;1})\right)}{c_{n_2-1}^{[\xi_1]} \xi_1}\\
&= -A_{n_1-2;1} + \frac{r_1\left( c_{n_2}^{[\xi_1]}(A_{n_1-1;1}-r_1s_{n_1-2;1} ) + c_{n_2-1}^{[\xi_1]}\omega_1 s_{n_1-2;1}\right)}{c_{n_2-1}^{[\xi_1]} \xi_1}\\
&= -A_{n_1-2;1} + \frac{r_1\left( c_{n_2}^{[\xi_1]}A_{n_1-1;1}-(c_{n_2}^{[\xi_1]}r_1- c_{n_2-1}^{[\xi_1]}\omega_1) s_{n_1-2;1}\right)}{c_{n_2-1}^{[\xi_1]} \xi_1}\\
&\underset{(A_{n_1-1;1}-r_1s_{n_1-2;1}\geq 0)}{\geq}  -A_{n_1-2;1} + \frac{r_1\left( c_{n_2}^{[\xi_1]}A_{n_1-1;1}-(c_{n_2}^{[\xi_1]}r_1- c_{n_2-1}^{[\xi_1]}\omega_1) \frac{A_{n_1-1;1}}{r_1}\right)}{c_{n_2-1}^{[\xi_1]} \xi_1} \\
&= \frac{1}{\xi_1}(\omega_1A_{n_1-1;1} -\xi_1A_{n_1-2;1})\\
&>0,
\endaligned$$where the last inequality follows from (\ref{eq 44.1}). Thus the exponent of $x_{f;t_{i+2}}$ in the expansion of (\ref{0303eq5a}) is positive. The proof for the expansion of (\ref{0303eq5b}) uses a similar argument.
\end{proof}


\section{Example}\label{sect example}
\begin{example}  Let $\Sigma_{t_0}$ be a seed connected to the initial seed $\Sigma_{t_3}$ by the following sequence of mutations.
\[ \xymatrix@C35pt{t_0\ar@{~>}[r]^{2}&t_1' \ar@{~>}[r]^{{1}} &t_1 \ar@{~>}[r]^{3} &t_2' \ar@{~>}[r]^{{1}}& t_2 \ar@{~>}[r]^{2 }&\scriptstyle\bullet\ar@{~>}[r]^{1}&\scriptstyle\bullet\ar@{.}[r]  & \scriptstyle\bullet \ar@{~>}[r]^{{2}}& \scriptstyle\bullet \ar@{~>}[r]^{{1}}& t_3
}
\]
Suppose  that the  quivers corresponding to the first 5 seeds are as follows
$$
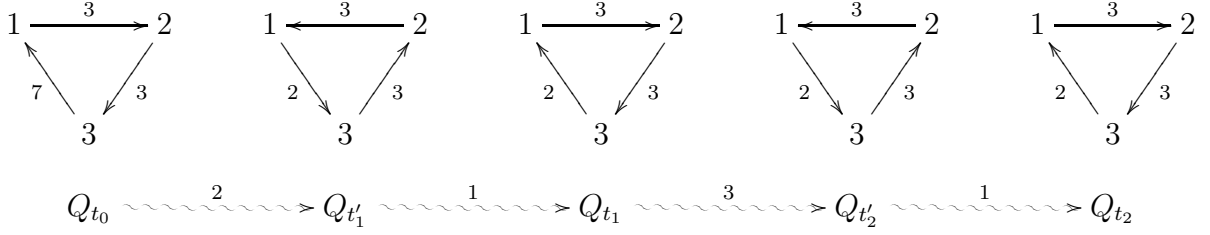
\xymatrix@R10pt@C10pt{1\ar[rr]^3&&2\ar[ddl]^3  &&
1\ar[ddr]_2&&2\ar[ll]_3 &&
1\ar[rr]^3&&2\ar[ddl]^3  &&
1\ar[ddr]_2&&2\ar[ll]_3&&
1\ar[rr]^3&&2\ar[ddl]^3
\\
\\&3\ar[uul]^7 &&& &3\ar[uur]_3 &&&&3\ar[uul]^2&&&&3\ar[uur]_3 &&&&3\ar[uul]^2
\\&
Q_{t_0} \ar@{~>}[rrrr]^2&&&&
Q_{t_1'}\ar@{~>}[rrrr]^1 &&&&
Q_{t_1} \ar@{~>}[rrrr]^3&&&&
Q_{t_2'}\ar@{~>}[rrrr]^1 &&&&
Q_{t_2} }$$

We want to illustrate the proof of  positivity for $x_{2;t_0}$.

First, we compute its expansion in the cluster $\mathbf{x}_{t_1}$ using Theorem~\ref{mainthm12062011}. The mutation sequence from $t_0$ to $t_1$ is in the vertices 1 and 2. Thus we have $r=3$, $n=3$, $p=1$, $q=0$. Moreover,
$$c_1=0,\, c_2=1,\, c_3=3,\qquad
A_1=1,\, A_2=3, \qquad
s_1=\tau_0,$$
and the summation is on $\tau_0=0,1$.
The condition  $s_{n-2}\le A_{n-1}/r$ in the first sum of Theorem~\ref{mainthm12062011} becomes $\tau_0\le 1$ which is always satisfied, so the second sum in the theorem is empty. Finally, the variables in the theorem are $x_2=x_{2;t_1}$, $z_3=x_{3;t_1}$ and $x_3= x_{1;t_1'}=\frac{x_{3;t_1}^{2}+x_{2;t_1}^3}{x_{1;t_1}}$. Thus
 $x_{2;t_0}$ is equal to
\begin{equation}\label{122020110001}
x_{2;t_1}^{-1}\gchooseround{1}{0}\left(\frac{x_{3;t_1}^{2}+x_{2;t_1}^3}{x_{1;t_1}}\right)^{3}=x_{2;t_1}^{-1}\gchooseround{1}{0} x_{1;t_1'}^{3}
\end{equation}
\begin{equation}\label{122020110002}
+x_{2;t_1}^{-1}\gchooseround{1}{1} x_{3;t_1}^3.
\end{equation}

Now we compute the expansion of this expression in the cluster   $\mathbf{x}_{t_2}$ again using Theorem~\ref{mainthm12062011}. We treat the two terms (\ref{122020110001}) and (\ref{122020110002}) separately. For the first term, we need to expand
$ x_{1;t_1'}^{3}$ which lies in the cluster $\mathbf{x}_{t_1'}$. The mutation sequence from $t_1'$ to $t_2$ is in the vertices 1 and 3. Thus we have $r=2$, $n=4$, $p=3$, $q=0$. Moreover,
$$c_1=0,\, c_2=1,\, c_3=2, \, c_4=3,\qquad
A_1=3,\, A_2=6,\, A_3=9, $$
$$s_1=\tau_0,\, s_2=2\tau_0+\tau_1,\,s_3=3\tau_0+2\tau_1+\tau_2.$$
The two binomial coefficients in the first sum are ${3\choose \tau_0}$ and $ {6-2\tau_0\choose \tau_1} $, so their product is zero unless
$$ 0\le\tau_0\le 3,\qquad 0\le \tau_1\le 6-2\tau_0.$$
Finally, the condition $s_{n-2}\le A_{n-1}/r$ in the first sum of Theorem~\ref{mainthm12062011} implies that $\tau_0 < 3$, and that  $\tau_1\le 4 $ if $\tau_0=0$, $\tau_1\le 2 $ if $\tau_0=1$ and $\tau_1=0 $ if $\tau_0=2$.
Therefore, the first sum is over the following pairs $(\tau_0,\tau_1)$
\[ (0,0)\ (0,1)\ (0,2)\ (0,3)\ (0,4)\  (1,0)\ (1,1)\ (1,2)\  (2,0), \]
and the corresponding terms are

\begin{equation}\label{122020110003}
x_{2;t_2}^{-1}\gchooseround{1}{0}x_{3;t_2}^{-6}\gchooseround{3}{0}\gchooseround{6}{0}\exoneprime^{9}
\end{equation}
\begin{equation}\label{122020110004}
+x_{2;t_2}^{-1}\gchooseround{1}{0}x_{3;t_2}^{-6}\gchooseround{3}{0}\gchooseround{6}{1}\exoneprime^{7}x_{2;t_2}^3
\end{equation}
\begin{equation}\label{122020110005}
+x_{2;t_2}^{-1}\gchooseround{1}{0}x_{3;t_2}^{-6}\gchooseround{3}{0}\gchooseround{6}{2}\exoneprime^{5}x_{2;t_2}^6
\end{equation}
\begin{equation}\label{122020110006}
+x_{2;t_2}^{-1}\gchooseround{1}{0}x_{3;t_2}^{-6}\gchooseround{3}{0}\gchooseround{6}{3}\exoneprime^{3}x_{2;t_2}^9
\end{equation}
\begin{equation}\label{122020110007}
+x_{2;t_2}^{-1}\gchooseround{1}{0}x_{3;t_2}^{-6}\gchooseround{3}{0}\gchooseround{6}{4}\exoneprime^{1}x_{2;t_2}^{12}
\end{equation}
\begin{equation}\label{122020110008}
+x_{2;t_2}^{-1}\gchooseround{1}{0}x_{3;t_2}^{-4}\gchooseround{3}{1}\gchooseround{4}{0}\exoneprime^{5}x_{2;t_2}^{3}
\end{equation}
\begin{equation}\label{122020110009}
+x_{2;t_2}^{-1}\gchooseround{1}{0}x_{3;t_2}^{-4}\gchooseround{3}{1}\gchooseround{4}{1}\exoneprime^{3}x_{2;t_2}^{6}
\end{equation}
\begin{equation}\label{122020110010}
+x_{2;t_2}^{-1}\gchooseround{1}{0}x_{3;t_2}^{-4}\gchooseround{3}{1}\gchooseround{4}{2}\exoneprime^{1}x_{2;t_2}^{9}
\end{equation}
\begin{equation}\label{122020110011}
+x_{2;t_2}^{-1}\gchooseround{1}{0}x_{3;t_2}^{-2}\gchooseround{3}{2}\gchooseround{2}{0}\exoneprime^{1}x_{2;t_2}^{6}
\end{equation}

The second sum in Theorem~\ref{mainthm12062011} is over all compatible pairs $(S_1,S_2)$ in $\mathcal{D}^{9 \times 6}$ such that $|S_2|>9/2$.  
The condition $|S_2|>9/2$ implies that $S_2$ must be equal to $\mathcal{D}_2$ or $\mathcal{D}_2\setminus\{$any single vertical edge$\}$. If $S_2=\mathcal{D}_2$ then $S_1$ must be the empty set. If $S_2=\mathcal{D}_2\setminus\{v_{2i-1}\}$ for $i=1,2,3$, then $S_1=\{u_{3i-2}\}$ or $\emptyset$. If $S_2=\mathcal{D}_2\setminus\{v_{2i}\}$ for $i=1,2,3$, then $S_1=\{u_{3i}\}$ or $\emptyset$. Therefore the second sum  in Theorem~\ref{mainthm12062011} is equal to

\begin{equation}\label{122020110012}
+x_{2;t_2}^{-1}\gchooseround{1}{0}6x_{3;t_2}^{-6}{x_{1;t_2}}x_{2;t_2}^{12}
\end{equation}
\begin{equation}\label{122020110013}
+x_{2;t_2}^{-1}\gchooseround{1}{0}x_{3;t_2}^{-6}x_{1;t_2}^3 x_{2;t_2}^{9}
\end{equation}
\begin{equation}\label{122020110014}
+x_{2;t_2}^{-1}\gchooseround{1}{0}6x_{3;t_2}^{-4}{x_{1;t_2}} x_{2;t_2}^{9}.
\end{equation}

This shows that (\ref{122020110001}) is equal to the sum of all terms (\ref{122020110003})-(\ref{122020110014})

Applying Theorem~\ref{mainthm12062011} to the expression (\ref{122020110002}) and using a similar analysis, we see that  (\ref{122020110002}) is equal to
\begin{equation}\label{122020110015}
x_{2;t_2}^{-1}\gchooseround{1}{1}  x_{3;t_2}^{-3}\gchooseround{3}{0}\exoneprime^{6}
\end{equation}
\begin{equation}\label{122020110016}
+x_{2;t_2}^{-1}\gchooseround{1}{1}  x_{3;t_2}^{-3}\gchooseround{3}{1}\exoneprime^{4}x_{2;t_2}^{3}
\end{equation}
\begin{equation}\label{122020110017}
+x_{2;t_2}^{-1}\gchooseround{1}{1}  x_{3;t_2}^{-3}\gchooseround{3}{2}\exoneprime^{2}x_{2;t_2}^{6}
\end{equation}
\begin{equation}\label{122020110018}
+x_{2;t_2}^{-1}\gchooseround{1}{1}x_{3;t_2}^{-3}\gchooseround{3}{3}  x_{2;t_2}^{9}.
\end{equation}

So $x_{2;t_0}$ is equal to the sum of all terms (\ref{122020110003})-(\ref{122020110018}).
Observe that the powers of the variables $\exoneprime, x_{1;t_2}$ in all terms are positive
and that the powers of the variable $x_{2;t_2}$ are positive in all terms except for
(\ref{122020110003}) and (\ref{122020110015}).

On the other hand,
$$(\ref{122020110003})+(\ref{122020110015}) = x_{3;t_2}^{-6}\exoneprime^{6}\left(\frac{\exoneprime ^{3} +  x_{3;t_2}^{3}}{x_{2;t_2}} \right) =  x_{3;t_2}^{-6}\exoneprime^{6} x_{2;t_2''}$$
where $\mathbf{x}_{t_2''}=\mu_2(\mathbf{x}_{t_2'})$ denotes the cluster  obtained from $\mathbf{x}_{t_2'}$ by mutation in $2$.

Thus we obtain an expression for  $x_{2;t_0}$ as a Laurent polynomial in the variables
$x_{2;t_2''}$, $\exoneprime$, $x_{1;t_2}$, $x_{2;t_2}$, $x_{3;t_2}$ with nonnegative coefficients and in which only the variable $x_{3;t_2}$ appears with negative powers. Note that $x_{2;t_2''}=\widetilde{\widetilde{x_{2;t_2}}}$ and $\exoneprime=\widetilde{x_{1;t_2}}$, thus the sum (\ref{122020110003})+(\ref{122020110015}) is of the form of the first sum in Theorem~\ref{mainthm11102011_00}, the sum of (\ref{122020110004})-(\ref{122020110011}) and  (\ref{122020110016})-(\ref{122020110018})   is of the form of the second sum, and the sum of (\ref{122020110012})-(\ref{122020110014})  is of the form of the third sum in Theorem~\ref{mainthm11102011_00}.

Since the mutation sequence linking the seeds $\Sigma_{t_2''},\Sigma_{t_2'}$ and $\Sigma_{t_2}$ to the seed $\Sigma_{t_3}$ consists of mutations in the vertices 1 and 2 only, we see that $x_{3;t_2}=x_{3;t_3}$ and replacing the other  variables with their expansions in the seed $\Sigma_{t_3}$ (which have nonnegative coefficients by the rank 2 case) produces again a Laurent polynomial with nonnegative coefficients in the cluster $\mathbf{x}_{t_3}$.
\end{example}


\begin{thebibliography}{99}

\bibitem{ADSS} I. Assem, G. Dupont, R. Schiffler and D. Smith,  Friezes, strings and cluster variables, {\em Glasgow Math. J.\/} {\bf 54}, 1 (2012) 27--60.
\bibitem{ARS} I. Assem, C. Reutenauer and D. Smith, Friezes.
{\em Adv. Math.\/} {\bf 225} (2010), no. 6, 3134--3165.

\bibitem{BBH} A. Beineke, T. Br\"ustle and L. Hille, Cluster-cyclic quivers with three vertices and the Markov equation, {\em Algebr. Represent. Theory\/} {\bf 14} (2011), no. 1, 97--112.

\bibitem{BFZ}  A. Berenstein, S. Fomin and  A. Zelevinsky,  Cluster algebras III: Upper bounds and double Bruhat cells. 	{\em Duke Math. J.\/}   {\bf 126} (2005),  No. 1, 1--52.	

\bibitem{BZ}
A. Berenstein, and A. Zelevinsky, Quantum cluster algebras. {\em Adv. Math.\/} \textbf{195} (2005), no. 2, 405--455. 

\bibitem{CC}
P. Caldero and F. Chapoton, Cluster algebras as Hall algebras of quiver representations, {\em Comment. Math. Helv.\/} \textbf{81} (2006), No. 3, 595--616.

\bibitem{CK1}
P. Caldero and B. Keller, From triangulated categories to cluster algebras. {\em Invent. Math. \/}\textbf{172} (2008), No. 1, 169--211.



\bibitem{DWZ}
H. Derksen, J. Weyman and A. Zelevinsky, Quivers with potentials and their representations II: applications to cluster algebras. {\em J. Amer. Math. Soc.\/} \textbf{23} (2010), No. 3, 749--790. 

\bibitem{Ef}
A. I. Efimov, Quantum cluster variables via vanishing cycles, arXiv:1112.3601.
	
\bibitem{FeShTu} A. Felikson, M. Shapiro and P. Tumarkin, Cluster algebras and triangulated orbifolds, {\em Adv. Math.\/} {\bf 231} (2012), no. 5, 2953--3002.



	
\bibitem{FG} V. Fock and A. Goncharov, Moduli spaces of local
  systems and higher Teichm\"uller theory.  {\em Publ. Math. Inst. Hautes
  \'Etudes Sci.}  No. {\bf 103}  (2006), 1--211.

\bibitem{FST} S. Fomin, M. Shapiro, and D. Thurston, Cluster algebras and triangulated surfaces. Part I: Cluster complexes, \emph{Acta Math.} {\bf 201} (2008), 83-146. 


\bibitem{FZ}
S. Fomin and  A. Zelevinsky, Cluster algebras I: Foundations, {\em J. Amer. Math. Soc.\/}  \textbf{15} (2002) No. 2  497--529. 

\bibitem{FZ4}
S. Fomin and  A. Zelevinsky, Cluster algebras IV: Coefficients, {\em Comp. Math.\/} \textbf{143} (2007), 112--164.


\bibitem{FK} C. Fu and B. Keller, On cluster algebras with coefficients and 2-Calabi-Yau categories.
{\em Trans. Amer. Math. Soc.\/} {\bf 362} (2010), no. 2, 859--895.



\bibitem{HL}
D. Hernandez and B. Leclerc, Cluster algebras and quantum affine algebras,
{\em Duke Math. J.\/} \textbf{154} (2010), No. 2, 265--341.



\bibitem{KQ}Y. Kimura and F. Qin,
Graded quiver varieties, quantum cluster algebras and dual canonical basis, arXiv:1205.2066.

\bibitem{L} K. Lee, On cluster variables of rank two	acyclic cluster algebras, {\em Ann. Comb.\/} {\bf 16} (2012), no. 2, 305--317.
\bibitem{LS}
K. Lee and R. Schiffler, A Combinatorial Formula for Rank 2 Cluster Variables, {\em J. Alg. Comb.\/}  {\bf 37} (1), 67--85. 

\bibitem{LS3arXiv}
K. Lee and R. Schiffler, Positivity for clyster algebras of rank 3 (version 1), arXiv:1205.5466v1.

\bibitem{LLZ} K. Lee, L. Li and A. Zelevinsky, Greedy elements in rank 2 cluster algebras, {\em Selecta Math. (N.S.)}, to appear.

\bibitem{MSW}
G. Musiker, R. Schiffler and L. Williams, Positivity for cluster algebras from surfaces, {\em Adv. Math.\/} {\bf 227} (2011) 2241--2308.

\bibitem{N}
H. Nakajima, Quiver varieties and cluster algebras, Kyoto J. Math. \textbf{51} (2011), No. 1, 71--126. (Memorial Issue for the Late Professor Masayoshi Nagata).

\bibitem{Q}
F. Qin, Quantum cluster variables via Serre polynomials, {\em J. Reine Angew. Math.\/} {\bf 668} (2012), 149--190.

\bibitem{R}  D. Rupel, Proof of the Kontsevich non-commutative cluster positivity conjecture,  {\em C. R. Math. Acad. Sci. Paris} {\bf 350} (2012), no. 21-22, 929--932.
\bibitem{S3} R. Schiffler, On cluster algebras arising from unpunctured surfaces. II. {\em Adv. Math.\/} {\bf 223} (2010), no. 6, 1885--1923.
\bibitem{S2} R. Schiffler, A cluster expansion formula ($A_n$ case). {\em Electron. J. Combin.\/} {\bf 15} (2008), no. 1, Research paper 64, 9 pp.
\bibitem{ST}  R. Schiffler  and H. Thomas, On cluster algebras arising from unpunctured surfaces. {\em Int. Math. Res. Not.\/} IMRN 2009, no. {\bf 17}, 3160--3189.

\end{thebibliography}
\end{document}